\documentclass{amsart}
\usepackage{tikz}
\usetikzlibrary{arrows,matrix}
\usepackage[all]{xy}

\begin{document}
\title[Purely arithmetic PDE'S over a $p$-adic field I]
{Purely arithmetic PDE'S over a  $p$-adic field I:\\
$\delta$-characters and $\delta$-modular forms}
\author{Alexandru Buium}
\address{Department of Mathematics and Statistics,
University of New Mexico, Albuquerque, NM 87131, USA}
\email{buium@math.unm.edu} 

\author{Lance Edward Miller}
\address{Department of Mathematical Sciences,  309 SCEN,
University of Arkansas, 
Fayetteville, AR 72701}
\email{lem016@uark.edu}

\def \cH{\mathcal H}
\def \cB{\mathcal B}
\def \d{\delta}
\def \ra{\rightarrow}
\def \bZ{{\mathbb Z}}
\def \cO{{\mathcal O}}
\newcommand{\Hom}{\operatorname{Hom}}

\def\uy{\underline{y}}
\def\uT{\underline{T}}

\newcommand{\OCp}{ {\mathbb C}_p^{\circ}  }

\newtheorem{THM}{{\!}}[section]
\newtheorem{THMX}{{\!}}
\renewcommand{\theTHMX}{}
\newtheorem{theorem}{Theorem}[section]
\newtheorem{maintheorem}{Theorem}
\newtheorem{corollary}[theorem]{Corollary}
\newtheorem{lemma}[theorem]{Lemma}
\newtheorem{proposition}[theorem]{Proposition}
\theoremstyle{definition}
\newtheorem{definition}[theorem]{Definition}
\theoremstyle{remark}
\newtheorem{remark}[theorem]{Remark}
\newtheorem{example}[theorem]{\bf Example}
\numberwithin{equation}{section}
\subjclass[2000]{11 F 32, 11 F 85, 11 G 07}

\begin{abstract}
A formalism of  arithmetic partial differential equations (PDEs) is being developed in which one considers several arithmetic differentiations at one fixed prime. In this theory solutions can be defined in algebraically closed $p$-adic fields. As an application we show that for at least two arithmetic directions every elliptic curve possesses a non-zero  arithmetic PDE Manin map of order $1$; such  maps do not exist in the arithmetic ODE case.   Similarly we construct and study ``genuinely PDE" differential modular forms. As  further applications  we derive a Theorem of the Kernel and  a Reciprocity Theorem  for arithmetic PDE Manin maps and also a finiteness Diophantine result for modular parameterizations.  We also prove structure results for   the spaces of  ``PDE differential modular forms defined on the ordinary locus." We also produce a system of differential equations satisfied by our PDE modular forms based on Serre and Euler operators. \end{abstract}

\maketitle

\tableofcontents

\section{Introduction}

Arithmetic differential equations were introduced in \cite{Bu05} and were successfully  applied to a series of problems in  Diophantine geometry \cite{Bu96, Bu97, Bu00, Bu05, BP09}. For example letting
$R$ be the completed valuation ring of the maximal unramified extension of $\mathbb Q_p$
 a special case of the main results of \cite{BP09} states for 
the modular curve $X := X_1(N)$ over $R$ (with $N\geq 4$ coprime to $p$) and an elliptic curve $E$ over $R$ if $\Theta \colon X \to E$ is any surjective morphism then the intersection the $\textup{CL}$-locus of $X$ with the inverse along $\Theta$ of any finite rank subgroup of $E(R)$ must be finite. This is morally done by producing interesting homomorphisms $\psi \colon E(R) \to R$ and  showing that $\textup{CL} \cap \Theta^{-1}( \ker \psi)$ is finite where $\textup{CL}$ is the CL-locus of canonical lift points (which are the analogues, in the local setting, of CM points). 

The functions $\psi$ arise from arithmetic versions of Manin maps (``$\delta$-characters") of order $2$ while  the theory of   ``$\delta$-modular forms" provides functions that vanish on CL. These objects are arithmetic analogues of ODEs and
 are built using the canonical  Frobenius lift on $R$. In particular, this theory  only applies to  unramified  settings and concerns a single  Frobenius lift.  In this paper, we describe a significant enhancement  of the foundational theory of arithmetic differential equations. On the one hand, for an arbitrary  root $\pi$ of an Eisenstein polynomial with coefficients in $R$, we consider the ramified setting $R_\pi := R[\pi]$. On the other hand, 
 we will consider several Frobenius lifts leading to an arithmetic  PDE theory. We will present a series of applications.  As an example 
 for $E$ an elliptic curve over $R_\pi$, we  produce a genuinely new homomorphism $\psi \colon E(R_\pi) \to R_\pi$ which is an order $1$   PDE analogue of a Manin map, for which we can prove the following finiteness theorem replacing $\textup{CL}$ with the locus of  quasi-canonical lifts considered in \cite{Gro86}. For an open set $X\subset X_1(X)_{R_{\pi}}$, denote by $\textup{QCL}(X(R_{\pi}))$ the set of {\bf quasi-canonical lift points} in $X(R_{\pi})$ i.e., points corresponding to ordinary elliptic curves whose Serre-Tate parameter is  a root of unity. 
Here is a sample of our results:

\begin{maintheorem}\label{thm:mainapp} Consider a surjective morphism of $R_{\pi}$-schemes $\Theta:X_1(N)\rightarrow E$  and denote by $\Theta_{R_{\pi}}:X_1(N)(R_{\pi})\rightarrow E(R_{\pi})$ the induced map. There exists an open set $X\subset X_1(N)$ with non-empty reduction mod $\pi$ such that for all except finitely many cosets $C$ of $\textup{Ker}(\psi)$ in $E(R_{\pi})$ the  set  $\textup{QCL}(X(R_{\pi}))\cap \Theta_{R_{\pi}}^{-1}(C)$ is finite.
\end{maintheorem}

\subsection{Background}
The present paper should be viewed as the start of a series of papers on the subject and is essentially self contained. However,  for convenience, we explain its background  in what follows.

A theory of arithmetic ordinary differential equations (ODEs) was initiated in \cite{Bu95}. In this theory derivation operators  acting on functions  were replaced by Fermat quotient operators $\delta_p$
(referred to as {\it $p$-derivations}) 
acting on $p$-adic numbers; one interpreted $\delta_p$ as an ``arithmetic  differentiation" with respect to the ``arithmetic direction" $p$.
This theory had a series  of Diophantine applications; cf. \cite{Bu95, Bu96, Bu97, BP09}.
In particular in \cite{Bu95} arithmetic analogues of the classical Manin maps \cite{Man63} were constructed and in \cite{Bu95, Bu97}  arithmetic analogues of Manin's Theorem of the Kernel were proved.
We recall that, for a function field $F$,  the classical Manin maps are $F$-valued non-linear differential operators of order $2$ defined  on the set of $F$-rational points of an abelian $F$-variety. Similarly  the arithmetic Manin maps in \cite{Bu95} had order $2$ and were defined on the set of points of an abelian variety  over a $p$-adic field. 
Other basic ODEs were shown  to have arithmetic analogues. This is the case for Schwarzian-type ODEs satisfied by classical modular forms, cf. \cite{Bu00, Bar03} and \cite[Ch. 8]{Bu05}  where a theory of differential modular forms was developed. 

\

In the framework of \cite{Bu95, Bu05} the only solutions of arithmetic ODEs that were defined
were ``unramified solutions" i.e., solutions (with coordinates) in the completion $R$ of the maximum unramified extension of $\mathbb Z_p$. Subsequently the $\delta$-overconvergence machinery in \cite{BuSa11, BM20} allowed one to define ``ramified solutions" to the main arithmetic ODEs of the theory, i.e., solutions  in the ring of integers $R^{\text{alg}}$ of the algebraic closure $K^{\text{alg}}$ of $K:=R[1/p]$ and sometimes even in the ring of integers $\mathbb C_p^{\circ}$ of the complex $p$-adic field $\mathbb C_p$.

\

A  theory of  arithmetic PDEs with two ``directions" one of which was arithmetic and the other geometric  was  then developed
in \cite{BuSi10a, BuSi10b}. This theory combined an arithmetic differentiation $\delta_p$ in the ``arithmetic direction $p$" with 
usual differentiation $\delta_q:=\frac{d}{dq}$ with respect to a ``geometric direction" defined by a variable $q$. The
two operators $\delta_p$ and $\delta_q$ were viewed as acting on the power series ring $R[[q]]$ and solutions were well defined (and extensively studied)  in this ring. A somewhat surprising outcome of \cite{BuSi09a} was that, in this arithmetic PDE context, analogues of Manin maps exist that have order $1$ (rather than $2$) and interesting interactions were found between 
the order $2$ ODE Manin maps (both arithmetic and geometric) and the newly discovered 
order $1$ PDE Manin maps.  In some sense the existence of order $1$ Manin maps was an effect of the arithmetic direction $p$ and the geometric direction $q$ ``conspiring" to create  lower order Manin maps. In \cite{BuSi10b} a theory of differential modular forms in this setting was developed.
This version of the theory was an ``unramified theory" in the sense that  solutions were defined in $R[[q]]$ and did not make sense in $R^{\text{alg}}[[q]]$.

\

It is reasonable instead to hope for a ``purely arithmetic" PDE theory i.e., a PDE theory in which all the directions are ``arithmetic." Along these lines a theory of  arithmetic PDEs with $n\geq 2$ arithmetic directions was developed in \cite{BuSi09b, BB11} in which  $n$ arithmetic differentiation operators were attached to  $n$ distinct prime integers. In this version of the theory, the solutions of arithmetic PDEs were  only defined in  number fields that were unramified at the primes in question. Arithmetic Manin maps were constructed in this context using a technique introduced in \cite{BuSi09b} called  {\it analytic continuation between primes}. The order of the arithmetic Manin maps in this setting was $2n\geq 4$; hence, in some sense, the  primes involved acted as if they obstructed each other in the process of creating Manin maps.

\ 

There is a basic version of the theory that is missing from the above series of approaches, namely a purely arithmetic PDE theory where the  {\it several arithmetic differentiations} are all attached to {\it one} fixed prime $p$. For such a theory to be relevant one needs to make sense of  ``ramified solutions," i.e. of solutions in $R^{\text{alg}}$. It is the aim of the present paper to systematically develop such a theory and provide new applications. We have already made clear in the introduction that tangible Diophantine applications will come out of this enhancement. However, there is even more. Additionally, certain ODE versions of the PDEs appearing in classical Riemannian geometry related to Chern and Levi-Civita connections have been developed (cf. \cite{Bu17}). The fundamental framework we describe here will have ramifications in that theory as well, and will be explored in detail in  subsequent papers in this series where we will also show how the above versions of the arithmetic PDE theory can be unified.

\subsection{Framework of this paper}

Our starting point is the observation that in the ramified $p$-adic world
one should envision not one but many arithmetic directions  reflecting the fact that the absolute Galois group of $\mathbb Q_p$ does not have one but  several (in fact one can take $4$) topological generators (cf. \cite{JW82}). These topological generators can be chosen to be    {\bf Frobenius automorphisms of} $\mathbb Q_p^{\textup{alg}}$; cf Definition \ref{frobautt}. 
One can then develop  the theory
starting from an arbitrary finite collection  $\phi_1,\ldots,\phi_n$ of   Frobenius automorphisms
of $K^{\textup{alg}}$. Remarkably this approach, combined with the $\delta$-overconvergence technique in \cite{BuSa11, BM20}, allows one to define solutions to our equations in $R^{\text{alg}}$. As an application we will again construct arithmetic Manin maps which (as in \cite{BuSi10a} but unlike in \cite{BuSi09b}) have order $1$;
so the various arithmetic differentiation operators at $p$ conspire, again,  to create lower order arithmetic Manin maps. On the other hand for $n=2$  one can introduce  a remarkable order $2$ arithmetic PDE Manin map that can be viewed as the ``Laplacian" of our context. This is very different from the order $4$ arithmetic Laplacian in the context of \cite{BuSi09b}. An arithmetic PDE version of the theory of differential modular forms in \cite{Bu00, Bar03, Bu05} will also  be developed in this paper and a series of new ``purely PDE" phenomena will be put forward. We summarize our discussion above in the following table. Here $N_{\textup{pr}}$ below is the number of primes involved, $N_{\textup{ari}}$ is the number of arithmetic directions, and $N_{\textup{geo}}$ is the number of geometric directions. 

\bigskip

\begin{center}
\begin{tabular}{||c|c|c|c|c||} \hline \hline
reference  & $N_{\textup{pr}}$ & $N_{\textup{ari}}$  & $N_{\textup{geo}}$ & ramified solutions defined\\ 
\hline
\cite{Bu95} & 1 & 1 & 0 & NO \\
\hline
\cite{BM20} & 1 & 1 & 0 & YES\\
\hline
\cite{BuSi10a} & 1 & 1 & 1 & NO\\
\hline
\cite{BuSi09b} & n & n & 0 & NO\\
\hline
This paper & 1 & n & 0 & YES\\
\hline
\hline
\end{tabular}
\end{center}

\bigskip

\subsection{Terminology}
In this paper, unless otherwise stated,  all rings will be commutative with identity.  A morphism of Noetherian rings will be called {\bf smooth} if it is of finite type
and is $0$-smooth in the sense of \cite[p. 193]{Mat92}.
Throughout this paper we fix an odd prime $p\in \mathbb Z$ and for any ring $S$ and any Noetherian scheme $X$ we denote by   $\widehat{S}$ and $\widehat{X}$ the respective $p$-adic completions. The superscript ``alg"  will mean algebraic closure. The superscript ``ur"  will mean maximum unramified extension. By an {\bf elliptic curve} over a ring we mean 
an abelian scheme  of relative dimension one. 
There are two contexts in which the word ``ordinary" appears in this paper: one  as in ``ordinary versus partial differential equation"; and the other 
 as in ``ordinary versus supersingular elliptic curve."
To avoid confusion we will always say ``ODE" instead of ``ordinary" in the first situation.
Also we will often use ``ODE" and ``PDE" as adjectives as in ``ODE arithmetic Manin maps," ``PDE differential modular forms," etc.

\subsection{Main results}
In what follows we explain some of our main results including the previously mentioned theorem in more context.  For the precise definitions of our concepts we refer to the body of the paper. For simplicity we assume, for the rest of this introduction,  that the number of Frobenius automorphisms is $n=2$. Some of results below have variants that will be proved for arbitrary $n$.

\

 Let $\pi$  be a  root 
in ${\mathbb Q}_p^{\text{alg}}$
of an Eisenstein polynomial  coefficients in $\bZ_p^{\text{ur}}$ having the property that the extension
${\mathbb Q}_p(\pi)/{\mathbb Q}_p$ is Galois. With $R=\widehat{\mathbb Z_p^{\textup{ur}}}$ and $K=R[1/p]$ as above let $R_{\pi}:=R[\pi]$ and $K_{\pi}:=K(\pi)$. 
Recall from \cite{Bu95} that  a {\bf $\pi$-derivation} on a flat $R_{\pi}$-algebra $A$ is a map $\delta_{\pi}:A\rightarrow A$ such that the map $\phi:A\rightarrow A$ defined by $\phi(x)=x^p+\pi \delta_{\pi}(x)$ is a ring homomorphism which is then referred to as a {\bf $\pi$-Frobenius lift}. 
We fix 
  a pair    $\Phi=(\phi_1,\phi_2)$ of Frobenius automorphisms of $K^{\text{alg}}$; the automorphisms $\phi_1,\phi_2$ induce {\bf $\pi$-derivations}   on $R_{\pi}$. 
  For any smooth scheme $X$ over $R_{\pi}$ we will define a sequence of $p$-adic formal schemes   $J^r_{\pi,\Phi}(X)$ 
  called the {\bf partial $\pi$-jet spaces} of $X$. The ring of functions on $J^r_{\pi,\Phi}(X)$ will be referred to as 
  the ring of (purely) {\bf arithmetic PDEs} on $X$ order $\leq r$ (cf. Definition \ref{APDE}). We will then define  its subring of {\bf totally $\delta$-overconvergent} elements (cf. Definition \ref{overconvergence}). There is a natural action of $\phi_1,\phi_2$ on the colimit as $r\rightarrow \infty$ of these rings. 
  Every arithmetic PDE $f$ on $X$ defines a map of sets $f_{R_{\pi}}:X(R_{\pi})\rightarrow R_{\pi}$. If $f$ is 
  totally $\delta$-overconvergent then the map $f_{R_{\pi}}$ extends to a  map of sets $f^{\textup{alg}}:=f^{\textup{alg}}_{R_{\pi}}:X(R^{\textup{alg}})\rightarrow K^{\textup{alg}}$ and  the preimage of $0$ under this map is  the {\bf set of solutions} in $R^{\textup{alg}}$ of the arithmetic PDE $f$.

Let $E$ be an elliptic curve over $R_{\pi}$. 
We define a {\bf partial $\delta_{\pi}$-character} of order $\leq r$ on $E$ to be an arithmetic PDE of order $\leq r$ which, viewed as a morphism
$J^r_{\pi,\Phi}(E)\rightarrow \widehat{\mathbb G_a}$, is a group homomorphism; cf. Definition \ref{deltacharr}.
Extending  terminology from \cite{Bu95} ``partial $\delta_{\pi}$-characters" is  the name for our ``arithmetic Manin maps" in our PDE setting. 
To each $E$ 
and every basis $\omega$ for the $1$-forms on $E$ we will attach 
two families of elements in $R_{\pi}$  called   (primary, respectively secondary) {\bf arithmetic Kodaira-Spencer classes}; cf. Definitions \ref{primaryy} and \ref{secondaryy}.
Finally
to each partial $\delta_{\pi}$-character of $E$ 
we will  attach a {\bf Picard-Fuchs symbol} which is a formal $K_{\pi}$-linear combination of non-commutative monomials in $\phi_1,\phi_2$; cf Definition \ref{PFsymbol}. 
The  arithmetic Kodaira-Spencer classes   appear then as coefficients of the symbols of certain distinguished partial $\delta_{\pi}$-characters.
Among the primary Kodaira-Spencer
classes a special role will be played by elements denoted by 
$f_1,f_2\in R_{\pi}$.  
One of our  main results  will be the following (cf. Corollaries \ref{C1}, \ref {mystery} and Proposition \ref{injonchar}). This can be viewed as a simultaneous generalization of the main results in \cite{Bu95} and \cite{BM20}.

\begin{theorem}\label{goforit}
Let $E$ be an elliptic curve over $R_{\pi}$.

1) If $f_1\neq 0$ or $f_2\neq 0$ then the $R_{\pi}$-module of partial $\delta_{\pi}$-characters of order $\leq r$ has rank equal to $2^{r+1}-3$.

2) If $f_1=f_2=0$ then the $R_{\pi}$-module of partial $\delta_{\pi}$-characters of order $\leq r$ has rank equal to $2^{r+1}-2$.

3) Every partial $\delta_{\pi}$-character $\psi$ is totally $\delta$-overconvergent and the induced  group homomorphism
$\psi^{\textup{alg}}:E(K^{\textup{alg}})\rightarrow K^{\textup{alg}}$
 can be extended to a continuous homomorphism
$\psi^{\mathbb C_p}:E(\mathbb C_p)\rightarrow \mathbb C_p$.
If $\phi_1,\phi_2$ are monomially independent 
then $\psi$ is uniquely determined by $\psi^{\textup{alg}}$.
\end{theorem}

The homomorphisms $\psi^{\textup{alg}}$ are not given by algebraic (or even by analytic) functions 
in the coordinates but rather by analytic (in fact rigid analytic)  functions in the coordinates and their various ``$\delta_{\pi'}$-derivatives" for various $\pi'$'s dividing $\pi$. The recipe for defining these maps involves the notion of   total $\delta$-overconvergence which is analogous to the one in \cite{BM20} and  will be explained in the body of the paper. Note that since $E$ is projective over $R_{\pi}$ we have $E(K^{\text{alg}})=E(R^{\text{alg}})$; however, it is an important feature of the theory that the images of the maps $\psi^{\textup{alg}}$ are not contained in $R^{\text{alg}}$.

For the case of order $\leq 2$ we have more precise results. Consider the  subset $\mathbb M_2^{2,+}:=\{1,2,11,22,12,21\}$ of non-empty words of length $\leq 2$ in
 the free monoid generated by the set $\{1,2\}$.
 There are $6$ primary Kodaira-Spencer classes of order $\leq 2$,
 \begin{equation}\label{prim}
 f_{\mu},\ \ \mu\in \mathbb M_2^{2,+}.
 \end{equation}
 The classes $f_1,f_2,f_{11},f_{22}$
  come from the ODE theory \cite{Bu05}. On the other hand the classes $f_{12}, f_{21}$  are ``genuinely PDE" (not ``reducible to ODEs"). The secondary Kodaira-Spencer classes will be denoted by
 \begin{equation}
 \label{sec}
 f_{\mu,\nu},\ \ \mu,\nu\in \mathbb M_2^{2,+},\ \ \mu\neq \nu.
 \end{equation}
 They satisfy $f_{\mu,\nu}+f_{\nu,\mu}=0$.
  The classes
  $f_{11,1}, f_{22,2}$  come from the ODE theory while the others classes $f_{\mu,\nu}$
  are, again, ``genuinely PDE".   
  By the theory in \cite{Bu05} the secondary classes $f_{ii,i}$, $i\in \{1,2\}$,  are known  to be expressible in terms of the primary ones $f_i$ as 
  $f_{ii,i}=p\phi_if_i$;
   cf. Remark \ref{lala}. Note that if $E$ has ordinary reduction then $f_i=0$ for some $i$ if and only if  $f_{\mu}=f_{\mu,\nu}=0$ for all $\mu$ and $\nu$, if and only if ``the" Serre-Tate parameter of $E$ is a root of unity; cf. Proposition \ref{chara}.
   Finally note  (cf. Theorem   \ref{fat2}) that there exist  $\pi,\phi_1,\phi_2$ and  a pair $(E,\omega)$ over $R_{\pi}$ such that  $E$ has ordinary reduction and all classes (\ref{prim}) and (\ref{sec}) 
attached to $(E,\omega)$
are non-zero. 
   
  In case $\pi=p$ we can be more specific. Indeed   for all $(E,\omega)$ over $R$ we have $f_1=f_2$, 
$f_{11}=f_{22}$ and $f_{1,2}=f_{11,22}=0$. 
(In \cite{Bar03} and \cite{Bu05} $f_i$ was denoted by $f^1$ and $f_{ii}$ was denoted by $f^2$.)  In this case we have that $f_i=0$ if and only if $E$ has ordinary reduction and is a canonical lift of its reduction; cf. Remark \ref{fifj}. Also, if $E$ comes from a curve $E_{\mathbb Z_p}$ over $\mathbb Z_p$ and has ordinary reduction but is not a canonical lift then $f_{ii}=a_pf_i\neq 0$ where $a_p\in \mathbb Z$ is the trace of Frobenius on the reduction mod $p$ of $E_{\mathbb Z_p}$; cf. Remark \ref{irir}.

Going back to the general situation when $\pi$ is arbitrary 
we let $N(\pi)$ be the smallest integer $N\in \mathbb Z$ such that for all integers $n\geq 1$
we have $\pi^n/n\in p^{-N}\mathbb Z_p$; in particular $N(p)=-1$.
If $\mu=i\in \{1,2\}$ we set $\phi_{\mu}=\phi_i$ while for $\mu=ij$ with $i,j\in \{1,2\}$ we set $\phi_{\mu}=\phi_i\phi_j$. Also we set $\tilde{f}_{\mu}=p^{N(\pi)+1}f_{\mu}$. We will prove (see Corollaries \ref{C2} and \ref{C3}) the following summary result, which also may be viewed as a generalization of the main results of \cite{Bu95}. 

\bigskip

\begin{theorem}\label{goforitt}
Assume in Theorem \ref{goforit}
 that $f_1f_2\neq 0$. The following hold:
 
 1) For all $\mu,\nu\in \mathbb M_2^{2,+}$ there is a unique $\delta_{\pi}$-character $\psi_{\mu,\nu}$ with Picard-Fuchs symbol $\tilde{f}_{\nu}\phi_{\mu}-\tilde{f}_{\mu}\phi_{\nu}+f_{\mu,\nu}$.
 
2)  A basis modulo torsion of the $R_{\pi}$-module of partial $\delta_{\pi}$-characters of order $\leq 1$
consists of $\psi_{1,2}$.

3) A basis  modulo torsion of the $R_{\pi}$-module of partial $\delta_{\pi}$-characters of order $\leq 2$
 consists of the elements
$\psi_{1,2},\ \phi_1\psi_{1,2},\ \phi_2 \psi_{1,2},\ \psi_{11,1},\ \psi_{22,2}.$
 \end{theorem}

Here and later by a {\bf basis modulo torsion} of an $R_{\pi}$-module $M$ we mean a family of elements in $M$ inducing a basis of the $K_{\pi}$-linear space $M\otimes_{R_{\pi}}K_{\pi}$. 

One is tempted to view $\psi_{11,22}$  as the ``Laplacian" equation in our context while $\psi_{12,21}$ reflects, in some sense,  the non-commutation of $\phi_1$ and $\phi_2$ and can be viewed as a ``Poisson bracket operator." In case $f_1=f_2=0$ a result similar to Theorem \ref{goforitt} will be proved; cf. Corollary \ref{C1}.

The main flavor of our results above is ``global on $E$;" however, by looking at the completion of $E$ at the origin one  obtains in particular,  the following integrality statement; cf. Corollary \ref{3elliptic}:

\begin{theorem}\label{macheama}
Let $E$ be an elliptic curve over $R_{\pi}$ with  logarithm $\sum_{N=1}^{\infty} \frac{b_N}{N}T^N$, $b_N\in R_{\pi}$. Let $\mu,\nu\in \mathbb M_2^{2,+}$ and let $r,s\in \{1,2\}$ be the lengths of the words $\mu,\nu$, respectively. Assume $r\geq s$.
Then the following relations hold for all $N\geq 1$:
\begin{equation}
\tilde{f}_{\nu}\frac{\phi_{\mu}(b_N)}{N}-\tilde{f}_{\mu}\frac{\phi_{\nu}(b_{p^{r-s}N})}{p^{r-s}N}+f_{\mu,\nu}\frac{b_{p^rN}}{p^rN}\in pR_{\pi}.
\end{equation}
\end{theorem}

This integrality statement can be viewed as an analogue (for several ``conjugates" of an elliptic curve) of  the integrality statement of Atkin and Swinnerton-Dyer for a given elliptic curve \cite{ASD71, S87}.

The next step in the theory will be to extend some of 
the theory of $\delta$-modular forms \cite{Bu05} to the PDE case by defining 
{\bf partial  $\delta$-modular forms} (whose weights are $\mathbb Z$-linear combinations of non-commutative monomials in $\phi_1,\phi_2$) and {\bf isogeny covariance} for such forms; cf Section \ref{modforms}.
We will also attach  {\bf symbols} to isogeny covariant partial $\delta$-modular forms for weights of degree $-2$; these symbols are, again, $K$-linear combinations of non-commutative monomials in $\phi_1,\phi_2$. 

To state our main result we need to consider  the standard modular curve $Y_1(N)=X_1(N)\setminus \{\textup{cusps}\}$ over $R_{\pi}$ (with $N\geq 4$ coprime to $p$) and the 
natural $\mathbb G_m$-bundle $B=B_1(N)$ over the $Y_1(N)$; so $B$ classifies pairs  consisting of an elliptic curve with $\Gamma_1(N)$-structure and  a basis for the $1$-forms.  Let $B_{\textup{ord}}$ be the preimage in $B$ of the ordinary locus in $Y_1(N)$.
We will show (cf. Theorems \ref{theyareisogcov}, \ref{theyareover}, \ref{roo}, \ref{quadocc}, \ref{mult1},  Proposition  \ref{zaza} and Corollary \ref{corfrodo2}) the following characterization of these forms.

\begin{theorem} The following holds:

1) The  classes $f_{\mu}$ and $f_{\mu,\nu}$ 
 are induced by  isogeny covariant partial $\delta$-modular forms, denoted by $f^{\textup{jet}}_{\mu}$ and $f^{\textup{jet}}_{\mu,\nu}$, of weight $-1-\phi_{\mu}$ and $-\phi_{\mu}-\phi_{\nu}$, respectively.
 
 2) There exists $c\in \mathbb Z_p^{\times}$ such that for every distinct words $\mu,\nu\in \mathbb M_2^{2,+}$ of lengths  $r,s\in \{1,2\}$, respectively,  the symbols of $f^{\textup{jet}}_{\mu}$
 and $f^{\textup{jet}}_{\mu,\nu}$ are equal to  $c(\phi_{\mu}-p^r)$ and $c(p^s\phi_{\mu}-p^r\phi_{\nu})$, respectively.

  3) The 
 forms  $f^{\textup{jet}}_{\mu}$ and $f^{\textup{jet}}_{\mu,\nu}$  naturally induce  
 totally overconvergent arithmetic $\textup{PDE}$s on $B$ and the induced
 maps $B(R^{\textup{alg}})\rightarrow K^{\textup{alg}}$
restricted  to $B_{\textup{ord}}(R^{\textup{alg}})$
 extend to continuous  maps $B_{\textup{ord}}(\mathbb C_p^{\circ})\rightarrow \mathbb C_p$.
 
 4) The form  $f^{\textup{jet}}_{1,2}$ is a basis modulo torsion of the module of isogeny covariant partial $\delta$-modular forms of order $\leq 1$ and weight $-\phi_1-\phi_2$.
\end{theorem}

The forms $f^{\textup{jet}}_{\mu}, f^{\textup{jet}}_{\mu,\nu}$ in the theorem   satisfy a series of 
  cubic and quadratic relations (cf. Theorems \ref{roo} and \ref{quadocc}). We will use these relations to determine the {\bf Serre-Tate expansions} of the forms involved
  (cf. Theorem \ref{nonzzero}) which is what, in particular, leads to the determination of  the corresponding symbols in part 2 of the theorem above. As a consequence of these computations we will derive some explicit formulae for the values of $\delta$-characters in terms of Serre-Tate parameters. These formulae will exhibit a somewhat unexpected antisymmetry property that translates into a {\bf Reciprocity Theorem} for arithmetic Manin maps which we roughly explain in what follows; cf. Theorem \ref{reciporc} for  details and a precise statement.  
  
  \

  We now come to some applications of these results. Let $E_0$ be an ordinary elliptic curve over $R_{\pi}/\pi R_{\pi}$.
  For  every $\alpha, \beta\in R_{\pi}$ with absolute value less than 
  $p^{-\frac{1}{p-1}}$  let $E_{\alpha}$ and $E_{\beta}$ be the elliptic curves over $R_{\pi}$ with reduction  $E_0$ and with logarithms of the Serre-Tate parameters equal to $\alpha$ and $\beta$ respectively. Let $\omega_{\alpha}$ and $\omega_{\beta}$ be the $1$-forms on these elliptic curves respectively canonically induced from the universal elliptic curve lifting $E_0$.
  Furthermore let $P_{\alpha,\beta}\in E_{\beta}(R_{\pi})$ and
  $P_{\beta,\alpha}\in E_{\alpha}(R_{\pi})$
  be the points whose elliptic logarithms equal $\alpha$ and $\beta$, respectively. Finally, for every distinct $\mu,\nu\in \mathbb M^{2,+}_2$ let $\psi_{\mu,\nu,\alpha}$ and $\psi_{\mu,\nu,\beta}$ be the corresponding partial $\delta_{\pi}$-characters attached to $(E_{\alpha},\omega_{\alpha})$ and $(E_{\beta},\omega_{\beta})$, respectively.  The Reciprocity Theorem for arithmetic Manin maps referred to above is the following statement:
  
 \begin{theorem} The following equality holds,
  $$\psi_{\mu,\nu,\beta}^{\textup{alg}}(P_{\alpha,\beta})=\psi_{\nu,\mu,\alpha}^{\textup{alg}}(P_{\beta,\alpha}).$$
 \end{theorem}
 
 In particular, the $\delta_{\pi}$-character $\psi_{\mu,\nu,\beta}^{\textup{alg}}$ of $E_{\beta}$ vanishes at the point $P_{\beta,\beta}$ corresponding to the Serre-Tate parameter of $E_{\beta}$, i.e., $\psi_{\mu,\nu,\beta}^{\textup{alg}}(P_{\beta,\beta})=0$.
 In case $\pi=p$ this vanishing (but not the reciprocity statement itself) also follows from the fact that $P_{\beta,\beta}$ is infinitely $p$-divisible.
 
 In addition, our considerations above will lead to an explicit description 
 of the kernel of $\psi_{\mu,\nu,\beta}^{\textup{alg}}$ in terms of $\beta$. This  result can be viewed as an arithmetic PDE {\bf Theorem of the Kernel} analogue of Manin's Theorem of the Kernel \cite{Man63} and extending the arithmetic ODE results in \cite[Thm. A and B]{Bu95} and  \cite[Thm. 1.6]{Bu97}. This utilizes an interesting pairing defined as follows.
  
  For $\mu,\nu$ of lengths $r$ and $s$ respectively define the  $\mathbb Q_p$-bilinear map
$$\langle\ \ ,\ \ \rangle_{\mu,\nu}:K^{\textup{alg}}\times K^{\textup{alg}}\rightarrow K^{\textup{alg}}$$
by the formula
$$
\langle \alpha, \beta \rangle_{\mu,\nu}=\beta^{\phi_{\nu}}\alpha^{\phi_{\mu}}-\beta^{\phi_{\mu}}\alpha^{\phi_{\nu}}
+p^s(\alpha \beta^{\phi_{\mu}}-\beta\alpha^{\phi_{\mu}})+p^r(\beta\alpha^{\phi_{\nu}}-\alpha\beta^{\phi_{\nu}}).$$
 The version of the Theorem of the Kernel (cf.  Theorem \ref{frodooo}) is as follows.
 
 \begin{theorem}\label{frodooops} 
We have a natural group isomorphism
$$\text{Ker}(\psi_{\mu,\nu,\beta}^{\textup{alg}})\otimes_{\mathbb Z}\mathbb Q\simeq
\{\alpha\in K^{\textup{alg}}\ |\ \langle \alpha,\beta\rangle_{\mu,\nu} =0\}.$$
\end{theorem}

Note that for $E_{\beta}$ ordinary with $\beta$  not a root of unity
we have that $\text{Ker}(\psi_{\mu,\nu,\beta}^{\textup{alg}})$ (which always contains the torsion group of $E_{\beta}(R^{\textup{alg}})$) does not reduce to the torsion group.

 Another primary application of our theory of $\delta$-modular forms is the construction, for every weight $w$, of a {\bf $\delta$-period map}
 $$\mathfrak p_w:Y_1(N)(R^{\textup{alg}})_w^{\textup{ss}}{\longrightarrow} 
 \mathbb P^{N_w}(R^{\textup{alg}})$$
 where $Y_1(N)(R^{\textup{alg}})_w^{\textup{ss}}\subset Y_1(N)(R^{\textup{alg}})$ is a natural set of {\bf semistable} points; cf. Definition \ref{defpermap}. 
 The terminology is motivated by the following analogy with geometric invariant theory. Group actions are replaced, in our setting, with the action of Hecke correspondences and the ``components" of our $\delta$-period maps are given by isogeny covariant $\delta$-modular forms which should be viewed as analogues of invariant sections of line bundles in geometric invariant theory.
 As we shall see the $\delta$-period maps are rather non-trivial already for  $w$ of order $2$ and degree $4$;
 cf. Example \ref{opopop}.  On the other hand isogeny covariance will imply the following theorem (cf. Theorem \ref{thmpermap}).
 
  \begin{theorem}\label{thmpermap1}
 The $\delta$-period maps $\mathfrak p_w$
 are  constant on prime to $p$ isogeny classes.
 \end{theorem}
 
 Next, as in \cite{Bu00, Bar03, Bu05}, we will construct
 certain `crystalline forms' $f_{\mu}^{\textup{crys}}, f_{\mu,\nu}^{\textup{crys}}$ and prove they are proportional to the forms $f^{\textup{jet}}_{\mu}, f^{\textup{jet}}_{\mu,\nu}$; cf. Corollary \ref{schwazenger}.
 In addition, 
 we will consider {\bf $\delta$-modular forms on the ordinary locus}. (Such forms were called ``ordinary" in \cite[Ch. 8]{Bu05} but here we will avoid this term so that no confusion arises with its use in the ODE/PDE distinction.) Then
 using a crystalline construction as in loc.cit. we will   completely determine the structure of the spaces of isogeny covariant  $\delta$-modular forms on the ordinary locus for the weights of degree $0$ and $-2$; cf. Corollary \ref{stevee}.

Finally, we put Theorem~\ref{thm:mainapp} in context. Specifically, our results about $\delta$-characters and $\delta$-modular forms can be combined to give a Diophantine application along the lines of \cite{BP09}; cf. Corollary \ref{finnn}. The critical map $\psi$ in Theorem~\ref{thm:mainapp} is the following. Fix $\pi\in \Pi$, an elliptic curve $E$ over $R_{\pi}$ and we recall our $\delta_{\pi}$-character $\psi_{1,2}$. The map $\psi$ in Theorem~\ref{thm:mainapp} is the induced group homomorphism $$\psi_{R_{\pi}}:=(\psi_{1,2})_{R_{\pi}} \colon E(R_{\pi})\rightarrow R_{\pi}.$$

The proof of Theorem~\ref{thm:mainapp} utilizes a version of the classic theorem of Strassman, see Lemma~\ref{strass}, which may be of independent interest. In particular, the classic Strassman theorem concerns the case of the affine line, where as this version is a general curve over a DVR. It would be immediate to conclude a uniform version of Theorem~\ref{thm:mainapp} from a uniform version of Lemma~\ref{strass}.

\subsection{Leitfaden} In Section 2, we begin by discussing Frobenius lifts and Frobenius automorphisms of $K^{\textup{alg}}$ after which we introduce partial $\delta_{\pi}$-jet spaces which are a PDE analogue of the ODE $\pi$-jet spaces in \cite{Bu95}. In Section 3, we introduce and study $\delta_{\pi}$-characters of group schemes as well as their Picard-Fuchs symbols. 
Section 4, is devoted to analyzing these concepts for the mulitiplicative group $\mathbb G_m$.
Section 5, does a similar analysis for elliptic curves. Here we introduce and study  the arithmetic Kodaira-Spencer classes $f_{\mu},f_{\mu,\nu}$ and the  
$\delta_{\pi}$-characters $\psi_{\mu,\nu}$.
All of the above discussion is made in the context of an arbitrary number $n$ of Frobenius automorphisms and an arbitrary order $r$. We explain a detailed specialization of the discussion of elliptic curves to the case $n=r=2$. Here we derive, in this case, a series of basic quadratic and cubic relations satisfied by the arithmetic Kodaira-Spencer classes. Section 6, summarily explains how all of the above theory can be developed in a ``relative setting;" this is necessary for Section 7 where we introduce partial $\delta$-modular forms which are a PDE version of the ODE concept introduced in \cite{Bu00}. The relative arithmetic Kodaira-Spencer classes define such forms. We then introduce and compute the Serre-Tate expansions of these forms, we construct our $\delta$-period maps, and we derive  the Theorem of the Kernel and the Reciprocity Theorem for arithmetic Manin maps. We continue by discussing the crystalline side of the story and forms on the ordinary locus and we present a construction of finite covers defined by $\delta$-modular forms (cf. Theorem \ref{fllow}) which is then used to prove our main Diophantine application to modular parameterizations (cf. Corollary
\ref{finnn}). We end by introducing a PDE version of the ODE  $\delta$-Serre operators in \cite{Bar03,Bu05}; these PDE $\delta$-Serre operators lead to  genuine (not arithmetic) PDEs
satisfied by our arithmetic PDEs and can be viewed as Pfaffian systems of equations of the arithmetic jet spaces.

\subsection{Acknowledgements} The authors are indebted to  Florian Pop for inspiring comments on the Galois groups of local fields. The first author was partially supported by the Simons Foundation through award 615356. The second author is grateful for the hospitality extended during a lengthy visit to UNM where the paper was largely written. 

\section{Purely arithmetic PDEs}

\subsection{Frobenius automorphisms}
We start with the following standard definition.

\begin{definition}\label{frobliftt}
By a {\bf Frobenius lift} for an $A$-algebra $\varphi:A\rightarrow B$ 
we understand a ring homomorphism $\phi:A\rightarrow B$ such that the induced homomorphism $\overline{\phi}:A/pA\rightarrow B/pB$
equals the composition of the induced homomorphism $\overline{\varphi}:A/pA\rightarrow B/pB$ with the $p$-power Frobenius on $A/pA$. If $B=A$ and $\varphi=1_A$ we say that $\phi$ is a {\bf Frobenius lift} on $A$.\end{definition}

For every, not necessarily algebraic, field extension $F\subset L$ we denote by $\mathfrak G(L/F)$ the group of all field automorphisms of $L$ that are the identity on $F$.
For every field $L$  we denote by $\mathfrak G_L$ the absolute Galois group $\mathfrak G(L^{\textup{alg}}/L)$, where  $L^{\textup{alg}}$ is an algebraic closure of $L$.

We recall the main setting in \cite{BM20}. Consider the field of $p$-adic numbers with absolute value $|\ |$ normalized by $|p|=p^{-1}$.
Let ${\mathbb Q}_p^{\text{alg}}$ be an algebraic closure of ${\mathbb Q}_p$, let
${\mathbb Q}_p^{\text{ur}}$  be the maximum unramified extension of ${\mathbb Q}_p$
inside ${\mathbb Q}_p^{\text{alg}}$, let $K$ be the metric completion of ${\mathbb Q}_p^{\text{ur}}$ and let $K^{\text{alg}}$ be the algebraic closure of $K$ in the metric completion ${\mathbb C}_p$
of ${\mathbb Q}_p^{\text{alg}}$. We still denote by $|\ |$ the induced absolute value on all of these fields.  We denote by
${\mathbb Z}_p^{\text{ur}}, {\mathbb Z}_p^{\text{alg}}, R, R^{\text{alg}}, \OCp$
the valuation rings of 
${\mathbb Q}_p^{\text{ur}}, {\mathbb Q}_p^{\text{alg}}, K, K^{\text{alg}}, {\mathbb C}_p$, 
respectively.
In particular, 
$R:=\widehat{{\mathbb Z}_p^{\text{ur}}}$.
  We set $k:=R/pR$; so $k\simeq \mathbb F_p^{\textup{alg}}$. 
  
 \begin{remark}
 The natural ring homomorphism
  \begin{equation}
  \label{tensorstr}
  \mathbb Q_p^{\text{alg}}\otimes_{\mathbb Q_p^{\text{ur}}} K\rightarrow K^{\text{alg}}
  \end{equation}
  is an isomorphism.
  Indeed this map is surjective because
  by Krasner's Lemma,  we have $K^{\text{alg}}:=K{\mathbb Q}_p^{\text{alg}}$; cf.  \cite[pg. 149, Prop. 5]{BGR84}. To check that the map (\ref{tensorstr}) is injective write  and $\mathbb Q_p^{\text{alg}}=\cup F_i$ 
with $F_i/\mathbb Q_p$ finite and let $F_i^0\subset F_i$ be the maximum unramified extension of $\mathbb Q_p$ contained in $F_i$; so
 $F_i/F_i^0$ is totally ramified and $\mathbb Q_p^{\text{ur}}=\cup F_i^0$. 
 It is enough to check that $F_i\otimes_{F_i^0} K\rightarrow K^{\text{alg}}$ is injective for all $i$. To check this note that $F_i/F_i^0$ is generated by a root of an Eisenstein polynomial $f_i$ in $F_i^0[x]$; but every such polynomial is an Eisenstein polynomial in $K[x]$ and so $F_i\otimes_{F_i^0} K=K[x]/(f_i)$ is a field, therefore  it injects into $K^{\text{alg}}$. \end{remark}

   \begin{definition}\label{frobautt}
   Let $L$ be a subfield of $\mathbb C_p$ containing $\mathbb Q_p$. 
   A {\bf  Frobenius automorphism} of $L$
is a continuous automorphism $\phi\in \mathfrak G(L/\mathbb Q_p)$ such that $\phi$
induces the $p$-power Frobenius on the residue field of the valuation ring  of $L$.
We denote by $\mathfrak F(L/\mathbb Q_p)$ the set of Frobenius automorphisms of $L$.
\end{definition}

Note that if $\phi$ is an Frobenius automorphism of $K^{\textup{alg}}$  then $\phi$ 
sends $R$ into $R$, induces the Frobenius lift on $R$, and induces 
an automorphism of $R^{\text{alg}}$ (which  is however not  a Frobenius lift on $R^{\text{alg}}$ in the  sense of Definition \ref{frobliftt}).  Conversely every  automorphism $\phi\in \mathfrak G(K^{\textup{alg}}/\mathbb Q_p)$ extending the Frobenius lift on $R$ is a Frobenius automorphism of $K^{\textup{alg}}$. Indeed, for every finite Galois extension $L_0$ of $\mathbb Q_p$,  the 
field $L:=L_0K$ is sent onto itself by $\phi$ and the absolute values $|\ |$ and $|\phi(\ )|$ on $L$ have the same restriction to $K$, hence must coincide; cf. \cite[p.32]{L86}; in particular $\phi$ is continuous and induces the $p$-power Frobenius on $k$.

 The set $\mathfrak F(K^{\text{alg}}/\mathbb Q_p)$ is a principal homogeneous space for  the absolute Galois group $\mathfrak G_K$ 
under the action given by 
$(\gamma,\phi)\mapsto \gamma \phi$  for $\phi\in \mathfrak F(K^{\text{alg}}/\mathbb Q_p)$  and $\gamma\in \mathfrak G_K$. On the other hand, by the fact that the homomorphism (\ref{tensorstr}) is an isomorphism we immediately get that the restriction homomorphism $\mathfrak G_K\rightarrow \mathfrak G_{\mathbb Q_p^{\text{ur}}}$ an isomorphism of topological groups and  the  restriction map $\mathfrak F(K^{\text{alg}}/\mathbb Q_p)\rightarrow \mathfrak F(\mathbb Q_p^{\text{alg}}/\mathbb Q_p)$  is a bijection.
Note that $\mathfrak F(\mathbb Q_p^{\text{alg}}/\mathbb Q_p)$ has a purely (topological) group characterization as a subset
of $\mathfrak G_{\mathbb Q_p}$; cf. \cite{NSW00}, Lemma 12.1.8, p. 665.
(The elements of $\mathfrak F(\mathbb Q_p^{\text{alg}}/\mathbb Q_p)$ are referred to in loc.cit. as  {\it Frobenius lifts} but adopting that terminology here would conflict with our Definition \ref{frobliftt}.)
 
  By the way, 
  the absolute Galois group $\mathfrak G_{\mathbb Q_p}$ is known  to have $4$ topological generators one of which  is in $\mathfrak F(\mathbb Q_p^{\text{alg}}/\mathbb Q_p)$; the relations among these topological generators are also known, cf. \cite{JW82} or \cite[Thm. 7.5.10, p. 360]{N82}. 
  We say that a subset of a topological group is  a set of {\bf topological generators} if the subgroup generated by this set is dense in the group.
  One can  easily see, by the way,  that one can find a set of $4$ topological generators 
  of $\mathfrak G_{\mathbb Q_p}$ 
  that is contained in  $\mathfrak F(\mathbb Q_p^{\text{alg}}/\mathbb Q_p)$. We will not use this observation in what follows. What we will be interested in is the monoid (rather than the group) generated by our Frobenius automorphisms, as explained in the next subsection.

 \subsection{Monomial independence}
 In what follows  monoids will be assumed to possess an identity but will not necessarily be commutative.
 
Let ${\mathbb M}_n$ be the free (non-commutative) monoid  generated by
 the set $\{1,\ldots,n\}$, 
$$\mathbb M_n:=\{0\}\cup \{i_1\ldots i_s\ |\ l\in \mathbb N,\ i_1,\ldots,i_s\in \{1,\ldots,n\}\};$$
its elements will be referred to as {\bf words}, the {\bf length} $|\mu|$ of a word $\mu:=i_1\ldots i_s$ is defined by $|\mu|=s$,
$0$ is called the {\bf empty word} and its length is defined by $|0|=0$. 
 Multiplication is given by concatenation $(\mu,\nu)\mapsto \mu\nu$ and $0$ is the identity element. For all $r\in \mathbb N\cup\{0\}$ let ${\mathbb M}^r_n$ be the set of all elements in ${\mathbb M}_n$ of length $\leq r$. Set $\mathbb M_n^+:=\mathbb M_n\setminus \{0\}$ and 
$\mathbb M_n^{r,+}:=\mathbb M_n^r\setminus \{0\}$.

\begin{definition}
A family of distinct elements $\phi_1,\ldots,\phi_n$ in a  monoid $\mathfrak G$  is called {\bf monomially independent} if the  monoid homomorphism
$${\mathbb M}_n\rightarrow \mathfrak G,\ \ \mu=i_1\ldots i_l \mapsto \phi_{\mu}:=\phi_{i_1}\ldots \phi_{i_l},\ \ 0\mapsto 1$$
 is injective.\end{definition}

\begin{remark}
Note that in our notation above we have the formula
$\phi_{\mu}\phi_{\nu}=\phi_{\mu\nu}$ for $\mu,\nu \in \mathbb M_n$.
Note also that if  $\mathfrak G$ is a group and $\phi_1,\ldots,\phi_n\in \mathfrak G$ are monomially independent in $\mathfrak G$ then the subgroup of $\mathfrak G$ generated by $\phi_1,\ldots,\phi_n$ is not necessarily freely generated by $\phi_1,\ldots,\phi_n$; an example that naturally occurs in our context is given in Remark \ref{iwi}.
\end{remark}

The following lemma follows trivially from the well known ``algebraic independence of field automorphisms"
 but, for convenience, we provide a proof.


\begin{lemma}\label{preartin}
Let $L$ be a field of characteristic zero and let $\phi_1,\ldots,\phi_n$ 
be monomially independent elements in $\mathfrak G(L/\mathbb Q)$. Let 
$F=F(\ldots,x_{\mu},\ldots)$ be a polynomial with $L$-coefficients in the variables 
$x_{\mu}$ with $\mu\in \mathbb M_n$ and consider the function $f:L\rightarrow L$ defined by
$$f(\lambda)=F(\ldots,\phi_{\mu}(\lambda),\ldots),\ \ \lambda\in L.$$
Let $A$ be a subring of $L$ and 
assume $f(\lambda)=0$ for all $\lambda\in A$. Then $F=0$.
\end{lemma}

{\it Proof}. 
By  Artin's independence of characters, cf. \cite[p. 283]{L02}, if $\mathfrak A$ is a monoid then
 every family of distinct monoid homomorphisms $\mathfrak A\rightarrow L^{\times}$ is $L$-linearly independent in the $L$-linear space  of all maps $\mathfrak A\rightarrow L$. Let $\mathfrak A=A\setminus \{0\}$.
Then by Artin's independence of characters
it is enough to check that 
for distinct vectors $e:=(e_{\mu})_{\mu\in \mathbb M_n}$ with entries non-negative integers, almost all zero, 
the maps
$f_e:A\rightarrow L$ defined by
$$f_e(\lambda):=\prod_{\mu}(\phi_{\mu}(\lambda))^{e_{\mu}},\ \ \ \lambda\in A$$
are distinct. Assume $f_e=f_{e'}$ and let us show that $e=e'$. For all integers $m\in \mathbb Z$ we have
$$\prod_{\mu}(m+\phi_{\mu}(\lambda))^{e_{\mu}}=\prod_{\mu}(m+\phi_{\mu}(\lambda))^{e'_{\mu}},\ \ \lambda\in A.$$
Since  $L$ has characteristic zero we have an equality
$$\prod_{\mu}(t+\phi_{\mu}(\lambda))^{e_{\mu}}=\prod_{\mu}(t+\phi_{\mu}(\lambda))^{e'_{\mu}},\ \ \lambda\in A$$
in the ring of polynomials $L[t]$.
Looking at degrees in $t$ we get $\sum_{\mu} e_{\mu}=\sum_{\mu} e'_{\mu}=:d$.
Picking out the coefficient of $t^{d-1}$ we get
$$\sum_{\mu} e_{\mu}\phi_{\mu}(\lambda)=\sum_{\mu} e'_{\mu}\phi_{\mu}(\lambda),\ \ \lambda\in A.$$ 
By monomial independence of the $\phi_i$'s and,  again, by Artin's independence of characters, 
the family 
 $(\phi_{\mu})_{\mu\in \mathbb M_n}$
 is $L$-linearly independent in the $L$-linear space  of maps $\mathfrak A\rightarrow L$
 so, since $L$ has characteristic zero,  we conclude that $e_{\mu}=e'_{\mu}$ for all $\mu$.
\qed

\begin{example}
In what follows we show  that the set $\mathfrak F(K^{\text{alg}}/\mathbb Q_p)$ of Frobenius automorphisms of $K^{\textup{alg}}$ contains large subsets of monomially independent elements that remain monomially independent on ``small" (abelian) extensions of $K$.
We recall some standard constructions from Iwasawa theory; cf.  \cite{I55}. Let $l\neq p$ be a prime. Consider sequences $\pi_{m}\in K^{\text{alg}}$ and $\zeta_{l^m}\in K$ with $m\geq 0$ such that 
$$\pi_0=p,\ \ \zeta_{l^0}=1,\ \ \pi_{m+1}^l=\pi_m,
\ \ \zeta_{l^{m+1}}^l=\zeta_{l^m},\ \ m\geq 0.$$
 Since the polynomial $x^{l^m}-p$ is Eisenstein over $K$ and $\pi_m$ is one of its roots
we have that the field $K_{\pi_m}:=K(\pi_m)$ generated by $\pi_m$ is isomorphic to $K[x]/(x^{l^m}-p)$ and  $K_{\pi_m}$ is Galois over $K$ with cyclic Galois group of order $l^m$ generated by the automorphism $\tau_m$ satisfying $\tau_m \pi_m=\zeta_{l^m}\pi_m$. Define 
\begin{equation}
\label{Kelll}K^{(l)}:=\bigcup_{m\geq 0} K_{\pi_m}.\end{equation}
 Clearly the automorphisms $\tau_m$ are compatible and yield 
an automorphism $\tau_{(l)}\in \mathfrak G(K^{(l)}/K)$. For all $\gamma\in \mathbb Z_l$ one defines $\tau_{(l)}^{\gamma}\in \mathfrak G(K^{(l)}/K)$ as follows: if $\gamma\equiv b_m$ mod $l^m$ with $b_m\in \mathbb Z$ then one lets $\tau_{(l)}^{\gamma}$ to be $\tau_{(l)}^{b_m}$ on $K(\pi_m)$. Then the map $\mathbb Z_l\rightarrow \mathfrak G(K^{(l)}/K)$ given by $\gamma\mapsto \tau_{(l)}^{\gamma}$ is an isomorphism.
On the other hand the fields $K_{\pi_m}$ possess compatible automorphisms  extending the Frobenius lift on $R$ and fixing the $\pi_m$'s; they induce an automorphism $\phi_{(l)}$ on $K^{(l)}$. One trivially checks that $\phi_{(l)}\tau_{(l)}$ and $\tau_{(l)}^p\phi_{(l)}$ coincide on all roots of unity in $K$ 
(and hence on $K$) and also on all $\pi_m$'s; so $\phi_{(l)}\tau_{(l)}=\tau_{(l)}^p\phi_{(l)}$ in $\mathfrak G(K^{(l)}/\mathbb Q_p)$. 
For each $\gamma\in \mathfrak B$ we
set $\phi^{(\gamma)}_{(l)}:=\tau_{(l)}^{\gamma}\phi_{(l)}\in \mathfrak G(K^{(l)}/\mathbb Q_p)$ and we
 let $\phi^{(\gamma)}\in \mathfrak F(K^{\text{alg}}/\mathbb Q_p)$ be an arbitrary extension of $\phi^{(\gamma)}_{(l)}$.\end{example}

\begin{proposition}\label{P1} The following hold:

1) $\phi^{(0)}_{(l)},\ldots, \phi^{(p-1)}_{(l)}$ are monomially independent in
$\mathfrak G(K^{(l)}/\mathbb Q_p)$. In particular  $\phi^{(0)},\ldots, \phi^{(p-1)}$ are monomially independent in
$\mathfrak G(K^{\textup{alg}}/\mathbb Q_p)$.

2) Let $\gamma_1,\ldots,\gamma_n\in \mathbb Z_l$ be $\mathbb Z$-linearly independent. 
Then $\phi^{(\gamma_1)}_{(l)},\ldots, \phi^{(\gamma_n)}_{(l)}$ are monomially independent in
$\mathfrak G(K^{(l)}/\mathbb Q_p)$; in particular  $\phi^{(\gamma_1)},\ldots, \phi^{(\gamma_n)}$ are monomially independent in
$\mathfrak G(K^{\textup{alg}}/\mathbb Q_p)$.
\end{proposition}

\begin{proof} 
We will prove Part 2. Part 1 is proved similarly.
Write $\phi_{i,l}:=\phi^{(\gamma_i)}_{(l)}$ for $i\in \{1,\ldots,n\}$.
Let 
$\mu=i_1\ldots i_s $ where $i_1,\ldots,i_s\in \{1,\ldots,n\}$ and similarly $\mu'=i'_1\ldots i'_{s'}$ where $i'_1,\ldots,i'_{s'}\in \{1,\ldots,n\}$. Assume 
$$\phi_{i_1,l}\ldots \phi_{i_s,l}=\phi_{i'_1,l}\ldots \phi_{i'_{s'},l}$$
 and let us prove that $\mu=\mu'$.
We first note that for all  integers $j\geq 0$ we have $\phi_{(l)} \tau_{(l)}^j=\tau_{(l)}^{pj}\phi_{(l)}$; this follows by induction on $j$. 
 We conclude that $\phi_{(l)} \tau_{(l)}^{\gamma}=\tau_{(l)}^{p\gamma}\phi_{(l)}$ for all $\gamma\in \mathbb Z_l$;
 this equality holds  because it holds on every $K_{\pi_m}$.
Next note that for all integers $i\geq 1$ and for all $\gamma\in \mathbb Z_l$ we have $\phi_{(l)}^i \tau_{(l)}^{\gamma}=\tau_{(l)}^{p^i \gamma} \phi_{(l)}^i$; this follows by induction on $i$. 
Using the latter equalities we get
$$\phi_{i_1,l}\ldots \phi_{i_s,l}=\tau_{(l)}^{\gamma_{i_1}}\phi_{(l)}\ldots \tau_{(l)}^{\gamma_{i_s}}\phi_{(l)}=\tau_{(l)}^{\gamma_{i_1}+p\gamma_{i_2}+\cdots+p^{s-1}\gamma_{i_s}}\phi_{(l)}^s$$
and similarly for $\mu'$ so we get
$$\tau_{(l)}^{\gamma_{i_1}+p\gamma_{i_2}+\cdots+p^{s-1}\gamma_{i_s}}\phi_{(l)}^s=\tau_{(l)}^{\gamma_{i'_1}+p\gamma_{i'_2}+\cdots+p^{s'-1}\gamma_{i_{s'}}}\phi_{(l)}^{s'}.$$
Since $\phi_{(l)}$ has infinite order on $K$ we get $s=s'$. Since $\tau_{(l)}$ has infinite order on $K^{(l)}$ we get
\begin{equation}
\label{twosums}
\gamma_{i_1}+p\gamma_{i_2}+\cdots+p^{s-1}\gamma_{i_s}=\gamma_{i'_1}+p\gamma_{i'_2}+\cdots+p^{s-1}\gamma_{i'_{s}}\end{equation}
 in $F$.
We will be done   if we prove the following:

\medskip

{\it Claim}. An equality of the form (\ref{twosums}) 
implies that $i_j=i'_j$ for all $j\in \{1,\ldots,s\}$.

\medskip

 The claim can proved by induction on $s$. 
The case $s=1$ is trivial.  The induction step follows if we show that the equality (\ref{twosums}) implies that $i_1=i'_1$. Assume $i_1\neq j_1$ and seek a contradiction. 
Recalling that $\gamma_1,\gamma_2,\ldots,\gamma_n$ 
are $\mathbb Z$-linearly independent 
write the left hand side of (\ref{twosums}) as a sum
 $\sum_{i=1}^nc_i\gamma^i$  with $c_i\in \mathbb Z$ and write the right hand side of (\ref{twosums}) as a sum
 $\sum_{i=1}^nc'_i\gamma_i$ with $c'_i\in \mathbb Z$. So $c_i=c'_i$ for all $i$. Since $\gamma_{i_1}\neq \gamma_{i'_1}$ we get that $c_{i_1}\equiv 1$ mod $p$ while $c'_{i_1}\equiv 0$ mod $p$, a contradiction.
This ends the proof of our claim  and hence of our proposition.
\end{proof}

\begin{remark}\label{iwi}
Note that, in spite of the fact that $s_1:=\phi^{(0)}_{(l)}=\phi_{(l)}$ and $s_2:=\phi_{(l)}^{(1)}=\tau_{(l)}\phi_{(l)}$ are monomially independent 
in $\mathfrak G(K^{(l)}/\mathbb Q_p)$ we have that the subgroup of $\mathfrak G(K^{(l)}/\mathbb Q_p)$ generated by $s_1$ and $s_2$ is not freely generated by 
$s_1$ and $s_2$; indeed we have the following relation:
$$s_1 (s_2 s_1^{-1})=(s_2 s_1^{-1})^p s_1.$$
\end{remark}

\subsection{$\pi$-Frobenius lifts}

Throughout the paper we denote by $\Pi$ the set of all roots $\pi$ in ${\mathbb Q}_p^{\text{alg}}$ of Eisenstein polynomials with coefficients in $\bZ_p^{\text{ur}}$ having the property that the extension
${\mathbb Q}_p(\pi)/{\mathbb Q}_p$ is Galois. 
So ${\mathbb Q}_p^{\text{alg}}={\mathbb Q}_p^{\text{ur}}(\Pi)$.
For any $\pi\in \Pi$ write $K_{\pi}=K(\pi)$ and let 
$R_{\pi}=R[\pi]$ be the valuation ring of $K_{\pi}$. We write $\pi'|\pi$ if and only if $K_{\pi}\subset K_{\pi'}$. 
Note that
$K^{\text{alg}}=\bigcup_{\pi\in \Pi} K_{\pi}$.
Clearly for $\pi\in \Pi$ the field $K_{\pi}$ is mapped into itself by 
every  Frobenius automorphism $\phi$ of  $K^{\text{alg}}$. By continuity of $\phi$
 we have an induced automorphism 
$\phi_{\pi}:R_{\pi}\ra R_{\pi}$ (which we sometimes still denote by $\phi$) inducing the $p$-power Frobenius on $R_{\pi}/\pi R_{\pi}=k$. 
More generally we will need the following:

\begin{definition}
Let $A$ be an $R_{\pi}$-algebra.  
By a {\bf $\pi$-Frobenius lift} for an $A$-algebra $\varphi:A\rightarrow B$
 we understand a ring homomorphism $\phi:A\rightarrow B$ such that the induced homomorphism $\overline{\phi}:A/\pi A\rightarrow B/\pi B$
equals the composition of the induced homomorphism $\overline{\varphi}:A/\pi A\rightarrow B/\pi B$ with the $p$-power Frobenius on $A/\pi A$. If $B=A$ and $\varphi=1_A$ we say that $\phi$ is a {\bf $\pi$-Frobenius lift} on $A$. \end{definition}

In particular for every  Frobenius automorphism $\phi$ of  $K^{\text{alg}}$ and every $\pi\in \Pi$ the induced automorphism $\phi_{\pi}$ of $R_{\pi}$ 
is a $\pi$-Frobenius lift.

\subsection{Rings of symbols}

\begin{definition}
Consider a family  $\Phi:=(\phi_1,\ldots,\phi_n)$, $\phi_i \in \mathfrak F(K^{\text{alg}}/\mathbb Q_p)$ of distinct Frobenius automorphisms
and let $\pi\in \Pi$. Let $\mathbb M_{\Phi}$ be the free monoid on the set $\Phi$; so we have an isomorphism $\mathbb M_n \simeq \mathbb M_{\Phi}$, $i\mapsto \phi_i$.
We define the {\bf ring of symbols}
 $K_{\pi,\Phi}$ 
 to be the free $K_{\pi}$-module with basis $\mathbb M_{\Phi}$ 
 equipped with multiplication defined by
 \begin{equation}
 \label{relationss}
 \phi_i \cdot \lambda=\phi_i(\lambda)\cdot \phi_i\end{equation}
 for $\lambda\in K_{\pi}$, $i\in \{1,\ldots,n\}$.
 If in the above definition we replace $K_{\pi}$ we obtain a ring
 $R_{\pi,\Phi}$. \end{definition}

  So every element in 
$K_{\pi,\Phi}$ (respectively $R_{\pi,\Phi}$) can be uniquely written as 
$$\sum_{\mu\in \mathbb M_n}\lambda_{\mu}\phi_{\mu}$$ with $\lambda_{\mu}$ in $K_{\pi}$
(respectively in $R_{\pi}$). These rings have filtrations ``by order" given by the subgroups:
$$K^r_{\pi,\Phi}:=\{\sum_{\mu\in \mathbb M_n^r}\lambda_{\mu}\phi_{\mu}\ |\ \lambda_{\mu}\in K_{\pi}\}\subset K_{\pi,\Phi},$$
$$R^r_{\pi,\Phi}:=\{\sum_{\mu\in \mathbb M_n^r}\lambda_{\mu}\phi_{\mu}\ |\ \lambda_{\mu}\in R_{\pi}\}\subset R_{\pi,\Phi}.$$
The ring $K_{\pi,\Phi}$ is a $K_{\pi}$-linear space with  left multiplication by scalars but, of course, it is not a $K_{\pi}$-algebra.
If $\text{End}_{\text{gr}}(K^{\text{alg}})$ denotes the ring of all group endomorphisms of $K^{\text{alg}}$ then we have a natural $K_{\pi}$-linear ring homomorphism 
\begin{equation}
\label{thetatoend}
K_{\pi,\Phi}\rightarrow \text{End}_{\text{gr}}(K^{\text{alg}}),\ \ \theta\mapsto \theta^{\textup{alg}}.
\end{equation}

\begin{remark}\label{Artin}
Note that if $\phi_1,\ldots,\phi_n\in \mathfrak F(K^{\text{alg}}/\mathbb Q_p)$ are monomially independent 
in $\mathfrak G(K^{\text{alg}}/\mathbb Q_p)$
 then, by Lemma \ref{preartin} (and in fact directly from
 Artin's ``independence of characters")   the natural ring homomorphism (\ref{thetatoend}) 
 is injective. 
  \end{remark}
  
  \begin{remark}\label{integral symbols}
  One can also consider the free ring $\mathbb Z_{\Phi}$ generated by $\Phi$ which we refer to as the ring of {\bf integral symbols}; as an abelian group it is the free abelian group with basis $\mathbb M_{\Phi}$. So every element of this ring can uniquely be written as 
  $$w=\sum_{\mu\in \mathbb M_n} m_{\mu}\phi_{\mu},\ \ \ m_{\mu}\in \mathbb Z.$$
  This ring has an order (with non-negative elements defined as those with non-negative coefficients) and has a filtration ``by order" given  by the subgroups $\mathbb Z^r_{\Phi}$ consisting of $\mathbb Z$-linear combinations of elements $\phi_{\mu}$ with $\mu\in \mathbb M_n^r$. 
  Then for all $\lambda\in R_{\pi}^{\times}$ and all $w\in \mathbb Z_{\Phi}$ as above we write
  $$\lambda^w=\prod_{\mu\in \mathbb M_n} 
  (\phi_{\mu}(\lambda))^{m_{\mu}}\in R_{\pi}^{\times}.$$
   For every $w=\sum m_{\mu}\phi_{\mu}\in \mathbb Z_{\Phi}$ we define the {\bf degree} of $w$ to be $\deg(w)=\sum m_{\mu}$.
  \end{remark}

\subsection{Partial $\pi$-jet spaces}
For $\pi\in \Pi$ let
 $C_p(X,Y) \in \mathbb Z[X,Y]$ be the polynomial
 \[C_p(X,Y):=\frac{X^p+Y^p-(X+Y)^p}{p}.\]
 Following \cite{J85,Bu95,Bu96}  a 
 $\pi$-{\it derivation} from an $R_{\pi}$-algebra
 $A$ into an $A-$algebra $B$  is a map $\delta_{\pi}:A \ra B$, $x\mapsto \delta_{\pi}x$,  such that $\d_{\pi}(1)=0$ and
\[\begin{array}{rcl}
\d_{\pi}(x+y) & = &  \d_{\pi} x + \d_{\pi} y
+\frac{p}{\pi}C_p(x,y)\\
\ & \ & \ \\
\d_{\pi}(xy) & = & x^p \cdot \d_{\pi} y +y^p \cdot \d_{\pi} x
+\pi \cdot \d_{\pi} x \cdot \d_{\pi} y,
\end{array}\] for all $x,y \in A$. Given a
$\pi-$derivation as above and denoting by $\varphi:A\rightarrow B$ the structure map of the $A$-algebra $B$ we always denote by $\phi_{\pi}:A \ra B$ the map
$\phi(x)=\varphi(x)^p+\pi \d_{\pi} x$; then $\phi_{\pi}$ is a $\pi$-Frobenius lift.
If $\pi$ is a non-zero divisor in $B$ then the above formula gives a bijection between the set
of $\pi$-derivations from $A$ to $B$ and the set of $\pi$-Frobenius lifts from $A$ to $B$. 

\begin{definition}
By a {\bf partial $\delta_{\pi}$-ring} we understand an $R_{\pi}$-algebra $A$
equipped with an $n$-tuple $(\delta_{\pi,1},\ldots,\delta_{\pi,n})$ of $\pi$-derivations $A\rightarrow A$ . We do {\it not} assume any ``commutation relation" between them. 
\end{definition}

Assume we are given  a family $\Phi:=(\phi_1,\ldots,\phi_n)\in \mathfrak F(K^{\text{alg}}/\mathbb Q_p)^n$ of distinct  Frobenius automorphisms of $K^{\text{alg}}$.  
Note that  for every $\pi\in \Pi$ we get an induced tuple $\Phi_{\pi}=(\phi_{\pi,1},\ldots,\phi_{\pi,n})$ of (not necessarily distinct) 
$\pi$-Frobenius lifts  on $R_{\pi}$,  called the restriction of $\Phi$ to $R_{\pi}$. We therefore get an induced tuple
$(\delta_{\pi,1},\ldots,\delta_{\pi,n})$ of
$\pi$-derivations 
 on $R_{\pi}$ and hence a structure of  partial $\delta_{\pi}$-ring
on
$R_{\pi}$.

Following the lead of \cite{Bu95} we need to  consider the following generalization of the notion of partial $\delta_{\pi}$-ring. 

\begin{definition}\label{proll}
Define a category ${\bf Prol}^*_{\pi,\Phi}$ as follows. An object of this category is  a countable family of $p$-adically complete $R_\pi$-algebras $S^* = (S^r)_{r \geq 0}$ equipped with the following data:

1) $R_\pi$-algebra homomorphisms $\varphi \colon S^r \to S^{r+1}$;

2) $\pi$-derivations $\delta_{\pi,j} \colon S^r \rightarrow S^{r+1}$ for $1 \leq j \leq n$. 

 \noindent We require that $\delta_{\pi,i}$ be compatible with the $\pi$-derivations on $R_{\pi}$ and with $\varphi$, i.e.,  $\delta_{\pi,j} \circ \varphi = \varphi \circ \delta_{\pi,j}$. 
Morphisms are defined in a natural way. We denote by $\phi_{\pi,j}:S^r\rightarrow S^{r+1}$ the corresponding $\pi$-Frobenius lifts, defined by $\phi_{\pi,j}(x)=\varphi(x)^p+\pi\delta_{\pi,j}x$. Also, for all $\mu:=i_1\ldots i_l\in \mathbb M_n$
and all $x\in S^r$ we set $\delta_{\pi,\mu}x:=(\delta_{\pi,i_1} \circ \ldots \circ  \delta_{\pi,i_l})(x)\in S^{r+l}$ and $\phi_{\pi,\mu}x:=(\phi_{\pi,i_1} \circ \ldots \circ  \phi_{\pi,i_l})(x)\in S^{r+l}$.
\end{definition}

The objects of ${\bf Prol}^*_{\pi,\Phi}$ are called {\bf prolongation sequences} (over $R_{\pi}$ with respect to $\Phi$ or $\Phi_{\pi}$). We sometimes identify elements  $a\in S^r$ with the elements  $\varphi(a)\in S^{r+1}$ if no confusion arises. We sometimes write $S^*=(S^r,\varphi,\delta_{\pi,1},\ldots,\delta_{\pi,n})$. We denote by ${\bf Prol}_{\pi,\Phi}$ the full subcategory of 
${\bf Prol}^*_{\pi,\Phi}$ whose objects are the prolongation sequences $(S^r)$ such that all $S^r$'s are Noetherian and flat over $R_{\pi}$.
\begin{remark}\ 

1) If $S$ is a $p$-adically complete partial $\delta_{\pi}$-ring  whose $\pi$-derivations are compatible with those on $R_{\pi}$ then the sequence $S^*=(S^r)$ with $S^r=S$ has a natural structure of object of ${\bf Prol}^*_{\pi,\Phi}$ with $\varphi$ the identity and obvious $\delta_{\pi,j}$. If in addition $S$ is Noetherian and flat over $R_{\pi}$ then $S^*$ is an object of ${\bf Prol}_{\pi,\Phi}$.
The initial object in ${\bf Prol}^*_{\pi,\Phi}$ (and also of ${\bf Prol}_{\pi,\Phi}$) is the sequence $R^*_{\pi}=(R_{\pi}^r)$ with $R^r_{\pi}:=R_{\pi}$. 

2) If $S^*=(S^r,\varphi,\delta_{\pi,1},\ldots,\delta_{\pi,n})$ is an  object of ${\bf Prol}^*_{\pi,\Phi}$ then the ring
$$\lim_{\stackrel{\rightarrow}{\varphi}} S^r$$
has a natural structure of partial $\delta_{\pi}$-ring.
\end{remark}

\begin{remark} \label{dodonic}
For every $\pi' | \pi$ and 
every object $S^*$ in ${\bf Prol}_{\pi,\Phi}$ the sequence  $$S^* \otimes_{R_{\pi}} R_{\pi'} := (S^r \otimes_{R_{\pi}} R_{\pi'})_{r \geq 0}$$ is naturally an object of ${\bf Prol}_{\pi',\Phi}$;
cf. \cite[Sec. 4.1]{BM20}.
\end{remark}

\begin{remark}
For $\mu=i_1\ldots i_r\in \mathbb M^r_n\setminus \mathbb M^{r-1}_n$ we define the integral symbol:
$$w(\mu):=1+\phi_{i_1}+\phi_{i_1i_2}+\phi_{i_1i_2i_3}+\ldots+\phi_{i_1i_2i_3\ldots i_{r-1}}\in \mathbb Z_{\Phi}.$$
For every object $S^*=(S^r)$ in ${\bf Prol}^*_{\pi,\Phi}$, every $r\geq 1$, every $\mu\in \mathbb M^r_n \setminus \mathbb M_n^{r-1}$ and every $a\in S^0$ there exists $a_{\mu}\in S^{r-1}$ such that
\begin{equation}\label{remainderr}
\phi_{\pi,\mu} a=\pi^{w(\mu)} \delta_{\pi,\mu} a+ \varphi(a_{\mu});
\end{equation}
this is  trivially proved by induction on $r$.
\end{remark}

\begin{definition}\label{permu}
Consider two  families of distinct  Frobenius automorphisms 
$\Phi':=(\phi'_1,\ldots,\phi'_{n'})$ and 
$\Phi'':=(\phi''_1,\ldots,\phi''_{n''})$
of $K^{\textup{alg}}$. Also let $\pi\in \Pi$.  A map of sets
\begin{equation}
\label{facemap}
\epsilon:\{1,\ldots,n'\}\rightarrow \{1,\ldots,n''\}\end{equation}
is called a {\bf selection map} (with respect to $(\Phi',\Phi'',\pi)$) if
 for all $j\in \{1,\ldots, n'\}$ we have that
  $\phi_{\pi,j}=\phi_{\pi,\epsilon(j)}$.
Consider next an object of ${\bf Prol}^*_{p,\Phi''}$,
$$
S^*=(S^r,\varphi,\delta''_{\pi,1},\ldots,\delta''_{\pi,n}),
$$
and let $\epsilon$ be a selection map as above.
One defines the object $S^*_{\sigma}$ in ${\bf Prol}^*_{p,\Phi'}$ by:
$$S^*_{\sigma}:=(S^r,\varphi,\delta''_{\pi,\epsilon(1)},\ldots,\delta''_{\pi,\epsilon(n')}).
$$
This construction depends only on the restrictions $\Phi'_{\pi}$ and $\Phi''_{\pi}$
 of $\Phi'$ and $\Phi''$ to $K_{\pi}$. \end{definition}

Motivated by Proposition~\ref{P1}, introduce variables denoted by $\delta_{\pi,\mu}y_j$ for $\mu\in {\mathbb M}_n$, $\pi\in \Pi$, $j\in \{1,\ldots,N\}$. 
Fix integer $N$ and consider the 
the ring $R_{\pi}[y_1,\ldots,y_N]$ and  the rings 
\begin{equation}
\label{mottby}
J_{\pi,\Phi}^r(R_{\pi}[y_1,\ldots,y_N]):=R_{\pi}[\delta_{\pi,\mu}y_j\ |\ \mu\in {\mathbb M}_n^r, j\in \{1,\ldots,N\}]^{\widehat{\ }}.\end{equation}
 The sequence
$J^*_{\pi,\Phi}(R_{\pi}[y_1,\ldots,y_N]):=(J^r_{\pi,\Phi}(R_{\pi}[y_1,\ldots,y_N]))$ has a unique structure of object in ${\bf Prol}_{\pi,\Phi}$ such that $\delta_{\pi,i} \delta_{\pi,\mu}y:=\delta_{\pi, i\mu}y$ for all $i=1,\ldots,n$.  We have an induced
evaluation map $F_{R_{\pi}}:R_{\pi}^N\rightarrow R_{\pi}$: for $(a_1,\ldots,a_N)\in
R_{\pi}^N$ we let $F_{R_{\pi}}(a_1,\ldots,a_N)\in R_{\pi}$ be obtained from $F$
 by replacing the variables
$\delta_{\pi, \mu}y_j$  with the elements $\delta_{\pi, \mu}a_j$. Note that the map 
\begin{equation}
\label{e0}
J^r_{\pi,\Phi}(R_{\pi}[y_1,\ldots,y_N])\rightarrow \text{Fun}(R_{\pi}^N, R_{\pi}),\ F\mapsto F_{R_{\pi}}
\end{equation} is not injective in general, even if $\Phi$ is monomially independent.
 Here and later ``Fun" stands for the set of set theoretic maps. For instance if $\pi=p$ we have $(\delta_{p,i}y-\delta_{p,j}y)_R=0$. This is in stark contrast with \cite{BM20}. See, however,  Remark \ref{injinj}.

\begin{definition}
For every 
  $R_{\pi}$-algebra of finite type $A := R_{\pi}[y_1,\ldots,y_N]/I$, we define
$$J^r_{\pi,\Phi}(A):=J_{\pi,\Phi}^r(R_{\pi}[y_1,\ldots,y_N])/(\delta_{\pi,\mu}I\ |\ \mu\in {\mathbb M}_n^r).$$ This algebra is called the {\bf partial $\pi$-jet algebra} of $A$  of order $r$. \end{definition}

Note that $J^r_{\pi,\Phi}(A)$ is Noetherian and $p$-adically complete but  generally not flat over $R_{\pi}$, even if $\pi=p$ and $A$ is flat over $R_{\pi}$ as one can see by taking $A=R[x]/(x^p)$. 
It is trivial to see that the sequence $J^*_{\pi,\Phi}(A):=(J^r_{\pi,\Phi}(A))$ has a natural structure of prolongation sequence, i.e., it is an object of ${\bf Prol}^*_{\pi,\Phi}$ (but, as just noted,  it is not generally an object of ${\bf Prol}_{\pi,\Phi}$).
Also note that $J^r_{\pi,\Phi}(A)$ depends only on $r,\pi,A$ and on the restriction $\Phi_{\pi}$  of $\Phi$ to $R_{\pi}$.

\begin{proposition}\label{thm:etale}
If 
$A$ is a smooth $R_{\pi}$-algebra, and $u \colon R_{\pi}[T_1,\ldots,T_d] \to A$ is an \'etale morphism of $R_{\pi}$-algebras, then 
there is a (unique) isomorphism 
$$A[ \delta_{\pi,\mu}T_j\ |\ \mu\in {\mathbb M}_n^{r,+}, j\in \{1,\ldots,d\}]^{\widehat{\ }}\cong J^r_{\pi,\Phi}(A)$$
sending $\delta_{\pi,\mu} T_j$ into $\delta_{\pi,\mu}(u(T_j))$ for all $j$ and $\mu$.
In particular $J^r_{\pi,\Phi}(A)$ is flat over $R_{\pi}$ so the sequence $J^*_{\pi,\Phi}(A)$ is an object of ${\bf Prol}_{\pi,\Phi}$.
\end{proposition}

{\it Proof}. Similar to \cite{Bu05}, Proposition 3.13.

\bigskip

We have the following universal property.

\begin{proposition}\label{thm:univ}
Assume  $A$ is a finitely generated (respectively smooth) $R_{\pi}$-algebra. 
For every object $T^*$ of 
${\bf Prol}^*_{\pi,\Phi}$ 
(respectively in ${\bf Prol}_{\pi,\Phi}$)
and every $R_{\pi}$-algebra map $u:A\rightarrow T^0$ there is a unique morphism  $J^*_{\pi,\Phi}(A)\rightarrow T^*$ over $S^*$ in 
${\bf Prol}^*_{\pi,\Phi}$ (respectively in ${\bf Prol}_{\pi,\Phi}$) 
compatible with $u$. 
\end{proposition}

{\it Proof}.  Similar to \cite[Prop. 3.3]{Bu05}.\qed

\bigskip

We next record the existence of ``prolongations of derivations."
Let $S$ be a ring. Recall that by an {\bf $S$-derivation} from an $S$-algebra  $A$ to an $A$-algebra $B$ one understands an $S$-module endomorphism  $A\rightarrow B$ satisfying the Leibniz rule.

\begin{proposition}\label{thm:derivations}
Let $A$ be a smooth $R_{\pi}$-algebra equipped with an $R_{\pi}$-derivation $D:A\rightarrow A$. Then for every $r\geq 1$ and every $\mu\in \mathbb M^r_n$  there exists a unique $R_{\pi}$-derivation $D_{\mu}:J^r_{\pi,\Phi}(A)\rightarrow J^r_{\pi,\Phi}(A)$ satisfying the following properties:

1) $D_{\mu}\phi_{\mu} a=p^r\cdot \phi_{\mu} D a$ for all $a\in A$;

2) $D_{\mu} \phi_{\nu} a=0$ for all $a\in A$ and all $\nu\in \mathbb M^r_n\setminus \{\mu\}.$
\end{proposition}

{\it Proof}.
Similar to \cite[Prop. 3.43]{Bu05}. We recall the argument. Uniqueness is clear. To prove existence 
let $u:S:=R_{\pi}[T_1,\ldots,T_d]
\rightarrow A$ an \'{e}tale map and let $a_i:=D T_i\in A$. Then consider the derivation 
$$\frac{p^r}{\pi^{w(\mu)}} \sum_{i=1}^d a_i^{\phi_{\mu}} \frac{\partial}{\partial \delta_{\pi,\mu}T_i} : J^r_{\pi,\Phi}(S)=R_{\pi}[\delta_{\pi,\nu}T\ |\ \nu\in \mathbb M_n^r]
\rightarrow J^r_{\pi,\Phi}(A).$$
By Proposition \ref{thm:etale} this derivation extends to a derivation $D_{\mu}:J^r_{\pi,\Phi}(A)\rightarrow J^r_{\pi,\Phi}(A)$. To check properties 1) and 2) it is enough to check them for $a=T_i$ because if 1) and 2) hold for two elements of $J^r_{\pi,\Phi}(S)$ then 1) and 2) hold for their sum and their product. But for $a=T_i$ the equalities 1) and 2) hold in view of formula (\ref{remainderr}).
\qed

\bigskip

The jet construction can be globalized as follows.

\begin{definition}\label{APDE}
For every smooth
 scheme $X$  over $R_{\pi}$ 
define the $p$-adic formal scheme 
$$J^r_{\pi,\Phi}(X)=\bigcup Spf(J^r_{\pi,\Phi}(\mathcal O(U_i))),$$
called the {\bf partial $\pi$-jet space} of order $r$ of $X$, where 
$X=\cup U_i$ is (any) affine open cover. The gluing involved in this definition is 
well defined because the formation of $\pi$-jet spaces is compatible with fractions; cf. Proposition \ref{thm:etale}.
The elements of the ring $\mathcal O(J^r_{\pi,\Phi}(X))$, identified with morphisms of $p$-adic formal schemes $J^r_{\pi,\Phi}(X)\rightarrow \widehat{\mathbb A^1}$,
are called (purely) {\bf arithmetic PDEs} on $X$ over $R_{\pi}$ of order $\leq r$.
\end{definition}

For all $\pi'|\pi$  we write $X_{\pi'}:=X\otimes_{R_{\pi}} R_{\pi'}$.
Clearly $J^0_{\pi',\Phi}(X_{\pi'})=\widehat{X_{\pi'}}$. Note also that
$J^r_{\pi',\Phi}(X_{\pi'})$ only depends on $r,\pi',X$ and 
on the restriction $\Phi_{\pi}$ of $\Phi$ to $R_{\pi}$.

\begin{proposition} \label{inj}
Assume $A$ is a smooth  $R_{\pi}$-algebra.
For all $\pi''|\pi'|\pi$ there  are natural homomorphisms
\begin{equation}
\label{e1}
\iota_{\pi'',\pi'}:
J^r_{\pi'',\Phi}(A)\rightarrow J^r_{\pi',\Phi}(A) \otimes_{R_{\pi'}} R_{\pi''}\end{equation}
such that  the homomorphism
\begin{equation}
\label{e2}
\iota_{\pi'',\pi}:J^r_{\pi'',\Phi}(A)\rightarrow 
J^r_{\pi,\Phi}(A) \otimes_{R_{\pi}} R_{\pi''}
\end{equation}
equals the composition
\begin{equation}
\label{e3}
J^r_{\pi'',\Phi}(A)\stackrel{\iota_{\pi'',\pi'}}{\longrightarrow} J^r_{\pi',\Phi}(A) \otimes_{R_{\pi'}} R_{\pi''}\stackrel{\iota_{\pi',\pi}\otimes 1}{\longrightarrow} 
(J^r_{\pi,\Phi}(A) \otimes_{R_{\pi}} R_{\pi'})\otimes_{R_{\pi'}}R_{\pi''},
\end{equation}
where the targets of the maps (\ref{e2}) and (\ref{e3}) are naturally identified.
Moreover  the  homomorphisms \eqref{e1} are injective.
\end{proposition}

\begin{proof} This follows similar to \cite[Prop. 4.1(1)]{BM20} and \cite[Prop. 2.2]{BuSa11}. The map $\iota_{\pi'',\pi'}$ is guaranteed by Proposition~\ref{thm:univ} as $(J^r_{\pi',\Phi}(A) \otimes_{R_{\pi'}} R_{\pi''})$ is naturally an object of ${\bf Prol}_{\pi'',\Phi}$; cf Remark \ref{dodonic}. The factorization \eqref{e2} arises from naturality of base change. Finally, to address the injectivity of \eqref{e1}, pick an \'etale homomorphism $R_{\pi}[T_1,\ldots,T_d] \to A$. Both the source and target of \eqref{e1} then embed in the common ring $$K_{\pi''}[[\delta_{\pi'',\mu}T_j\ |\ \mu\in {\mathbb M}_n^r, j=1,\ldots,d]] \cong K_{\pi''}[[\delta_{\pi',\mu}T_j\ |\ \mu\in {\mathbb M}_n^r, j=1,\ldots,d]]$$ recovering the natural base change \eqref{e1} from which the injectivity is clear.
\end{proof}

\bigskip

\begin{remark}\label{iubi}
For every smooth algebra $A$ over $R_{\pi}$ and every selection map $\epsilon$ with respect to $(\Phi',\Phi'',\pi)$ we get (by the universality property of $J^r$)  a natural morphisms of prolongation sequences over $R_{\pi}$ with respect to $\Phi'$,
$J^*_{\pi,\Phi'}(A)\rightarrow J^*_{\pi,\Phi''}(A)_{\epsilon}$. Hence for every smooth scheme $X$ over $R_{\pi}$ we get morphisms
\begin{equation}
\label{projectionszz}
J^r_{\pi,\Phi''}(X)\rightarrow J^r_{\pi,\Phi'}(X).\end{equation}
We shall be interested later in $3$ special cases of this construction.

1) Assume $\pi=p$, $\Phi'=\Phi''$, and $\epsilon:\{1,\ldots,n\}\rightarrow \{1,\ldots,n\}$ is a bijection. Then the above construction defines an action of the symmetric group $\Sigma_n$ on $J^r_{\pi,\Phi}(X)$.

2) Assume $n'=s$, $n''=n$,  $\Phi'=(\phi'_1,\ldots,\phi'_s)$, $\Phi''=\Phi=(\phi_1,\ldots,\phi_n)$, 
$$\phi'_1=\phi_{i_1},\ldots,\phi'_s=\phi_{i_s},\ \ 1\leq i_1<i_2<\ldots i_s\leq n,\ \ \epsilon(1)=i_1,\ldots,\epsilon(s)=i_s.$$ Then we get a natural morphism (referred to as a {\bf face} morphism)
$$J^r_{\pi,\Phi}(X)=J^r_{\pi,\phi_1,\ldots,\phi_n}(X)\rightarrow J^r_{\pi,\phi_{i_1},\ldots,\phi_{i_s}}(X).$$

3) Assume $\pi=p$, $n'=n$, $n''=1$, $\Phi'=\Phi=(\phi_1,\ldots,\phi_n)$, $\Phi''=\{\phi\}$, and hence $\epsilon$ is the constant map. Then we get a natural morphism (referred to as the {\bf degeneration} morphism):
$$J^r_{\pi,\phi}(X)\rightarrow J^r_{\pi,\Phi}(X).$$

4) Assume $\pi=p$ and $\Phi=\{\phi_1,\ldots,\phi_n\}$. Then one trivially checks that for all $i\in \{1,\ldots,n\}$ the composition of the face and degeneration morphisms below is the identity:
$$\textup{id}:J^r_{p,\phi_i}(X)\rightarrow J^r_{p,\Phi}(X)\rightarrow
J^r_{p,\phi_i}(X).
$$
\end{remark}

\subsection{Total $\delta$-overconvergence}
The  notion of $\delta$-overconvergence was introduced in \cite{BuSa11} and exploited in \cite{BM20}, cf. \cite[Def. 2.5]{BM20}. 

\begin{definition} \label{overconvergence}
Assume $A$ is a smooth $R_{\pi}$-algebra.
An element  $f_{\pi}\in J^r_{\pi,\Phi}(A)$ is called {\bf totally $\delta$-overconvergent} if it has the following property: for all $\pi'|\pi$ there exists an integer $N\geq 0$ such that $p^N f_{\pi}\otimes 1$ is in the image of 
the map 
\begin{equation}
\label{e2t}\iota_{\pi',\pi}:J^r_{\pi',\Phi}(A)\rightarrow 
J^r_{\pi,\Phi}(A) \otimes_{R_{\pi}} R_{\pi'}.\end{equation} 
Denote by $J^r_{\pi,\Phi}(A)^!$ the $R$-algebra of all totally $\delta$-overconvergent elements in $J^r_{\pi,\Phi}(A)$. 
 For every smooth scheme $X/R_{\pi}$ 
an element (arithmetic PDE), $f\in \mathcal O(J^r_{\pi,\Phi}(X))$, will be called 
{\bf totally $\delta$-overconvergent} if for all  affine open set $U\subset X$ (equivalently for every affine open set of a given affine open cover of $X$) the image of $f$ in 
  the ring $\mathcal O(J^r_{\pi,\Phi}(U))=J^r_{\pi,\Phi}(\mathcal O(U))$ is totally $\delta$-overconvergent. We denote by $\mathcal O(J^r_{\pi,\Phi}(X))^!$ the ring of all totally $\delta$-overconvergent elements of 
$\mathcal O(J^r_{\pi,\Phi}(X))$. A morphism $J^r_{\pi,\Phi}(X)\rightarrow \widehat{\mathbb A^1}$ will be called  {\bf totally $\delta$-overconvergent} if the corresponding element in $\mathcal O(J^r_{\pi,\Phi}(X))$ is totally $\delta$-overconvergent.
   \end{definition}
   
   Note that, again,  the ring
$\mathcal O(J^r_{\pi,\Phi}(X))^!$ depends only on $r,\pi,X$ and on the restriction $\Phi_{\pi}$ of $\Phi$ to $R_{\pi}$.

Using Proposition \ref{thm:etale} one trivially checks the following two propositions.

\begin{proposition}
For every smooth scheme $X$ over $R_{\pi}$,  every $r\geq 0$,  and every  map as in (\ref{facemap})  the 
ring homomorphisms
$$\mathcal O(J^r_{\pi,\Phi'}(X))\rightarrow \mathcal O(J^r_{\pi,\Phi''}(X))$$
induced by the morphisms
(\ref{projectionszz})
induce  ring homomorphisms
$$\mathcal O(J^r_{\pi,\Phi'}(X))^!\rightarrow \mathcal O(J^r_{\pi,\Phi''}(X))^!.$$
\end{proposition}

We will usually view the above ring homomorphisms as inclusions.

\begin{proposition}\label{etalestructure}
Assume that $u:\widehat{X}\rightarrow \widehat{Y}$ is a morphism between the $p$-adic completions of two  smooth $R_{\pi}$-schemes and let
$f:J^r_{\pi,\Phi}(Y)\rightarrow \widehat{\mathbb A^1}$ be a totally $\delta$-overconvergent morphism.
Then the composition 
$$J^r_{\pi,\Phi}(X)\stackrel{J^r(u)}{\longrightarrow} J^r_{\pi,\Phi}(Y)\stackrel{f}{\rightarrow} \widehat{\mathbb A^1}$$
is totally $\delta$-overconvergent, where $J^r(u)$ is the morphism induced by $u$ via the universal property. 
\end{proposition}

Similar to \cite{BM20} we make the following definition.

\begin{definition}
 For every $f\in \mathcal O(J^r_{\pi,\Phi}(X))$  and every object $S^*=(S^r)$ in ${\bf Prol}_{\pi,\Phi}$ 
 the universal property of $\pi$-jet spaces yields a map of sets
  \begin{equation}
  \label{raduu0}
  f_{S^*}:X(S^0)\rightarrow S^r.\end{equation}
 On the other hand, if $f\in \mathcal O(J^r_{\pi,\Phi}(X))^!$ then  for every object $S^*=(S^r)$ in ${\bf Prol}_{\pi,\Phi}$ we can
 define the map of sets
  \begin{equation}
  \label{raduu1}
  f^{\textup{alg}}_{S^*}:X(S^0\otimes_{R_{\pi}}R^{\text{alg}})\rightarrow S^r\otimes_{R_{\pi}}K^{\text{alg}}\end{equation}
  as follows. We may assume $X=Spec\ A$ is affine because the construction below allows gluing in the obvious sense. Let $P\in X(S^0\otimes_{R_{\pi}}R^{\text{alg}})$. Choose $\pi'|\pi$ such that $P\in X(S^0\otimes_{R_{\pi}}R_{\pi'})$
  and choose $N\geq 1$ such that $p^{N}f\otimes 1 \in J^r_{\pi,\Phi}(A)\otimes_{R_{\pi}}R_{\pi'}$ is  the image of some (necessarily unique) element
  $f_{\pi',N}\in J^r_{\pi',\Phi}(A)$ via the map
  (\ref{e2t}). 
  View $P$ as a morphism $P:A\rightarrow S^0\otimes_{R_{\pi}} R_{\pi'}$. By the universal property of $\pi'$-jet spaces we have an induced morphism $J^r(P):J^r_{\pi',\Phi}(A)\rightarrow S^r\otimes_{R_{\pi}}R_{\pi'}$.
  Then we define
  $$f^{\textup{alg}}_{S^*}(P)=p^{-N}(J^r(P))(f_{\pi',N})\in S^r\otimes_{R_{\pi}} K_{\pi'}\subset S^r\otimes _{R_{\pi}}K^{\text{alg}}.$$
  The definition is independent of the choice of $\pi'$ and $N$ due to the injectivity part of Proposition \ref{inj}. On the other hand $f^{\textup{alg}}_{S^*}$ effectively depends on $\Phi$ (and not only on the restriction $\Phi_{\pi}$ on $K_{\pi}$). For $S^*=R^*_{\pi}$ 
  we write $f_{R_{\pi}}:=f_{R_{\pi}^*}$ and 
  \begin{equation}
  \label{raduu2}
  f^{\textup{alg}}:=f^{\textup{alg}}_{R_{\pi}}:=f^{\textup{alg}}_{R_{\pi}^*}:X(R^{\text{alg}})\rightarrow K^{\text{alg}}
  \end{equation}
\end{definition}

\begin{proposition}
Let $f\in \mathcal O(J^r_{\pi,\Phi}(X))$ and assume the map $f_{S^*}$ is the zero map for every object
$S^*$ in ${\bf Prol}_{\pi,\Phi}$ 
with the property that $S^r$ are integral domains and $\varphi:S^r\rightarrow S^{r+1}$ are injective.
Then $f=0$. In particular if $f\in \mathcal O(J^r_{\pi,\Phi}(X))^!$ and the map $f_{S^*}^{\textup{alg}}$ is the zero map for every object
$S^*$ in ${\bf Prol}_{\pi,\Phi}$ as above then $f=0$.
\end{proposition}

{\it Proof}.
Take $S^*=(S^r)$, $S^r:=\mathcal O(J^r_{\pi,\Phi}(U))$ for various affine open sets $U\subset X$; one gets that the image of $f$ in  $\mathcal O(J^r_{\pi,\Phi}(U))$ is $0$, hence $f=0$.
\qed

\begin{remark}\label{injinj}
Assume $\phi_1,\ldots,\phi_n$ are monomially independent in $\mathfrak G(K^{\textup{alg}}/\mathbb Q_p)$.
It would be interesting to know when/if the ring homomorphism
\begin{equation}
\label{raduu4}
\mathcal O(J^r_{\pi,\Phi}(X))^! \rightarrow \text{Fun}(X(R^{\text{alg}}),K^{\text{alg}}),\ \ \ f\mapsto f^{\textup{alg}}\end{equation}
is injective.  
For $n=1$ this is true by \cite{BM20}. See also  Proposition \ref{injonchar} and Proposition \ref{zaza} for  related results. Clearly, if we do not assume $\phi_1,\ldots,\phi_n$ are monomially independent  in $\mathfrak G(K^{\textup{alg}}/\mathbb Q_p)$ then
 (\ref{raduu4}) is not injective in general: to get an example take  $X$  the affine line, $n=2$, and $\phi_1=\phi_2$.\end{remark}

\section{Partial $\delta$-characters}

\subsection{Definition and  the additive group}
We start with the following PDE definition extending the ODE case in \cite{Bu95}.

\begin{definition}\label{deltacharr}
A {\bf partial $\delta_{\pi}$-character of order $\leq r$}  of a  commutative smooth group scheme $G/R_{\pi}$  is a group homomorphisms  $J^r_{\pi,\Phi}(G)\rightarrow \widehat{\mathbb G_a}$
in the category of $p$-adic formal schemes. (So in particular a  partial $\delta_{\pi}$-character
can be identified with an element of $\mathcal O(J^r_{\pi,\Phi}(G))$, i.e., with  an arithmetic PDE of order $\leq r$.)
We denote by
$\mathbf{X}^r_{\pi,\Phi}(G)=\text{Hom}(J^r_{\pi,\Phi}(G), \widehat{\mathbb G_a})$ the 
$R_{\pi}$-module of partial $\delta_{\pi}$-characters of $G$ of order $\leq r$ which we identify with a submodule of $\mathcal O(J^r_{\pi,\Phi}(G))$. For $\pi'|\pi$ we set $\mathbf{X}_{\pi',\Phi}(G):=\mathbf{X}_{\pi',\Phi}(G_{\pi'})$ and we call the elements of the latter {\bf partial $\delta_{\pi'}$-characters} of $G$.
For  $n=1$ partial $\delta_{\pi}$-characters will be referred to as {\bf  ODE $\delta_{\pi}$-characters}. An element of $\mathbf{X}^r_{\pi,\Phi}(G)$ will be said to have order  $r$
if it is not in the image of the canonical (injective) map  $\varphi:\mathbf{X}^{r-1}_{\pi,\Phi}(G)\rightarrow \mathbf{X}^r_{\pi,\Phi}(G)$ induced by $\varphi$. We also consider the naturally induced semilinear maps $\phi_i:\mathbf{X}^{r-1}_{\pi,\Phi}(G)\rightarrow \mathbf{X}^r_{\pi,\Phi}(G)$ induced by $\phi_i$.
\end{definition}

Consider the $R_{\pi}$-module $\mathbf{X}^r_{\pi,\Phi}(G)^!$ of totally $\delta$-overconvergent partial $\delta_{\pi}$-characters of $G$. 
So inside the ring $\mathcal O(J^r_{\pi,\Phi}(G))$ we have
$$\mathbf{X}^r_{\pi,\Phi}(G)^!:=\mathbf{X}^r_{\pi,\Phi}(G)\cap \mathcal O(J^r_{\pi,\Phi}(G))^!.$$
Note that if $\psi\in \mathbf{X}^r_{\pi,\Phi}(G)$ and if $p^{N}\psi\otimes 1\in \mathcal O(J^r_{\pi,\Phi}(G))\otimes_{R_{\pi}} R_{\pi'}$ is the image of some (necessarily unique) 
$\psi_{\pi'}\in 
\mathcal O(J^r_{\pi',\Phi}(G))$ then $\psi_{\pi'}\in \mathbf{X}^r_{\pi',\Phi}(G)$.
In particular we have:

\begin{lemma}
  The image of the natural map
$$\mathbf{X}^r_{\pi,\Phi}(G)^!\rightarrow 
\textup{Fun}(G(R^{\textup{alg}}),K^{\textup{alg}}),\ \ \ \psi\mapsto \psi^{\textup{alg}}$$
is contained in $\textup{Hom}_{\textup{gr}}(G(R^{\textup{alg}}),K^{\textup{alg}})$ where 
$\textup{Hom}_{\textup{gr}}$ 
is the $\textup{Hom}$ in the category of abstract groups.\end{lemma}

We also record the following obvious lemma.

\begin{lemma}
\label{satur}
If an element $\psi$ of $\mathbf{X}^r_{\pi,\Phi}(G)$ times a power of $p$ belongs to  
$\mathbf{X}^r_{\pi,\Phi}(G)^!$ then $\psi$ itself belongs to $\mathbf{X}^r_{\pi,\Phi}(G)^!$.
\end{lemma}

We have the following description of partial $\delta_{\pi}$-characters of the additive group $\mathbb G_a=Spec\ R_{\pi}[T]$. For $\mu=i_1\ldots i_s\in \mathbb M_n$ and $r\geq s$ recall that we write
$$\phi_{\pi,\mu}T:=\phi_{\pi,i_1}\ldots\phi_{\pi,i_s} T\in R_{\pi}[\delta_{\pi,\nu}T\ |\ \nu\in \mathbb M_n^r]^{\widehat{\ }}=\mathcal O(J^r_{\pi,\Phi}(\mathbb G_a)).$$
Consider the embedding
$$\mathcal O(J^r_{\pi,\Phi}(\mathbb G_a))\subset K_{\pi}[[\delta_{\pi,\nu}T\ |\ \nu\in \mathbb M_n^r]]=K_{\pi}[[\phi_{\pi,\nu}T\ |\ \nu\in \mathbb M_n^r]]$$
and consider the groups
$$K^r_{\pi,\Phi} T:=\sum_{\mu\in \mathbb M_n^r} K_{\pi}\phi_{\pi,\mu}T\subset K_{\pi}[[\phi_{\pi,\nu}T\ |\ \nu\in \mathbb M_n^r]],$$
$$R^r_{\pi,\Phi} T:=\sum_{\mu\in \mathbb M_n^r} R_{\pi}\phi_{\pi,\mu}T\subset K^r_{\pi,\Phi}T.$$
These groups are naturally isomorphic to the groups of symbols $K^r_{\pi,\Phi}$ and $R^r_{\pi,\Phi}$, respectively.

\begin{proposition}\label{asun}
The following equality holds,
\begin{equation}
\label{XGa}
\mathbf{X}^r_{\pi,\Phi}(\mathbb G_a)=(K_{\pi,\Phi}^r T)\cap (R_{\pi}[\delta_{\pi,\nu}T\ |\ \nu\in \mathbb M_n^r]),\end{equation}
where the intersection is taken inside the ring $K_{\pi}[[\phi_{\pi,\nu}T\ |\ \nu\in \mathbb M_n^r]]$.
In particular
$$\mathbf{X}^r_{\pi,\Phi}(\mathbb G_a)=\mathbf{X}^r_{\pi,\Phi}(\mathbb G_a)^!.$$
\end{proposition}

{\it Proof}.
The inclusion $\supset$ in Equation (\ref{XGa}) is clear. 
To check the inclusion $\subset$ note that every element in
$\mathbf{X}^r_{\pi,\Phi}(\mathbb G_a)$ defines an additive polynomial in the ring 
$K_{\pi}[[\phi_{\pi,\nu}T\ |\ \nu\in \mathbb M_n^r]]$ hence, since $K_{\pi}$ has characteristic $0$,  a linear polynomial.
\qed

\begin{lemma} \label{poa}
For $\psi:=\sum_{\mu\in \mathbb M_n^r} \lambda_{\mu}\phi_{\mu}T$ in the group in Equation (\ref{XGa}) the following hold:

1) $\lambda_0\in R_{\pi}$.

2) If $\psi\equiv 0$ mod $T$ in $K_{\pi}[\delta_{\pi,\nu}T\ |\ \nu\in \mathbb M_n^r]$ then
$\lambda_{\mu}=0$ for all $\mu\in \mathbb M_n^{r,+}$.

3) If $n=1$ then $\lambda_{\mu}\in R_{\pi}$ for all $\mu\in \mathbb M_n^r$. In other words
the intersection in the right hand side of Equation (\ref{XGa}) 
equals $R^r_{\pi,\Phi}T$.
\end{lemma}

{\it Proof}. Part 1 follows by picking out the coefficient of $T$.

Part 2 follows by induction on the number of non-zero terms in $\psi$. For the induction step one orders $\mathbb M_n^r$ by letting all members of $\mathbb M_n^s\setminus \mathbb M_n^{s-1}$ be greater than all members of $\mathbb M_n^{s-1}$ for all $s\in \{1,\ldots,r\}$ and by taking an arbitrary total order on each set $\mathbb M_n^s\setminus \mathbb M_n^{s-1}$. Then one picks out  the coefficient of $\delta_{\pi,\mu}T$ in $\psi$ where $\mu$ is the largest element
in $\mathbb M_n^r$ with $\delta_{\pi,\mu}T$ appearing in $\psi$. 

Part 3 follows again by induction on the number of non-zero terms in $\psi$. For the induction step
one picks out the coefficient of $T^{p^n}$ in $\psi$ where $n$ is the largest integer such that $T^{p^n}$ appears in $\psi$.
\begin{remark}
Assertion 3 in Lemma \ref{poa} fails for $n\geq 2$. For instance,  for $\pi=p$, 
one immediately checks that
$$\frac{1}{p}\phi_1\phi_2T-\frac{1}{p}\phi_2\phi_1T\in \mathbf{X}^2_{p,\Phi}(\mathbb G_a)\setminus (R^r_{\Phi}T).$$
\end{remark}

\subsection{Picard-Fuchs symbol}

\begin{definition}\label{PFsymbol}
Let $G$ have relative dimension $1$ over $R_{\pi}$ and let  $\omega$ be an invariant $1$-form on $G$.
By an {\bf admissible coordinate} for $G$ we mean  an \'{e}tale coordinate $T\in \mathcal O(U)$ on a neighborhood $U$ of the origin of $G$ generating the ideal of the origin of $G$ in $\mathcal O(U)$.   Let 
$$\ell(T)=\ell_{\omega}(T)=\sum_{m=1}^{\infty} \frac{b_m}{m}T^m\in K_{\pi}[[T]],\ \ b_m\in R_{\pi},$$
 be the logarithm of the formal group of $G$ (with respect to $T$)  normalized with respect to $\omega$; so $\ell$ is the unique series in $K_{\pi}[[T]]$ without constant term
such that $d\ell=\omega$ in $K_{\pi}[[T]]dT$. (We have $b_1\in R_{\pi}^{\times}$.) Let $e(T)=e_{\omega}(T)\in K_{\pi}[[T]]$ be the exponential normalized with respect to $\omega$, i.e., the compositional inverse of $\ell(T)$.
Then the series $e(pT)$ belongs to $pR_{\pi}[T]^{\widehat{\ }}$ and so defines a morphism of groups in the category of $p$-adic formal schemes, $\mathcal E:\widehat{\mathbb G_a}\rightarrow \widehat{G}$. For every partial $\delta_{\pi}$-character $\psi\in \mathbf{X}^r_{\pi,\Phi}(G)$ the composition
$$\theta(\psi):J^r_{\pi,\Phi}(\mathbb G_a)\stackrel{\mathcal E^r}{\longrightarrow}
J^r_{\pi,\Phi}(G)\stackrel{\psi}{\longrightarrow} \widehat{\mathbb G_a}$$
is a partial $\delta_{\pi}$-character of $\mathbb G_a$ so, identified with an element of $\mathcal O(J^r_{\pi,\Phi}(\mathbb G_a))$, can be written
(cf. Proposition \ref{asun} and Lemma \ref{poa}, Part 1)
 as
\begin{equation}
\label{sigmapsi}
\theta(\psi)=\sum_{\mu\in \mathbb M_n^r} \lambda_{\mu} \phi_{\pi,\mu}T\in 
\mathbf{X}^r_{\pi,\Phi}(\mathbb G_a)\subset 
K^r_{\pi,\Phi}T,\ \ \lambda_{\mu}\in K_{\pi},\ \ \lambda_0\in R_{\pi}.\end{equation}
We define the {\bf Picard-Fuchs symbol} (still denoted by $\theta(\psi)$)  of $\psi$ (with respect to $T$ and $\omega$)  by
$$\theta(\psi):=\sum_{\mu\in \mathbb M_n^r} 
\lambda_{\mu} \phi_{\mu}\in K^r_{\pi,\Phi}.$$
The latter induces a $\mathbb Q_p$-linear map
$$\theta(\psi)^{\textup{alg}}:K^{\textup{alg}}\rightarrow K^{\textup{alg}}.$$
 \end{definition}

\begin{remark}\label{3parts}\  

1) By our very definition, 
viewing $\psi$ as an element of $R_{\pi}[[\delta_{\pi,\mu}T\ |\ \mu\in \mathbb M_n^r]]$,
we have the following equality in $K_{\pi}[[\delta_{\pi,\mu}T\ |\ \mu\in \mathbb M_n^r]]$:
$$\psi=\frac{1}{p} \theta(\psi) \ell(T).$$

2) For every $\psi$, writing $\theta(\psi)=\sum_{\mu} \lambda_{\mu}\phi_{\mu}$ we have that 
$$\lambda_0\in pR_{\pi}.$$
 Indeed by the equality in Part 1 we have that 
$$\theta(\psi)\ell(T)\in p R_{\pi}[[\delta_{\pi,\mu}T\ |\ \mu\in \mathbb M_n^r]]$$
and we are done by picking out the coefficient of $T$.

3) The map 
\begin{equation}
\label{thetah}
\theta:\mathbf{X}^r_{\pi,\Phi}(G)\rightarrow K^r_{\pi,\Phi},\ \ \psi\mapsto \theta(\psi)\end{equation}
is a group homomorphism. Moreover for all $\mu$ we have
$$\theta(\phi_{\mu} \psi)=\phi_{\mu} \theta(\psi).$$

4) If  $\pi=p$ then the action of $\Sigma_n$ on $J^r_{p,\Phi}(G)$ induces an action of $\Sigma_n$ on $\mathbf{X}^r_{p,\Phi}(G)$. We also we have an obvious action of $\Sigma_n$ on 
$K^r_{\Phi}$ and the  homomorphism (\ref{thetah}) is $\Sigma_n$-equivariant.
\end{remark}

In what follows let $\mathfrak m=\mathfrak m(R^{\textup{alg}})$ be the maximal ideal of $R^{\textup{alg}}$,
let 
$$G(\mathfrak m):=\textup{Ker}(G(R^{\textup{alg}})\rightarrow G(k))$$
 and let $\ell^{\textup{alg}}:G(\mathfrak m^{\textup{alg}})\rightarrow K^{\textup{alg}}$ be the map induced by the logarithm series $\ell(T)$.

\begin{corollary}\label{mystery} Let $\psi\in \mathbf{X}_{\pi,\Phi}(G)^!$
be a  totally $\delta$-overconvergent $\delta_{\pi}$-character and consider 
the  homomorphism
$\psi^{\textup{alg}}: G(R^{\textup{alg}})\rightarrow K^{\textup{alg}}$.
 The following hold:

1)  The restriction $\psi^{\mathfrak m}$ of 
$\psi^{\textup{alg}}$ to $G(\mathfrak m)$ fits into a commutative diagram
\begin{equation}
\label{patrat}
\begin{array}{rcl}
G(\mathfrak m) & \stackrel{\ell^{\textup{alg}}}{\longrightarrow} & K^{\textup{alg}}\\
\psi^{\mathfrak m} \downarrow & \ & \downarrow \theta(\psi)^{\textup{alg}}\\
K^{\textup{alg}} & = & K^{\textup{alg}}
\end{array}
\end{equation}

2)  The homomorphism $\psi^{\textup{alg}}$ can be extended to a (necessarily unique) continuous homomorphism 
$\psi^{\mathbb C_p}: G(\mathbb C_p^{\circ})\rightarrow \mathbb C_p$.
\end{corollary}

{\it Proof}. Part 1 follows directly from Remark \ref{3parts}, Part 1. To check Part 2 note that 
since $\psi^{\textup{alg}}$ is a homomorphism is it enough to check it can be extended by continuity on a ball in $G(\mathbb C_p^{\circ})$ around the origin, cf., \cite[Sec. 4.2]{BM20}.
This follows directly, exactly as in \cite[proof of Prop. 6.8]{BM20}, from Part 1.
\qed

\bigskip

The following is a PDE version of the arithmetic ODE analogue (cf., \cite{Bu95, Bu97}) of Manin's Theorem of the Kernel \cite{Man63}.

\begin{corollary}\label{ToftheK}
For every $\psi\in  \mathbf{X}_{\pi,\Phi}(G)^!$ there is a natural group isomorphism
\begin{equation}
\label{bubu}
(\textup{Ker}(\psi^{\textup{alg}}))\otimes_{\mathbb Z} \mathbb Q\simeq 
\textup{Ker}(\theta(\psi)^{\textup{alg}}).\end{equation}
\end{corollary}

{\it Proof}.
The exact sequence
$$0\rightarrow G(\mathfrak m)\rightarrow G(R^{\textup{alg}}) \rightarrow G(k)$$
induces an exact sequence
$$0\rightarrow \textup{Ker}(\psi^{\mathfrak m}) \rightarrow \textup{Ker}(\psi^{\textup{alg}})
\rightarrow G(k).$$
Since the group $G(k)$ is torsion we get an induced group isomorphism
\begin{equation}
\label{vulp1}
(\textup{Ker}(\psi^{\mathfrak m}))\otimes_{\mathbb Z} \mathbb Q\stackrel{\sim}{\rightarrow}   (\textup{Ker}(\psi^{\textup{alg}}))\otimes_{\mathbb Z} \mathbb Q.\end{equation}
On the other hand 
recall that the map $\ell^{\textup{alg}}$ in 
 diagram (\ref{patrat}) has a torsion kernel and cokernel;  cf. \cite[Prop. 3.2 and Thm. 6.4]{Sil86}. So tensoring the diagram (\ref{patrat}) with $\mathbb Q$ the resulting upper horizontal map is an isomorphism. Taking the kernels of the resulting vertical maps we get an induced group isomorphism
 \begin{equation}
 \label{vulp2}
 (\textup{Ker}(\psi^{\mathfrak m}))\otimes_{\mathbb Z} \mathbb Q\stackrel{\sim}{\rightarrow}   \textup{Ker}(\theta(\psi)^{\textup{alg}}).
 \end{equation}
 We are done by considering the composition of the map (\ref{vulp2}) with the inverse of the map (\ref{vulp1}).
\qed

\bigskip

The following strengthened version the above corollary is sometimes useful. Let $L$ be a filtered union of complete 
subfields of $K^{\textup{alg}}$, let $\mathcal O$ be the valuation ring of $L$, and let $\mathfrak m(\mathcal O)$ be the maximal ideal of $\mathcal O$.
Assume $G$ comes via base change from a smooth group scheme $G_{\mathcal O}$ over $\mathcal O$ and write $G_{\mathcal O}(\mathcal O)=:G(\mathcal O)$.

\begin{corollary}
\label{strongger} For every $\psi\in  \mathbf{X}_{\pi,\Phi}(G)^!$
the isomorphism
in Corollary \ref{ToftheK} induces an isomorphism
\begin{equation}
\label{bbbb}
(\textup{Ker}(\psi^{\textup{alg}})\cap G(\mathcal O))\otimes_{\mathbb Z} \mathbb Q\simeq 
\textup{Ker}(\theta(\psi)^{\textup{alg}})\cap L.\end{equation}
In particular if $\textup{Ker}(\theta(\psi)^{\textup{alg}})\cap L=0$ then
$$\textup{Ker}(\psi^{\textup{alg}})\cap G(\mathcal O)= G(\mathcal O)_{\textup{tors}}.$$\end{corollary}

{\it Proof}.
It is enough to prove this for $L$ complete.
Let $x$ be in the left hand side of Equation (\ref{bbbb}), hence in the left hand side of Equation (\ref{bubu}). The image $x'$ of $x$ in the right hand side of (\ref{bubu}) is obtained as follows. One takes an integer $N\geq 1$ such that $Nx=P\otimes 1$ with $P\in \textup{Ker}(\psi^{\mathfrak m})$. Then $x'=\frac{1}{N}\ell^{\textup{alg}}(P)$. Since $L$ is complete and $\ell$ has coefficients in $\mathcal O$ we get that $\ell^{\textup{alg}}$ sends $G_{\mathcal O}(\mathfrak m(\mathcal O))$ into $L$ so we get that $x'\in L$ hence $x'$ is in the right hand side of (\ref{bbbb}). Conversely let $y'$ be in the right hand side of (\ref{bbbb}). The image $y$ of $y'$ in the left hand side of (\ref{bubu}) is obtained as follows. There exists an integer $N\geq 1$ such that $Ny=\ell^{\textup{alg}}(P)\otimes 1$ with $P\in G_{\mathcal O}(\mathfrak m(\mathcal O))$.
By diagram (\ref{patrat}) we have $P\in \textup{Ker}(\psi^{\textup{alg}})$. Then $y=P\otimes \frac{1}{N}$. So $y$ is in the left hand side of (\ref{bubu}).
\qed

\bigskip

We end by providing an easy dimension evaluation.
Define:
\begin{equation}
\label{Dnr}
D(n,r):=\# \mathbb M_n^r= 1+n+n^2+\ldots +n^r.\end{equation}
The following proposition is trivial to check.

\begin{proposition}\label{rankineq}
The map $\mathbf{X}^r_{\pi,\Phi}(G)\rightarrow K^r_{\pi,\Phi}$,
$\psi\mapsto \theta(\psi)$ is injective. In particular
\begin{equation}
\textup{rank}_{R_{\pi}}\mathbf{X}^r_{\pi,\Phi}(G)\leq D(n,r).\end{equation}
\end{proposition}

\subsection{Functions on points}

The next proposition shows that, in the case of monomially independent Frobenius automorphisms,   polynomial combinations of $\delta$-char\-acters are completely determined by their functions on points.

\begin{proposition}\label{injonchar}
Assume $\phi_1,\ldots,\phi_n$ are monomially independent 
 in $\mathfrak G(K^{\textup{alg}}/\mathbb Q_p)$
and denote by
$R_{\pi}[\mathbf{X}^r_{\pi,\Phi}(G)^!]$  the $R_{\pi}$-subalgebra of $\mathcal O(J^r_{\pi,\Phi}(G))$ generated by $\mathbf{X}^r_{\pi,\Phi}(G)^!$.
 Then the $R_{\pi}$-algebra map
$$R_{\pi}[\mathbf{X}^r_{\pi,\Phi}(G)^!]\rightarrow 
\textup{Fun}(G(R^{\textup{alg}}),K^{\textup{alg}}),\ \ 
f\mapsto f^{\textup{alg}}$$
is injective. In particular the $R_{\pi}$-module homomorphism
$$\mathbf{X}^r_{\pi,\Phi}(G)^!\rightarrow 
\textup{Hom}_{\textup{alg}}(G(R^{\textup{alg}}),K^{\textup{alg}}),\ \ 
\psi\mapsto \psi^{\textup{alg}}$$
is injective.
\end{proposition}

{\it Proof}. Let $\psi_1,\ldots,\psi_N\in \mathbf{X}^r_{\pi,\Phi}(G)^!$, let $F\in R_{\pi}[y_1,\ldots,y_N]$ be a polynomial in the variables $y_1,\ldots,y_N$, and let 
$$f=F(\psi_1,\ldots,\psi_N)\in \mathcal O(J^r_{\pi,\Phi}(G)).$$
Assume that the induced map
$f^{\textup{alg}}:G(R^{\textup{alg}})\rightarrow K^{\textup{alg}}$ is the zero map. 
Then the composition of $f^{\textup{alg}}$ with the induced map $\mathcal E^{\text{alg}}:\mathbb G_a(R^{\text{alg}})\rightarrow G(R^{\text{alg}})$ is the zero map. Write
$$\theta(\psi_i)=\sum_{\mu}\lambda_{i,\mu}\phi_{\mu},\ \ \ \lambda_{i,\mu}\in K_{\mu}.$$
Then for every $\lambda\in R^{\textup{alg}}$ we have
$$0=(f^{\textup{alg}}\circ \mathcal E^{\text{alg}})(\lambda)=F\left(\sum_{\mu}\lambda_{1,\mu}\phi_{\mu}(\lambda),\ldots,\sum_{\mu}\lambda_{N,\mu}\phi_{\mu}(\lambda)\right).$$
Let $x_{\mu}$ be variables indexed by $\mu\in \mathbb M_n^r$ and consider the polynomial
$$G(\ldots,x_{\mu},\ldots):=F\left(\sum_{\mu}\lambda_{1,\mu}x_{\mu},\ldots,\sum_{\mu}\lambda_{N,\mu}x_{\mu}\right)\in R_{\pi}[\ldots,x_{\mu},\ldots].$$
By Lemma \ref{preartin} we get $G=0$.
But clearly $f$ is obtained from $G$ by replacing $x_{\mu}\mapsto \frac{1}{p}\phi_{\mu}\ell(T)$. So $f=0$.
\qed

\section{Multiplicative group}

For $\pi\in \Pi$ define
 \begin{equation}
 \label{defNpi}
 N(\pi):=\min \{N\in \mathbb Z\ |\ \frac{\pi^n}{n}\in \frac{1}{N}\mathbb Z_p\ \ \text{for all}\ \ n\geq 1\}.\end{equation}
 In particular $N(p)=-1$.
 
Denote by $\mathbb G_m$  the multiplicative group scheme $\text{Spec}\ R_{\pi}[x,x^{-1}]$ over $R_{\pi}$. 
Note that for each $i\in \{1,\ldots,n\}$ 
the series
\begin{equation}
\label{lola}
p^{N(\pi)}\log \left(\frac{\phi_{\pi,i} x}{x^p}\right):=
p^{N(\pi)}\log \left(1+\pi\frac{\delta_{\pi,i} x}{x^p}\right)
\end{equation}
(where $\log$ is the usual logarithm series)
belongs to $R_{\pi}[x,x^{-1},\delta_{\pi}x]^{\widehat{\ }}$ and defines a totally $\delta$-overconvergent  $\delta_{\pi}$-character
 $\psi_i\in 
\mathbf{X}^1_{\pi,\Phi}(\mathbb G_m)^!$.  Clearly
the symbol of $\psi_i$ is given by
$$\theta(\psi_i)=p^{N(\pi)+1}(\phi_i-p).$$

\begin{theorem}\label{lili}
For all $r\geq 1$ we have
$$\mathbf{X}^r_{\pi,\Phi}(\mathbb G_m)^!=\mathbf{X}^r_{\pi,\Phi}(\mathbb G_m)$$ and a
basis modulo torsion for this $R_{\pi}$-module is given by 
\begin{equation}
\label{Gmbasis}
\{\phi_{\mu}\psi_i\ |\ i\in \{1,\ldots,n\}, \mu\in \mathbb M_n^{r-1}\}.\end{equation}
\end{theorem}

{\it Proof}.
By Proposition \ref{rankineq} the rank of $\mathbf{X}^r_{\pi,\Phi}(\mathbb G_m)$ is at most
$1+n+\cdots+n^r$. The symbols of the members of (\ref{Gmbasis}) are linearly independent
so the $n+n^2+\cdots+n^r$ elements of (\ref{Gmbasis}) are linearly independent. 
It is enough to check that $\mathbf{X}^r_{\pi,\Phi}(\mathbb G_m)$ does not have rank 
$1+n+\cdots +n^r$; indeed this will 
make (\ref{Gmbasis}) a basis modulo torsion of $\mathbf{X}^r_{\pi,\Phi}(\mathbb G_m)$ and this will also 
force $\mathbf{X}^r_{\pi,\Phi}(\mathbb G_m)^!$ and $\mathbf{X}^r_{\pi,\Phi}(\mathbb G_m)$ to have the same rank and hence to be equal 
by Lemma \ref{satur}.
 Now if $\mathbf{X}^r_{\pi,\Phi}(\mathbb G_m)$ has rank 
$1+n+\cdots +n^r$, by Proposition \ref{rankineq} the map $\mathbf{X}^r_{\pi,\Phi}(\mathbb G_m)\rightarrow K^r_{\pi,\Phi}$ tensored with $K_{\pi}$ must be an isomorphism.
So there must be an element $\psi\in  \mathbf{X}^r_{\pi,\Phi}(\mathbb G_m)$ with symbol
$\theta(\psi)=:c\in K_{\pi}$. Taking $T=x-1$ we immediately get that the logarithm 
$\ell_{\mathbb G_m}(T)=\log(1+T)$ of the formal group of $\mathbb G_m$ times a power of $p$ belongs to $R_{\pi}[[T]]$ which is a contradiction.
\qed

\begin{remark}
The moral of the above Theorem is that the PDE story in the case of $\mathbb G_m$ can be reduced to the ODE story via face  maps. As we shall see in the next section no such reduction is possible in the case of elliptic curves where `genuinely PDE' $\delta$-characters exist.
\end{remark}

The next corollary is a strengthening of \cite{BM20}, Proposition 3.5.

\begin{corollary}
 Let $\pi\in \Pi$ and for $i\in \{1,\ldots, n\}$ let $\psi_i$ be the $\delta_{\pi}$-character in Equation (\ref{lola}). Consider the induced homomorphism $\psi_i^{\textup{alg}}:\mathbb G_m(R^{\textup{alg}})\rightarrow K^{\textup{alg}}$. Then 
\begin{equation}
\label{oi}
\textup{Ker}(\phi_i^{\textup{alg}})=\mathbb G_m(R^{\textup{alg}})_{\textup{tors}}.
\end{equation}
\end{corollary}

{\it Proof}. This follows directly from Corollary \ref{ToftheK} in view of the fact that 
the map 
$$\theta(\psi_i)^{\textup{alg}}:\mathbb Q_p^{\textup{alg}}\rightarrow \mathbb Q_p^{\textup{alg}},\ \ \beta\mapsto \phi_i(\beta)-p\beta$$ is injective because $\#\phi_i(\beta)=\#\beta$). 
\qed

\section{Elliptic curves}

\subsection{General case}
Throughout this subsection $E$ is an elliptic curve (abelian scheme of dimension $1$) over $R_{\pi}$ and we fix an invertible $1$-form $\omega$ on $E$. 
By a {\bf formal group} over a ring we will always understand a formal group law (i.e, a tuple of elements in a formal power series ring).
For every family $\Phi:=(\phi_1,\ldots,\phi_n)$  of  Frobenius automorphisms of $K^{\text{alg}}$ define 
$$N^r_{\pi,\Phi}:=\textup{Ker}(J^r_{\pi,\Phi}(E)\rightarrow \widehat{E_{R_{\pi}}}).$$
Consider an admissible coordinate $T$ on $E$. Exactly as in \cite[Prop. 4.45]{Bu05} $N^r_{\pi,\Phi}$ is a group object in the category of $p$-adic formal
schemes over $R_{\pi}$ whose underlying $p$-adic formal scheme can be identified with the $p$-adic completion of the affine space
$$(\mathbb A_{R_{\pi}}^{n+\ldots+n^r})^{\widehat{\ }}=
\textup{Spf}\ R_{\pi}[\delta_{\pi,\mu} T\ |\ \mu\in \mathbb M_n^{r,+}]^{\widehat{\ }}$$
and whose group law is obtained as follows. One considers the formal group law $F(T_1,T_2)\in R_{\pi}[[T_1,T_2]]$ of $E$ with respect to $T$ and  one considers the group law on 
$N^r_{\pi,\Phi}$
defined by the  series
$F_{\mu}:=(\delta_{\pi,\mu}F)|_{T=0}$ for $\mu\in \mathbb M_n^r$ (which turn out to be restricted power series rather than just formal power series).

\begin{proposition}\label{dodo}
The $R_{\pi}$-module $\textup{Hom}(N^r_{\pi,\Phi},\widehat{\mathbb G_a})$  has rank
$$\textup{rank}_{R_{\pi}} \textup{Hom}(N^r_{\pi,\Phi},\widehat{\mathbb G_a})\leq D(n,r)-1= n+\ldots+n^r.$$
\end{proposition}

{\it Proof}.
Denote by $(N^r_{\pi,\Phi})^{\text{for}}$ the formal group over $R_{\pi}$ associated  to 
$N^r_{\pi,\Phi}$ and to the variables $\delta_{\pi,\mu} T$, $\mu\in \mathbb M_n^{r,+}$ and denote by $(N^r_{\pi,\Phi})^{\text{for}}_{K_{\pi}}$ the induced formal group law over $K_{\pi}$ which is  isomorphic to a power of the additive formal group law
$\mathbb G_{a/K_{\pi}}^{\text{for}}$ (since 
$(N^r_{\pi,\Phi})^{\text{for}}_{K_{\pi}}$ is
 commutative over a field of characteristic zero).
 We have natural injective maps of $K_{\pi}$-vector spaces
$$\textup{Hom}(N^r_{\pi,\Phi},\widehat{\mathbb G_a})\otimes_{R_{\pi}}K_{\pi}
\rightarrow \text{Hom}_{\text{for.gr.}}((N^r_{\pi,\Phi})^{\text{for}}_{K_{\pi}},\mathbb G_{a,K_{\pi}}^{\text{for}})\simeq K_{\pi}^{n+\ldots +n^r}$$
from which our proposition follows.
\qed

\bigskip

On the other hand we will construct, in what follows, a series of elements in $\textup{Hom}(N^r_{\pi,\Phi},\widehat{\mathbb G_a})$.
 Recall the logarithm series $\ell(T)=\ell_{\omega}(T)\in K_{\pi}[[T]]$ normalized with respect to $\omega$. 
 Recalling the integer $N(\pi)$ in Equation (\ref{defNpi}) we have that the series
  \begin{equation}
\label{Lmu}
L^{\mu}_{\pi,\Phi}:=(\phi_{\mu}(\ell(T)))|_{T=0}\in K_{\pi}[[\delta_{\pi,\nu}T\ |\ \nu\in \mathbb M_n^{r,+}]]
\end{equation}
satisfies
\begin{equation}
\tilde{L}^{\mu}_{\pi,\Phi}:=p^{N(\pi)}L^{\mu}_{\pi,\Phi} \in R_{\pi}[\delta_{\pi,\nu}T\ |\ \nu\in \mathbb M_n^{r,+}]^{\widehat{\ }}.
\end{equation}
It follows that
\begin{equation}
\label{thethebasis}
 \tilde{L}^{\mu}_{\pi,\Phi} \in 
\textup{Hom}(N^r_{\pi,\Phi},\widehat{\mathbb G_a}).\end{equation}

\begin{remark}\label{ovc}
Exactly as in \cite[Prop. 4.6]{BM20} we  get that for every $\pi' |\pi$ and 
$\mu\in \mathbb M_n^r$
the element $p^{N(\pi')-N(\pi)}\tilde{L}^{\mu}_{\pi,\Phi}\otimes 1$ is  the image of
$\tilde{L}^{\mu}_{\pi',\Phi}$ via 
 the homomorphism
$$R_{\pi'}[\delta_{\pi',\mu}T\ |\ \mu\in \mathbb M_n^{r,+}]^{\widehat{\ }}
\rightarrow R_{\pi}[\delta_{\pi,\nu}T\ |\ \nu\in \mathbb M_n^{r,+}]^{\widehat{\ }}\otimes_{R_{\pi}}R_{\pi'}.$$
\end{remark}

\begin{remark}\label{ocong}
Assume $\pi=p$. As in \cite[p.124]{Bu05} for all $i_1\ldots i_r\in \mathbb M^{r,+}_n$
we have
$$\phi_{i_1\ldots i_r}T-T^{p^r}-p(\delta_{i_r}T)^{p^{r-1}}\in (pT,p^2)\subset R[\delta_{\nu}T\ |\ \nu\in \mathbb M^r_n].$$ 
Hence, following \cite[Prop. 4.41]{Bu05} we get that
$$\tilde{L}^{i_1\ldots i_r}_{\pi,\Phi}\equiv (\delta_{i_r}T)^{p^{r-1}}\ \ \textup{mod}\ \ p\ \ \text{in}\ \ 
R[\delta_{p,\mu}T\ |\ \mu\in \mathbb M_n^{r,+}]^{\widehat{\ }}.
$$
\end{remark}

\begin{remark} \label{brontt}
Consider the following standard cohomology sequence (cf. \cite[p. 191]{Bu05} for the case $n=1$):
\begin{equation}
\label{cohseq}
0=\text{Hom}(\widehat{E},\widehat{\mathbb G_a})\rightarrow
\textup{Hom}(J^r_{\pi,\Phi}(E),\widehat{\mathbb G_a})\rightarrow
\textup{Hom}(N^r_{\pi,\Phi},\widehat{\mathbb G_a})\stackrel{\partial^r}{\rightarrow}
H^1(\widehat{E},\mathcal O)\end{equation}
and consider the isomorphism defined by Serre duality,
$$\langle -,\omega\rangle : H^1(\widehat{E},\mathcal O)\rightarrow R_{\pi}.$$
It is  useful to recall the explicit construction of the map $\partial^r$. By Proposition \ref{etalestructure}
there exists an affine open cover $E=\bigcup_i U_i$ and sections $s_i^r:\widehat{U_i}\rightarrow \text{pr}_r^{-1}(U_i)$ of the projection $\text{pr}_r:\text{pr}_r^{-1}(U_i)\rightarrow \widehat{U_i}$
induced by the projection $\text{pr}_r:J^r_{\pi,\Phi}(E)\rightarrow \widehat{E}$.
Then for all $L\in \textup{Hom}(N^r_{\pi,\Phi},\widehat{\mathbb G_a})$ the element
$\partial^r(L)$ is defined as the cohomology class 
 $$[L\circ (s_i^r-s_j^r)]\in H^1(\widehat{E},\mathcal O)$$ 
of the cocycle 
$$(L\circ (s_i^r-s_j^r))_{ij},\ \ L\circ (s_i^r-s_j^r):\widehat{U_i}\cap \widehat{U_j}\rightarrow N^r_{\pi,\Phi}\rightarrow \widehat{\mathbb G_a}.$$
The definition above is independent of the choice of sections $s_i^r$; such a change would change the cocycle by a coboundary.\end{remark}

Following \cite[p. 194]{Bu05} and recalling the elements in Equation \ref{thethebasis} we make the following:

\begin{definition} \label{primaryy}
For $\mu\in \mathbb M_n^{r,+}$ we define 
the {\bf primary arithmetic Kodaira-Spencer class} 
 of $E$ attached to  $\mu$  by the formula
$$f_{\mu}:=\langle \partial^r(\tilde{L}^{\mu}_{\pi,\Phi}),\omega\rangle \in R_{\pi}$$
and consider the vector:
$$KS^r_{\pi,\Phi}(E)=(f_{\mu})_{\mu\in \mathbb M_n^{r,+}}\in R_{\pi}^{\mathbb M_n^{r,+}}.
$$
\end{definition}

\begin{remark}\label{4parts}\ 

1) The elements $f_{\mu}$ are easily seen to depend only on the pair $(E,\omega)$ and not on the choice of $T$. This follows from the easily checked fact that our construction can be presented in a coordinate free manner: instead of the rings $R_{\pi}[[T]]$, $K_{\pi}[[T]]$ one may consider the completion $A$ of the local ring of $E$ at the closed point of the identity section and the corresponding completion $A_{K_{\pi}}$ for $E\otimes_{R_{\pi}}K_{\pi}$.
Instead of 
$$R_{\pi}[[\delta_{\pi,\mu}T\ |\ \mu\in \mathbb M_n^r]], \ K_{\pi}[[\delta_{\pi,\mu}T\ |\ \mu\in \mathbb M_n^r]], \ R_{\pi}[[T]][\delta_{\pi,\mu}T\ |\ \mu\in \mathbb M_n^{r,+}]^{\widehat{\ }},$$
one may consider
 certain  new types of ``$\pi$-jet algebras" $A^r,A^r_{K_{\pi}},\tilde{A}^r$, attached to $A$ respectively, satisfying certain new corresponding universal properties. We will not go here into defining  these new types of $\pi$-jet algebras. Then $\ell\in A_{K_{\pi}}$ is defined by the condition that it ``vanishes" at the ideal of the zero section and $d\ell=\omega$ in the completed module of K\"{a}hler differentials of $A_{K_{\pi}}$, $\phi_{\mu}\ell$ makes sense as an element of $A^r_{K_{\pi}}$ while $\phi_{\mu}\ell_{|T=0}$ makes sense as  an element of $\tilde{A}^r$.

 2) We will write $f_{\mu}(E,\omega)$ instead of $f_{\mu}$ if we want to emphasize the dependence on $(E,\omega)$. With notation as in Remark \ref{integral symbols}, we have
$$f_{\mu}(E,\lambda \omega)=\lambda^{\phi_{\mu}+1}f_{\mu}(E,\omega)$$
for all $\lambda\in R_{\pi}^{\times}$; this follows from the fact that if one replaces   $\omega$ by $\lambda \omega$ in our construction then, since $\ell_{\lambda \omega}(T)=\lambda \ell_{\omega}(T)$, we have that $L_{\pi,\Phi}^{\mu}$ gets replaced by
$\phi_{\mu}(\lambda) L_{\pi,\Phi}^{\mu}$.

3)
It is easy to see that the elements $f_{\mu}$ do not change if $r$ changes, as long as
$\mu\in \mathbb M_n^{r,+}$; 
this follows from the fact that changing $r$  amounts to changing the defining 
cocycle in our construction by a coboundary which does not change the cohomology class. 
This justifies not including $r$ in our notation for $f_{\mu}$.
In particular if $KS^r_{\pi,\Phi}(E)\neq 0$ for some $r\geq 1$ then $KS^{r'}_{\pi,\Phi}(E)\neq 0$
for all $r'\geq r$.

4)
It is easy to check that the formation of the elements $f_{\mu}$ is compatible with face  maps  in the sense that for every $\mu'\in \mathbb M_{n'}^{r,+}$ we have, in the above notation:
\begin{equation}\label{compepsi}
f_{\mu'}=f_{\epsilon(\mu')}.\end{equation}
On the other hand note that in general 
$$f_{i\mu}\neq \phi_i f_{\mu}.$$

5) For every isogeny $u:E'\rightarrow E$ of degree $d$ prime to $p$ of elliptic curves over $R_{\pi}$
and every invertible $1$-form $\omega$ on $E$, setting $\omega'=u^*\omega$, we have
$$f_{\mu}(E',\omega') = d\cdot f_{\mu}(E,\omega).$$
The argument is entirely similar  to the one in \cite[p. 264]{Bu05}.

6) Let us write $f_{\pi,\mu}$ instead of $f_{\mu}$ if we want to emphasize dependence on $\pi$. Then for all $\pi'|\pi$ we clearly have
$$f_{\pi',\mu}(E_{\pi'},\omega_{\pi'})=p^{N(\pi')-N(\pi)}f_{\pi,\mu}(E,\omega)\in R_{\pi'}$$
where  $E_{\pi'}:=E\otimes_{R_{\pi}}R_{\pi'}$ and $\omega_{\pi'}$ is the induced form.
\end{remark}

A special role will be played later by the the primary arithmetic Kodaira-Spencer classes
$f_i, f_{ii}, f_{iii}, \ldots$ We will write
$$f_{i^r}:=f_{i\ldots i}\ \ \text{with $i$ repeated $r$ times}.$$
These classes  are the images, via the corresponding face  maps, of  the corresponding classes obtained by replacing  $\Phi$ by $\{\phi_i\}$ in all our constructions. Note that in \cite{Bar03} and \cite{Bu05} the forms $f_{i^r}$ were denoted by $f^r$.

\begin{proposition}
Assume $E$ over $R_{\pi}$ has ordinary reduction and assume $f_i=0$ for some $i$. Then for all $\mu\in \mathbb M_n$ we have $f_{\mu}=0$.
\end{proposition}

This proposition cannot be proved at this point in the paper but is an immediate consequence of Theorem \ref{chara} that will be stated and proved later. 

\begin{lemma}\label{fifjx}
Assume $\pi=p$. Then for all $\mu\in \mathbb M^r_n$ and all $\sigma\in \Sigma_n$ we have
$$f_{\sigma \mu}=f_{\mu}.$$ 
\end{lemma}

{\it Proof}.
Let $s_i^r:\widehat{U}_i\rightarrow J^r_{p,\Phi}(\textup{pr}_r^{-1}(U_i))$ be local sections of the projection 
$\textup{pr}_r:J^r_{p,\Phi}(E)\rightarrow \widehat{E}$ as in Remark \ref{brontt} and consider the group automorphism over $E$,
$$\sigma:J^r_{p,\Phi}(E)\rightarrow J^r_{p,\Phi}(E)$$
induced by $\sigma\in \Sigma_n$.
Consider the local sections $\sigma \circ s^r_i:\widehat{U}_i\rightarrow J^r_{p,\Phi}(\textup{pr}^{-1}_r(U_i))$.
By the independence of $f_{\mu}$ on the choice of local sections we get
$$\begin{array}{rcl}
f_{\mu} & = & \langle [\tilde{L}^{\mu}_{\pi,\Phi}\circ (s_i^r-s_j^r)],\omega\rangle\\
\ & \ & \ \\
\ & = & \langle [\tilde{L}^{\mu}_{\pi,\Phi}\circ (\sigma \circ s_i^r-\sigma \circ s_j^r),\omega\rangle\\
\ & \ & \ \\
\ & = & \langle [\tilde{L}^{\mu}_{\pi,\Phi}\circ \sigma\circ(s_i^r- s_j^r),\omega\rangle\\
\ & \ & \ \\
\ & = &  \langle [\tilde{L}^{\sigma \mu}_{\pi,\Phi}\circ (s_i^r- s_j^r),\omega\rangle\\
\ & \ & \ \\
\ & = & f_{\sigma\mu}.
\end{array}
$$
\qed

\begin{remark}\label{fifj}
Assume $\pi=p$ and fix an index $i$. By its very construction $f_i=0$ if and only if $E$ 
possesses a Frobenius lift (i.e., a scheme endomorphism reducing modulo $p$ to the $p$-power Frobenius). 
Recall that if $E$ has ordinary reduction then $E$ has a Frobenius lift if and only if $E$ is a canonical lift of its reduction \cite[Appendix]{Me72}. On the other hand 
recall 
from \cite[Cor. 8.89]{Bu05} that if $E$ has supersingular reduction then 
$f_i\neq 0$. We conclude that for an arbitrary $E$ over $R=R_p$ we have $f_i=0$ if and only if $E$ has ordinary reduction and is a canonical lift of its reduction.
\end{remark} 

Consider the $K_{\pi}$-linear space
$$K^{r,+}_{\pi,\Phi}=\{\sum_{\mu\in \mathbb M_n^{r,+}} \lambda_{\mu} \phi_{\mu}\ |\lambda_{\mu}\in K_{\pi}\}\subset K^r_{\pi,\Phi}$$
and the projection
$$\rho:K^r_{\pi,\Phi}\rightarrow K^{r,+}_{\pi,\Phi},\ \ \rho\left(\sum_{\mu\in\mathbb M_n^r} \lambda_{\mu} \phi_{\mu}\right)=\sum_{\mu\in\mathbb M_n^{r,+}} \lambda_{\mu} \phi_{\mu}.$$

We may consider the $K_{\pi}$-linear  space of relations among the primary arithmetic Kodaira-Spencer classes:
\begin{equation}
KS^r_{\pi,\Phi}(E)^{\perp}:=\{\sum_{\mu\in\mathbb M_n^{r,+}} \lambda_{\mu} \phi_{\mu}\in
K^{r,+}_{\pi,\Phi}\ |\ 
 \ \sum_{\mu\in\mathbb M_n^{r,+}}\lambda_{\mu}f_{\mu}=0\}
\end{equation}
and its $R_{\pi}$-submodule of ``integral elements,"
$$KS^r_{\pi,\Phi}(E)^{\perp}_{\textup{int}}:=\{\sum_{\mu\in\mathbb M_n^{r,+}} \lambda_{\mu} \phi_{\mu}\in KS^r_{\pi,\Phi}(E)^{\perp}\ |\ \lambda_{\mu}\in R_{\pi}\}.$$
 Finally recall the {\bf symbol homomorphism} 
$$\theta:\mathbf{X}^r_{\pi,\Phi}(E)\rightarrow K^r_{\pi,\Phi},\ \ \psi\mapsto \theta(\psi).$$

\begin{theorem}\label{mainthm} The following claims hold.

1)
There exists an $R_{\pi}$-module homomorphism $P$ as in Equation (\ref{compp}) below
such that the composition
\begin{equation}\label{compp}
KS^r_{\pi,\Phi}(E)^{\perp}_{\textup{int}} \stackrel{P}{\longrightarrow} \mathbf{X}^r_{\pi,\Phi}(E)^!  \subset   \mathbf{X}^r_{\pi,\Phi}(E) \stackrel{\theta}{\longrightarrow}  K^r_{\pi,\Phi}
\stackrel{\rho}{\longrightarrow} K^{r,+}_{\pi,\Phi}
\end{equation}
is the multiplication by $p^{N(\pi)+1}$ map.   So for $\pi=p$ the composition (\ref{compp}) is the inclusion $KS^r_{\pi,\Phi}(E)^{\perp}_{\textup{int}}\subset  K^{r,+}_{\pi,\Phi}$.

2) The map $\rho\circ \theta$ is injective. In particular if $\theta(\psi)\in K_{\pi}$ for some $\psi\in \mathbf{X}^r_{\pi,\Phi}(E)$ then $\psi=0$ and hence $\theta(\psi)=0$.

\end{theorem}

\bigskip

{\it Proof}. To prove Part 1
note that if 
$$\Lambda:=\sum_{\mu\in\mathbb M_n^{r,+}} \lambda_{\mu} \phi_{\mu}\in KS^r_{\pi,\Phi}(E)^{\perp}_{\textup{int}}$$
and if 
$$L_{\Lambda}:=\sum_{\mu\in\mathbb M_n^{r,+}} 
\lambda_{\mu} \tilde{L}^{\mu}_{\pi,\Phi}\in
 \textup{Hom}(N^r_{\pi,\Phi},\widehat{\mathbb G_a})$$
then $\partial(L_{\Lambda})=0$ so $L_{\Lambda}$ is the restriction of a unique element
\begin{equation}
P(\Lambda)\in \textup{Hom}(J^r_{\pi,\Phi}(E),\widehat{\mathbb G_a})=\mathbf{X}^r_{\pi,\Phi}(E).\end{equation}
Clearly $\Lambda\mapsto P(\Lambda)$ is an $R_{\pi}$-linear map.
By an argument entirely similar to the one in the proof of \cite[Thm. 6.1]{BM20} and using Remark \ref{ovc} above
it follows that $P(\Lambda)$ is totally $\delta$-overconvergent: 
$P(\Lambda)\in \mathbf{X}^r_{\pi,\Phi}(E)^!$.
 By an argument entirely similar to the 
one  in the proof of \cite[Prop. 7.20]{Bu05} one gets that 
$$\theta(P(\Lambda))T\equiv p^{N(\pi)+1} \left(\sum_{\mu\in \mathbb M_n^{r,+}} \lambda_{\mu}\phi_{\mu}T\right) \ \ 
\text{mod}\ \ T$$
in the ring $K_{\pi}[\delta_{\pi,\mu}T\ |\ \mu\in \mathbb M_n^r]$. By Lemma \ref{poa}, Part 2, we have
$$\theta(P(\Lambda))T=p^{N(\pi)+1}\left(\sum_{\mu\in \mathbb M_n^{r,+}} \lambda_{\mu}\phi_{\mu}T\right) +\lambda_0T
$$ 
for some $\lambda_0\in R_{\pi}$.
Hence
$$\rho(\theta(P(\Lambda)))=p^{N(\pi)+1}\left(\sum_{\mu\in \mathbb M_n^{r,+}} \lambda_{\mu}\phi_{\mu}
\right)$$ 
and Part 1 follows. 

Part 2 follows from the observation that if $\theta(\psi)\in K_{\pi}$ for some $\psi\in \mathbf{X}^r_{\pi,\Phi}(E)$ then by Remark \ref{3parts}, Part 1 it easily follows that 
$$\psi\in \mathcal O(J^2_{\pi,\Phi}(E))\cap K_{\pi}[[T]]=\mathcal O(\widehat{E})$$
 and hence $\psi$ defines a homomorphism $\widehat{E}\rightarrow \widehat{\mathbb G_a}$; but the only such homomorphism is the zero homomorphism.
\qed

\begin{remark}
Note that $P$ in Theorem \ref{mainthm} is automatically injective. For all $\Lambda\in KS^r_{\pi,\Phi}(E)^{\perp}_{\textup{int}}$ we write $\psi_{\Lambda}:=P(\Lambda)$;
hence, by Remark \ref{3parts}, Part 2, we have
$$\theta(\psi_{\Lambda})=p^{N(\pi)+1} \Lambda+\lambda_0(\Lambda)$$
for some $\lambda_0(\Lambda)\in pR_{\pi}$. Clearly the map 
$$KS^r_{\pi,\Phi}(E)^{\perp}_{\textup{int}}\rightarrow pR_{\pi},\ 
\ \Lambda\mapsto \lambda_0(\Lambda)$$
 is an $R_{\pi}$-module homomorphism. 
\end{remark}

\begin{remark}\label{onemorecomp}
The map $P$ in Theorem \ref{mainthm} is compatible with the face  maps (\ref{facemap})
in an obvious sense.
\end{remark}

\begin{remark}
If $f_i=0$ for some $i$ then $\phi_i\in KS^r_{\pi,\Phi}(E)^{\perp}_{\textup{int}}$ hence 
$$\psi_i:=\psi_{\phi_i}\in \mathbf{X}^1_{\pi,\Phi}(E)^!.$$
Moreover the symbol of $\psi_i$ is given by
$$\theta(\psi_i)=p^{N(\pi)+1} \phi_i+\lambda_0(\phi_i).$$
\end{remark}

\begin{corollary}\label{C1}
 The  following claims hold.

1) If $KS^r_{\pi,\Phi}(E)\neq 0$ (in particular if $KS^1_{\pi,\Phi}(E)\neq 0$) then $\mathbf{X}^r_{\pi,\Phi}(E)=\mathbf{X}^r_{\pi,\Phi}(E)^!$
and the rank of this $R_{\pi}$-module equals $D(n,r)-2$.

2) If $KS^1_{\pi,\Phi}(E)=0$ (equivalently, if $f_i=0$ for all $i\in \{1,\ldots,n\}$) then 
$\mathbf{X}^r_{\pi,\Phi}(E)=\mathbf{X}^r_{\pi,\Phi}(E)^!$,
the rank of this $R_{\pi}$-module equals $D(n,r)-1$, and a basis modulo torsion for this $R_{\pi}$-module is given by 
\begin{equation}
\label{basss}
\{\phi_{\mu}\psi_i\ |\ \mu\in \mathbb M_n^{r,+},\ i\in \{1,\ldots,n\}\}.\end{equation}

3) The cokernel of the injective homomorphism $P$ in Theorem \ref{mainthm} 
 is a torsion $R_{\pi}$-module.
\end{corollary}

{\it Proof}.
If
 $KS^r_{\pi,\Phi}(E)\neq 0$ then the module $KS^r_{\pi,\Phi}(E)^{\perp}_{\textup{int}}$ has rank
$D(n,r)-2$. Since $P$ in Theorem \ref{mainthm} is injective the module
$\mathbf{X}^r_{\pi,\Phi}(E)^!$ has rank at least $D(n,r)-2$. On the other hand
by Proposition \ref{dodo} and by the exact sequence (\ref{cohseq}) the module 
$\mathbf{X}^r_{\pi,\Phi}(E)$ has rank at most $D(n,r)-2$.
So the modules $\mathbf{X}^r_{\pi,\Phi}(E)$ and $\mathbf{X}^r_{\pi,\Phi}(E)^!$
have the same rank $D(n,r)-2$ and hence they are equal
by Lemma \ref{satur}. This proves Part 1. 

Assume now $KS^1_{\pi,\Phi}(E)= 0$.  The 
subset  (\ref{basss}) of  $\mathbf{X}^r_{\pi,\Phi}(E)^!$
is linearly independent (because so is the set of symbols of its elements). It follows 
that $\mathbf{X}^r_{\pi,\Phi}(E)^!$ has rank at least 
$D(n,r)-1$.  On the other hand
by Proposition \ref{dodo}  and the sequence (\ref{cohseq}) the module 
$\mathbf{X}^r_{\pi,\Phi}(E)$ has rank at most $D(n,r)-1$.
So the modules $\mathbf{X}^r_{\pi,\Phi}(E)$ and $\mathbf{X}^r_{\pi,\Phi}(E)^!$
have the same rank $D(n,r)-1$ and hence they are equal
by Lemma \ref{satur}, with basis modulo torsion given  by (\ref{basss}). This proves Part 2.

 Part 3 follows from the fact that the source and target of $P$ have the same rank.
  \qed

\bigskip

As an application to Theorem \ref{mainthm} we construct a series of special $\delta_{\pi}$-characters of $E$ as follows. 
Let $\mu,\nu \in \mathbb M_n^{r,+}$ be distinct and let $\omega$ be an invertible $1$-form on $E$.  Recalling the integers $N(\pi)$ in Equation (\ref{defNpi}) set
\begin{equation}\label{tildef}
\tilde{f}_{\mu}:=p^{N(\pi)+1}f_{\mu},\ \ \mu\in \mathbb M_n.
\end{equation}
In particular if $\pi=p$ then $\tilde{f}_{\mu}=f_{\mu}$.

Note that 
$$f_{\nu}  \phi_{\mu} -f_{\mu} \phi_{\nu} \in KS^r_{\pi,\Phi}(E)^{\perp}_{\textup{int}}$$
 so we may consider the partial $\delta_{\pi}$-character 
\begin{equation}
\psi_{\mu,\nu}:=\psi_{f_{\nu}  \phi_{\mu} -f_{\mu} \phi_{\nu}}\in \mathbf{X}^r_{\pi,\Phi}(E)^!.\end{equation}
By Theorem \ref{mainthm} we have
\begin{equation}
\label{missing1}
\theta(\psi_{\mu,\nu})=
\tilde{f}_{\nu}  \phi_{\mu} -\tilde{f}_{\mu} \phi_{\nu}
+f_{\mu,\nu}\end{equation}
for some $f_{\mu,\nu}\in pR_{\pi}$. 

\begin{definition}\label{secondaryy}
The element $f_{\mu,\nu}\in pR_{\pi}$ is called the {\bf secondary arithmetic Kodaira-Spencer class} attached to $\mu$ and $\nu$.\end{definition}

\begin{remark}\label{weighttt}
Note that  $\psi_{\mu,\nu}$ and $f_{\mu,\nu}$ do not change 
if $r$ changes which justifies  $r$ not being  included in the notation. 
Note also that $\psi_{\mu,\nu}$ and $f_{\mu,\nu}$ effectively depend on ($E$ and) $\omega$ 
and if
we want to emphasize this dependence we denote them by 
$\psi_{\mu,\nu}(E,\omega)$ and $f_{\mu,\nu}(E,\omega)$, respectively.
Similarly we write $\tilde{f}_{\mu}(E,\omega)$ in place of $\tilde{f}_{\mu}$.
Then
for all $\lambda\in R_{\pi}^{\times}$ we have (using the notation in Remark \ref{integral symbols}):
$$f_{\mu,\nu}(E,\lambda\omega)=\lambda^{\phi_{\mu}+\phi_{\nu}} f_{\mu,\nu}(E,\omega).$$
Indeed, 
by Remark \ref{3parts}, Part 1, and Remark \ref{4parts}, Part 1, we have the following equalities
$$
\begin{array}{rcl}
\psi_{\mu,\nu}(E,\omega) & = & \frac{1}{p}(\tilde{f}_{\nu}(E,\omega)\phi_{\mu}-
\tilde{f}_{\mu}(E,\omega)\phi_{\nu}+f_{\mu,\nu}(E,\omega))\ell_{\omega}(T),\\
\ & \ & \ \\
\psi_{\mu,\nu}(E,\lambda\omega) & = & \frac{1}{p}(\tilde{f}_{\nu}(E,\lambda\omega)\phi_{\mu}-
\tilde{f}_{\mu}(E,\lambda\omega)\phi_{\nu}+f_{\mu,\nu}(E,\lambda\omega))\ell_{\lambda\omega}(T)\\
\ & \ & \ \\
\ & = & \frac{1}{p}(\lambda^{\phi_{\nu}+1} \tilde{f}_{\nu}(E,\omega)\phi_{\mu}-
\lambda^{\phi_{\mu}+1} \tilde{f}_{\mu}(E,\omega)\phi_{\nu}+f_{\mu,\nu}(E,\lambda\omega))(\lambda \ell_{\omega}(T))\\
\ & \ & \ \\
\ & = & \frac{1}{p}(\lambda^{\phi_{\nu}+\phi_{\mu}+1} \tilde{f}_{\nu}(E,\omega)\phi_{\mu}-
\lambda^{\phi_{\mu}+\phi_{\nu}+1} \tilde{f}_{\mu}(E,\omega)\phi_{\nu}+\lambda f_{\mu,\nu}(E,\lambda\omega))\ell_{\omega}(T).
\end{array}
$$
We get
$$\begin{array}{rcl}
\psi^* & := & \lambda^{\phi_{\mu}+\phi_{\nu}+1}\psi_{\mu,\nu}(E,\omega)-\psi_{\mu,\nu}(E,\lambda\omega)\\
\ & \ & \ \\
\ & = & \frac{1}{p}(\lambda^{\phi_{\mu}+\phi_{\nu}+1}f_{\mu,\nu}(E,\omega)-
\lambda f_{\mu,\nu}(E,\lambda\omega))\ell_{\omega}(T).\end{array}
$$
Hence 
$$\theta(\psi^*)=\frac{1}{p}(\lambda^{\phi_{\mu}+\phi_{\nu}+1}f_{\mu,\nu}(E,\omega)-
\lambda f_{\mu,\nu}(E,\lambda\omega))\in K_{\pi}.$$
By Theorem  \ref{mainthm}, Part 2, $\theta(\psi^*)=0$ which ends the proof.
\end{remark}

\begin{remark}
For all distinct $\mu,\nu$ we have
\begin{equation}\label{antisymmetry}
f_{\mu,\nu}+f_{\nu,\mu}=0.
\end{equation}
Indeed, switching $\mu$ and $\nu$ in Equation \ref{missing1} we get
\begin{equation}
\label{missing2}
\theta(\psi_{\nu,\mu})=
\tilde{f}_{\mu}  \phi_{\nu} -\tilde{f}_{\nu} \phi_{\mu}
+f_{\nu,\mu}.\end{equation}
Adding Equations \ref{missing1} and \ref{missing2} we may conclude by Theorem \ref{mainthm}, Part 2.
\end{remark}

\begin{remark}
Fix in what follows the elliptic curve $E$ over $R_{\pi}$ and an invertible $1$-form $\omega$. Write, as before, 
$\ell(T)=\ell_{\omega}(T)=\sum_{m=1}^{\infty}\frac{b_m}{m}T^m$, $b_m\in R_{\pi}$.
Let $\mu,\nu\in \mathbb M_n$ be distinct of lengths $r\geq s$, respectively.
By Remark \ref{3parts}, Part 1, we have that
\begin{equation}
\label{byrem}
\psi_{\mu,\nu}=\frac{1}{p}(\tilde{f}_{\nu}\phi_{\mu}-\tilde{f}_{\mu}\phi_{\nu}+f_{\mu,\nu})\ell(T)\in R_{\pi}[[\delta_{\pi,\eta}T\ |\ \eta\in \mathbb M_n^r]].\end{equation}
On the other hand we can write
$$\phi_{\pi, \mu}T=T^{p^r}+G_{\mu},\ \ \phi_{\pi, \nu}T=T^{p^s}+G_{\nu}$$
with $G_{\mu},G_{\nu}$ in the ideal $I_r$ of $R_{\pi}[\delta_{\pi,\eta}T\ |\ \eta\in \mathbb M_n^r]$
generated by the set 
$$\{\delta_{\pi,\eta}T\ |\ \eta\in \mathbb M_n^{r,+}\}.$$ 
A direct computation shows
$$\begin{array}{rcl}
p \psi_{\mu,\nu} & = & \tilde{f}_{\nu}\left(\sum_m \frac{\phi_{\mu}(b_m)}{m}(T^{p^r}+G_{\mu})^m\right)\\
\ & \ & \ \\
\ & \ &  - \tilde{f}_{\mu}\left(\sum_m \frac{\phi_{\nu}(b_m)}{m}(T^{p^s}+G_{\nu})^m\right)\\
\ & \ & \ \\
\ & \ & +
f_{\mu,\nu}\left(\sum_m \frac{b_m}{m}T^m\right).\end{array}
$$
Reducing  the above equality modulo the ideal $I_r$  we get
$$ \tilde{f}_{\nu}\left(\sum_m \frac{\phi_{\mu}(b_m)}{m}T^{p^rm}\right)- \tilde{f}_{\mu}\left(\sum_m \frac{\phi_{\nu}(b_m)}{m}T^{p^sm}\right)+f_{\mu,\nu}\left(\sum_m \frac{b_m}{m}T^m\right)\in pR_{\pi}[[T]].$$
For all integers $N\geq 1$, picking out the coefficients of $T^{p^r N}$, we get the following analogue of the integrality conditions of
Atkin and Swinnerton-Dyer \cite{ASD71, S87}.\end{remark}

\begin{corollary}\label{3elliptic}
For all integers $N\geq 1$
\begin{equation}
\tilde{f}_{\nu}\frac{\phi_{\mu}(b_N)}{N}-\tilde{f}_{\mu}\frac{\phi_{\nu}(b_{p^{r-s}N})}{p^{r-s}N}+f_{\mu,\nu}\frac{b_{p^rN}}{p^rN}\in pR_{\pi}.
\end{equation}
\end{corollary}

\begin{remark}\label{tutu}
For every isogeny $u:E'\rightarrow E$ of degree $d$ prime to $p$ of elliptic curves over $R_{\pi}$
and every invertible $1$-form $\omega$ on $E$, setting $\omega'=u^*\omega$, we have
$$f_{\mu,\nu}(E',\omega') = d\cdot f_{\mu,\nu}(E,\omega).$$
Indeed we may identify two admissible coordinates
 for  $E$ and $E'$ (call this parameter $T$) in which
case we identify the images of $\omega'$ and $\omega$ in $R_{\pi}[[T]]dT$ and we identify
the two series $\ell_{\omega}$ and $\ell_{\omega'}$ in $R_{\pi}[[T]]$. As in  Equation (\ref{byrem}) we consider the partial $\delta_{\pi}$-characters of $E$ and $E'$ respectively:
\begin{equation}
\label{byremm}
\psi:=\psi_{\mu,\nu}=\frac{1}{p}(\tilde{f}_{\nu}(E,\omega)\phi_{\mu}-\tilde{f}_{\mu}(E,\omega)\phi_{\nu}+f_{\mu,\nu}(E,\omega))\ell(T),\end{equation}
\begin{equation}
\label{byremmm}
\psi':=\psi'_{\mu,\nu}=\frac{1}{p}(\tilde{f}_{\nu}(E',\omega')\phi_{\mu}-\tilde{f}_{\mu}(E',\omega')\phi_{\nu}+f_{\mu,\nu}(E',\omega'))\ell(T).\end{equation}
Identifying $\psi$ with its image in the space of $\delta_{\pi}$-characters of $E'$ and 
using Remark \ref{4parts}, Part 5, we get that
$$\psi'-d\cdot \psi=
(f_{\mu,\nu}(E',\omega')-d\cdot f_{\mu,\nu}(E,\omega))\ell_{\omega}(T).$$
Hence
$$\theta(\psi'-d\cdot \psi)=f_{\mu,\nu}(E',\omega')-d\cdot f_{\mu,\nu}(E,\omega)\in R_{\pi}.$$
By Theorem \ref{mainthm}, Part 2, $\theta(\psi'-d\cdot \psi)=0$ which ends the proof.
\end{remark}

\begin{remark}\label{tamporal}
Let us write $f_{\pi, \mu,\nu}$ and 
$\psi_{\pi,\mu,\nu}$
instead of $f_{\mu,\nu}$ and $\psi_{\mu,\nu}$ if we want to emphasize dependence on $\pi$. Then for all $\pi'|\pi$ we  have
\begin{equation}
\label{lup1}
\psi_{\pi', \mu,\nu}=p^{2N(\pi')-2N(\pi)}\psi_{\pi, \mu,\nu}\in \mathbf{X}^r_{\pi',\Phi}(E),\end{equation}
\begin{equation}
\label{lup2}
f_{\pi', \mu,\nu}=p^{2N(\pi')-2N(\pi)}f_{\pi, \mu,\nu}\in R_{\pi'}.\end{equation}
This follows from Remark \ref{4parts}, Part 6 by an argument similar to the one in Remark \ref{tutu}.
\end{remark}

\subsection{The case $n=r=2$}\label{casenr2}

We continue to consider an elliptic curve $E$ over $R_{\pi}$ and a $1$-form $\omega$.
Consider, in what follows,  $\Phi=(\phi_1,\phi_2)$. We consider in this subsection the arithmetic Kodaira-Spencer classes of order $\leq 2$ and we derive some basic quadratic and cubic relations among them that will play a key role in the next section.

 Specializing the construction in the previous section to our case  we may consider the partial $\delta_{\pi}$-character 
$$\psi_{1,2}=\psi_{f_2  \phi_1 -f_1  \phi_2 }\in \mathbf{X}^1_{\pi,\phi_1,\phi_2}(E)^!.$$

\begin{remark}
If $f_1f_2\neq 0$ then $\psi_{1,2}$  is a ``genuinely partial" object (not expressible in terms of  ODE objects via face  maps); indeed, in this case,  by Theorem \ref{mainthm}, we have $\mathbf{X}^1_{\pi,\phi_1}(E)=\mathbf{X}^1_{\pi,\phi_2}(E)=0$. On the other hand $\psi^1_{12}$ can be viewed as an analogue of the transport equation 
 in \cite{BuSi10a}.\end{remark}
 
By Theorem \ref{mainthm} we have
$$\theta(\psi_{1,2})=\tilde{f}_2  \phi_1 -\tilde{f}_1  \phi_2 +f_{1,2}.$$
 By Remark \ref{weighttt} the dependence of $f_{1,2}$ on $\omega$ is as follows:
$$f_{1,2}(E,\lambda\omega)=\lambda^{\phi_1+\phi_2}f_{1,2}(E,\omega).$$

Next, for $i\in \{1,2\}$, we may consider the partial $\delta_{\pi}$-characters (induced via face  maps by the  ODE arithmetic Manin maps in \cite{Bu95}):
$$\psi_{ii,i}:=\psi_{\tilde{f}_i  \phi_i^2 -\tilde{f}_{ii}  \phi_i}\in \mathbf{X}^2_{\pi,\phi_1,\phi_2}(E)^!.$$
By Theorem \ref{mainthm} we have
$$\theta(\psi_{ii,i})=
\tilde{f}_i  \phi_i^2 -\tilde{f}_{ii}  \phi_i 
+f_{ii,i}.$$
 By Remark \ref{weighttt} the dependence of $f_{ii,i}$ on $\omega$ is as follows:
$$f_{ii,i}(E,\lambda\omega)=\lambda^{\phi_i+\phi_i^2}f_{ii,i}(E,\omega).$$

Finally   we may consider the partial $\delta_{\pi}$-character 
$$\psi_{11,22}:=\psi_{f_{22}  \phi_1^2 -f_{11}  \phi_2^2 }\in \mathbf{X}^2_{\pi,\phi_1,\phi_2}(E)^!.$$
By Theorem \ref{mainthm} we have 
$$\theta(\psi_{11,22})=\tilde{f}_{22}  \phi_1^2 -\tilde{f}_{11}  \phi_2^2 +f_{11,22}.$$
By Remark \ref{weighttt} the dependence of $f_{11,22}$ on $\omega$ is as follows:
$$f_{11,22}(E,\lambda\omega)=\lambda^{\phi_1^2+\phi_2^2}f_{11,22}(E,\omega).$$

 One has the following $6$  elements in the module $\mathbf{X}^2_{\pi,\phi_1,\phi_2}(E)$,
 \begin{equation}
 \label{thesix}
 \psi_{1,2},\  \phi_1\psi_{1,2},\ \phi_2 \psi_{1,2},\  \psi_{11,1},\ \psi_{22,2},\ \psi_{11,22}. 
 \end{equation}
 So if $f_1\neq 0$ or $f_2\neq 0$, since 
 $\mathbf{X}^2_{\pi,\phi_1,\phi_2}(E)$ has rank
 $2+2^2-1=5$ (cf. Theorem \ref{mainthm}), it follows that
  there must be a non-trivial $R_{\pi}$-linear relation among these $6$ elements:
 \begin{equation}\label{lliinneeaarr}
 \lambda_1 \psi_{1,2} + \lambda_2 \phi_1\psi_{1,2}+ \lambda_3 \phi_2 \psi_{1,2}+\lambda_4 \psi_{11,1}+\lambda_5 \psi_{22,2}+\lambda_6\psi_{11,22}=0,
 \end{equation}
 for some $\lambda_1,\ldots,\lambda_6\in R_{\pi}$, not all zero.
  The existence of such a  relation  implies the vanishing of all $6\times 6$  minors of the $6\times 7$ matrix $\Gamma$ of the coefficients of the Picard-Fuchs symbols of the elements in (\ref{thesix}) with respect to the basis
 \begin{equation}
 \label{basis}
 \phi_1^2,\ \phi_2^2,\ \phi_1\phi_2,\ \phi_2\phi_1,\ \phi_1,\phi_2, 1\end{equation}
 of $K^2_{\pi,\phi_1,\phi_2}$. One can compute this matrix explicitly. Indeed denote by $\theta_1,\ldots,\theta_6$ the Picard-Fuchs symbols of the elements in Equation (\ref{thesix}), let $e_1,\ldots,e_7$ be the elements in Equation (\ref{basis}) and let $\Gamma=(\gamma_{ij})$ be the $6\times 7$ matrix defined by the equalities
 $$\theta_i=\sum_{j=1}^7 \gamma_{ij}e_j,\ \ i=1,\ldots,6.$$
 We have the following matrix
 
 \bigskip

 $$
 \Gamma=\left(
 \begin{array}{ccccccc}
 0 & 0 & 0 & 0 & \tilde{f}_2 & -\tilde{f}_1 & f_{1,2}\\
 \ & \ & \ & \ & \ & \ & \ \\
 \tilde{f}_2^{\phi_1} & 0 & -\tilde{f}_1^{\phi_1} & 0 & f_{1,2}^{\phi_1} & 0 & 0 \\
  \ & \ & \ & \ & \ & \ & \ \\
  0 & -\tilde{f}_1^{\phi_2} & 0 & \tilde{f}_2^{\phi_2} & 0 & f_{1,2}^{\phi_2} & 0\\
   \ & \ & \ & \ & \ & \ & \ \\
   \tilde{f}_1 & 0 & 0 & 0 & -\tilde{f}_{11} & 0 & f_{11,1}\\
    \ & \ & \ & \ & \ & \ & \ \\
    0 & \tilde{f}_2 & 0 & 0 & 0 & -\tilde{f}_{22} & f_{22,2}\\
     \ & \ & \ & \ & \ & \ & \ \\
     \tilde{f}_{22} & -\tilde{f}_{11} & 0 & 0 & 0 & 0 & f_{11,22}
 \end{array}\right).
 $$
 
 \bigskip
 
 \noindent The upper left $5\times 5$ minor of the matrix $\Gamma$  is non-zero if $f_1f_2\neq 0$. In particular, the following corollary is proved.
 
 \begin{corollary} \label{C2}
 If  $f_1f_2\neq 0$ then
 the first $5$ elements in Equation (\ref{thesix}) are $R_{\pi}$-linearly independent and hence they form a basis up to torsion of $\mathbf{X}^2_{\pi,\Phi}(E)$. \end{corollary}
 
 On the other hand the linear combination of the rows of $\Gamma$ with coefficients $\lambda_1,\ldots,\lambda_6$ is $0$
 from which we get the following corollary.
 
 \begin{corollary}\label{C3} If $f_1f_2\neq 0$ then
 in Equation \ref{lliinneeaarr} we have $\lambda_2=\lambda_3=0$.\end{corollary}
 
  Assume $f_1f_2\neq 0$ and denote by  $\tilde{\Gamma}$  the $4\times 5$ matrix obtained from $\Gamma$ by removing the $3$rd and $4$th columns as well as the $2$nd and the $3$rd rows.
 The rows of $\tilde{\Gamma}$ are then linearly dependent so we get that all $4\times 4$ minors of the matrix $\tilde{\Gamma}$  vanish.
 The vanishing of the  minor 
 obtained by removing the fifth column of $\tilde{\Gamma}$ 
 is tautologically trivial so it does not yield any information.
 The vanishings of the rest of the minors of $\tilde{\Gamma}$ 
 is equivalent to one cubic relation (\ref{firstcol}) given in the following:
 
 \begin{corollary}\label{cubic}
 If $f_1f_2\neq 0$ then the following  relation holds in $R_{\pi}$,
  \begin{equation}
 \label{firstcol}
f_{11}f_{22}f_{1,2}+ f_2 f_{22}f_{11,1}-f_{11} f_1f_{22,2}-
  f_1f_2
f_{11,22}=0.\end{equation}
  \end{corollary}

 \begin{proposition}\label{iii}
 Assume  $\pi=p$. Then the following equalities hold in $R$,
 
 1) $f_1=f_2$, $f_{12}=f_{21}$.
 
 2)   $f_{11,1}=f_{22,2}$.
 
 3) $f_{1,2}=f_{11,22}=0$.
 \end{proposition}

 {\it Proof}. 
Part 1  follows from Lemma \ref{fifjx}.
Part 2 follows from the compatibility with face maps.
To check Part 3 
consider the compatible actions of $\Sigma_2=\{e,\sigma\}$ on $\mathbf{X}^1_{p,\Phi}(E)$ and $K^1_{p,\Phi}$.
We have 
$$\theta(\sigma\psi_{1,2})=\sigma (\theta(\psi_{1,2}))=\sigma(f_1\phi_1-f_1\phi_2+f_{1,2})=f_1\phi_2-f_1\phi_1+f_{1,2}.$$
Hence 
$$\theta(\psi_{1,2}+\sigma\psi_{1,2})=2f_{1,2}\in K_{\pi}.$$
By Theorem \ref{mainthm}, Part 2, it follows that $f_{1,2}=0$. The equality $f_{11,22}=0$ follows similarly.
 \qed

\begin{remark} \label{irir}
 Assume  that  $E$ comes
from a curve $E_{\mathbb Z_p}$ over $\mathbb Z_p$ and denote by 
$a_p$ is the trace of Frobenius on $E_{\mathbb Z_p}\otimes \mathbb F_p$. Also fix an index $i$. It follows  from \cite{Bu97}, Theorem 1.10, that if 
 $E$ is not a canonical lift of an ordinary elliptic curve then 
$$f_{ii}=a_p f_i,\ \ f_{ii,i}=pf_i.$$
Recall that, if in addition $p\geq 5$,  then $a_p=0$ if and only if $E$ has supersingular reduction. 
\end{remark}

 \bigskip

We continue by considering the partial $\delta_{\pi}$-character 
$$\psi_{12,1}:=\psi_{f_1 \phi_1\phi_2-f_{12}\phi_1}.$$
Its symbol is 
$$\theta(\psi_{12,1})=\tilde{f}_1 \phi_1\phi_2-\tilde{f}_{12} \phi_1+f_{12,1}.$$ 
 This symbol 
 must be a linear combination of the symbols of 
 $$\psi_{1,2},\ \phi_1\psi_{1,2},\  \phi_2\psi_{1,2},\  \psi_{11,1},\ \psi_{22,2}.$$
  Let $\Gamma'$ be the matrix obtained by replacing the last row in $\Gamma$ by the row
$$[0\ 0\ \tilde{f}_1\ 0\ -\tilde{f}_{12} \ 0\ f_{12,1}].$$
Then the determinants of the matrices obtained from $\Gamma'$ by deleting the $5$th and the $7$th columns respectively must be $0$. The vanishing of these determinants yields:

\begin{lemma}\label{quadraticc}
If $f_1f_2\neq 0$ then
the following relations hold in $R_{\pi}$,
\begin{equation}
\label{icu1}
f_{12,1}f_1^{\phi_1}-f_{11,1}f_2^{\phi_1}=0,
\end{equation}
\begin{equation}\label{ica1}
\tilde{f}_{12}\tilde{f}_1^{\phi_1}-\tilde{f}_{11}\tilde{f}_2^{\phi_1} -\tilde{f}_1f_{1,2}^{\phi_1}=0.\end{equation}
\end{lemma}

Similarly, by looking at the partial $\delta_{\pi}$-character 
$$\psi_{21,2}:=\psi_{f_2 \phi_2\phi_1-f_{21}\phi_2}$$
we get:

\begin{lemma}\label{quadraticccc}
If $f_1f_2\neq 0$ then
the following relations hold in $R_{\pi}$,
\begin{equation}
\label{icu2}
f_{21,2}f_2^{\phi_2}-f_{22,2}f_1^{\phi_2}=0,
\end{equation}
\begin{equation}\label{ica2}
\tilde{f}_{21}\tilde{f}_2^{\phi_2}-\tilde{f}_{22}\tilde{f}_1^{\phi_2} -\tilde{f}_2f_{2,1}^{\phi_2}=0.
\end{equation}
\end{lemma}

Next consider  the partial $\delta_{\pi}$-character 
$$\psi_{12,21}:=\psi_{f_{21} \phi_1\phi_2-f_{12}\phi_2\phi_1}.$$
Its symbol is 
$$\theta(\psi_{12,21})=\tilde{f}_{21} \phi_1\phi_2-\tilde{f}_{12}\phi_2\phi_1+f_{12,21}.$$ 
 Set
 $$\psi:=\tilde{f}_1\tilde{f}_2\psi_{12,21}-\tilde{f}_2 \tilde{f}_{21}\psi_{12,1}+\tilde{f}_1\tilde{f}_{12}\psi_{21,2}-\tilde{f}_{12}\tilde{f}_{21}\psi_{1,2}.$$
 One trivially checks the following identity
 $$\theta(\psi)=f_{12,21}\tilde{f}_1\tilde{f}_2-\tilde{f}_2 \tilde{f}_{21} f_{12,1}+\tilde{f}_1\tilde{f}_{12}f_{21,2}-
 \tilde{f}_{12}\tilde{f}_{21}f_{1,2}.$$
 By Theorem \ref{mainthm}, Part 1 we get
 
 \begin{lemma} \label{hotzu}
 If $f_1f_2\neq 0$ then
the following relation holds in $R_{\pi}$:
 \begin{equation}
 f_{12,21}f_1f_2-f_2 f_{21} f_{12,1}+f_1f_{12}f_{21,2}-
 f_{12}f_{21}f_{1,2}=0.\end{equation}
 \end{lemma}

 Similarly consider the partial $\delta_{\pi}$-characters
 $$\tilde{f}_1\psi_{11,2}-\tilde{f}_2\psi_{11,1}-\tilde{f}_{11}\psi_{12},$$
 $$\tilde{f}_1\psi_{11,12}-\tilde{f}_{12}\psi_{11,1}+\tilde{f}_{11}\psi_{12,1},$$
 $$\tilde{f}_1\psi_{12,2}-\tilde{f}_2\psi_{12,1}-\tilde{f}_{12}\psi_{1,2},$$
 $$\tilde{f}_2\psi_{11,21}-\tilde{f}_{21}\psi_{11,2}+\tilde{f}_{11}\psi_{21,2}.$$
 The symbols of these partial $\delta_{\pi}$-characters are equal to the expressions in the left hand sides of the equalities in the Lemma \ref{whyawake} below. By Theorem \ref{mainthm}, Part 1, since these symbols are in $R_{\pi}$ they must vanish. So we have the following lemma.
 
 \begin{lemma}
 \label{whyawake}
 If $f_1f_2\neq 0$ then
the following relations hold in $R_{\pi}$,
 \begin{equation}
 \label{pasa1}
 f_1f_{11,2}-f_2f_{11,1}-f_{11}f_{1,2}=0,
 \end{equation}
 \begin{equation}
 \label{pasa2}
 f_1f_{11,12}-f_{12}f_{11,1}+f_{11}f_{12,1}=0,
 \end{equation}
 \begin{equation}
 \label{pasa3}
 f_1f_{12,2}-f_2f_{12,1}-f_{12}f_{1,2}=0,
 \end{equation}
 \begin{equation}
 \label{pasa4}
 f_2f_{11,21}-f_{21}f_{11,2}+f_{11}f_{21,2}=0.
 \end{equation}
 Moreover the relations obtained from the above relations by switching the indices $1$ and $2$ also hold.
 \end{lemma}

\begin{remark}\label{uy}
One can ask if one can ``extend (\ref{firstcol}), (\ref{icu1}), (\ref{ica1}), (\ref{icu2}), (\ref{ica2}), (\ref{pasa1}), (\ref{pasa2}), (\ref{pasa3}), (\ref{pasa4}) by continuity" so that these remain true without the condition $f_1f_2\neq 0$. We claim this is the case as an immediate consequence of Theorems \ref{roo}, \ref{quadocc} and the  formulae (\ref{zipo}) and (\ref{zipopo}) to be stated and proved later. 
By the way we will later prove the following result.\end{remark}
 
\begin{theorem}\label{fat2}
Assume $\phi_1,\phi_2$ are monomially independent 
 in $\mathfrak G(K^{\textup{alg}}/\mathbb Q_p)$.
Then there exist $\pi\in \Pi$ and a pair $(E,\omega)$ over $R_{\pi}$ such that $E$ has ordinary reduction and all classes $f_{\mu}$, $f_{\mu,\nu}$ with $\mu,\nu\in \mathbb M^{2,+}_2$, $\mu\neq \nu$, attached to $(E,\omega)$ are non-zero. 
\end{theorem}

\section{The relative theory}\label{relativ}

The theory developed so far  over $R_{\pi}$ should be viewed as an ``absolute" theory and has a ``relative" version in which the $\delta_{\pi}$-prolongation sequence $R^*_{\pi}$ is replaced by an arbitrary object  $S^*$ in ${\bf Prol}^*_{\pi,\Phi}$. This relative version of the theory is crucial for developing the formalism of {\bf partial $\delta_{\pi}$-modular forms} in the next section. In the present section we present a quick discussion of this relative version of the theory. 

Again we consider the variables  $\delta_{\pi,\mu}y_j$ for $\mu\in {\mathbb M}_n$, $\pi\in \Pi$, $j\in \{1,\ldots,N\}$. 
Let $S^*=(S^r)$ be an object in ${\bf Prol}_{\pi,\Phi}$.
Fix $N$ a positive integer and consider the 
the ring $S^0[y_1,\ldots,y_N]$ and  the rings 
$$J_{\pi,\Phi}^r(S^0[y_1,\ldots,y_N]/S^*):=S^r[\delta_{\pi,\mu}y_j\ |\ \mu\in {\mathbb M}_m^r, j=1,\ldots,N]^{\widehat{\ }}.$$
 The sequence
$J^*_{\pi,\Phi}(S^0[y_1,\ldots,y_N]/S^*):=(J^r_{\pi,\Phi}(S^0[y_1,\ldots,y_N]/S^*))$ has, again, a unique structure of object in ${\bf Prol}^*_{\pi,\Phi}$ such that $\delta_{\pi,i} \delta_{\pi,\mu}y:=\delta_{\pi, i\mu}y$ for all $i=1,\ldots,n$; if $S^*$ is an object of  ${\bf Prol}_{\pi,\Phi}$ then $J^*_{\pi,\Phi}(S^0[y_1,\ldots,y_N]/S^*)$ is also an object of  ${\bf Prol}_{\pi,\Phi}$.

For every object $S^*$ in ${\bf Prol}^*_{\pi,\Phi}$ and every
  $S^0$-algebra of finite type written as $A := S^0[y_1,\ldots,y_N]/I$, we define the ring
$$J^r_{\pi,\Phi}(A/S^*):=J_{\pi}^r(S^0[y_1,\ldots,y_N]/S^*)/(\delta_{\pi,\mu}I\ |\ m\in {\mathbb M}_m^r).$$  
 If $S^*=R^*_{\pi}$ then   $J^r_{\pi,\Phi}(A/S^*)$ coincides with the previously defined ring $J^r_{\pi,\Phi}(A)$.
	
If $S^*$ is an object of ${\bf Prol}_{\pi,\Phi}$,
$A$ is a smooth $S^0$-algebra, and $u \colon S^0[T_1,\ldots,T_d] \to A$ is an \'etale morphism of $S^0$-algebras, then, again, 
there is a (unique) isomorphism 
$$(A\otimes_{S^0}S^r)[ \delta_{\pi,\mu}T_j\ |\ \mu\in {\mathbb M}_m^{r,+}, j=1,\ldots,d]^{\widehat{\ }}\cong J^r_{\pi,\Phi}(A/S^*)$$
sending $\delta_{\pi,\mu} T_j$ into $\delta_{\pi,\mu}(u(T_j))$ for all $j$ and $\mu$.
In particular $J^r_{\pi,\Phi}(A/S^*)$ is Noetherian and flat over $R_{\pi}$ so the sequence $J^*_{\pi,\Phi}(A/S^*)$ is an object of ${\bf Prol}_{\pi,\Phi}$.

As in Theorem \ref{thm:univ} we have the following universal property.
Assume $S^*$ is an object of ${\bf Prol}_{\pi,\Phi}$ and $A$ is a smooth $S^0$-algebra. 
For every object $T^*$ of 
${\bf Prol}_{\pi,\Phi}$ and every $S^0$-algebra map $u:A\rightarrow T^0$ and any morphism $S^*\rightarrow T^*$ in ${\bf Prol}_{\pi,\Phi}$
there is a unique morphism  $J^*_{\pi,\Phi}(A/S^*)\rightarrow T^*$ over $S^*$ in 
${\bf Prol}_{\pi,\Phi}$
compatible with $u$. (A similar result holds for ${\bf Prol}_{\pi,\Phi}^*$.)

As in Definition \ref{APDE} for every object $S^*$ in 
${\bf Prol}_{\pi,\Phi}$ and every 
smooth scheme $X$ over $S^0$ we 
define the {\bf relative partial $\pi$-jet space}
$$J^r_{\pi,\Phi}(X/S^*)=\bigcup Spf(J^r_{\pi,\Phi}(\mathcal O(U_i))/S^*),$$ where 
$X=\bigcup_i U_i$ is (any) affine open cover. 
If $S^*=R_{\pi}^*$, $J^r_{\pi,\Phi}(X/S^*)$ coincides with the previously defined formal scheme $J^r_{\pi,\Phi}(X)$. 

Let $S^*$ be an object in ${\bf Prol}_{\pi,\Phi}$ and $G$ a commutative smooth group scheme over $S^0$ we define
a {\bf relative partial $\delta_{\pi}$-character} of order $\leq r$  of $G$ over $S^*$ to be  a group homomorphisms  $J^r_{\pi,\Phi}(G/S^*)\rightarrow \widehat{\mathbb G_{a,S^r}}$
in the category of $p$-adic formal schemes. (Here $G_{a,S^r}$ is, of course the additive group scheme over $S^r$.)
We denote by
$\mathbf{X}^r_{\pi,\Phi}(G/S^*)=\text{Hom}(J^r_{\pi,\Phi}(G), \widehat{\mathbb G_{a,S^r}})$ the 
$S^r$-module of relative partial $\delta_{\pi}$-characters of $G$ of order $\leq r$. 

For a family  $\Phi:=(\phi_1,\ldots,\phi_n)$, $\phi_i \in \mathfrak F(K^{\text{alg}}/\mathbb Q_p)$ of distinct Frobenius automorphisms and for an object $S^*$ in
${\bf Prol}_{\pi,\Phi}$ we define the $S^r$-modules of {\bf symbols}
 $S^r_{\pi,\Phi}$ 
 to be the free $S^r$-module with basis $\{\phi_{\mu}\in \mathbb M_{\Phi}\ |\ \mu\in \mathbb M^r_n\}$. We consider the 
 rings $S^r\otimes \mathbb Q=S^r\otimes_{R_{\pi}}K_{\pi}$ and the 
 $S^r\otimes \mathbb Q$-modules $S^r_{\pi,\Phi}\otimes \mathbb Q$;
 they play the roles, in this relative setting, of $K_{\pi}$ and $K^r_{\pi,\Phi}$, respectively.
 
 As in Proposition \ref{asun} we have
 $$\mathbf{X}^r_{\pi,\Phi}(\mathbb G_a/S^*)=
((S^r\otimes \mathbb Q )T)\cap (S^r[\delta_{\pi,\nu}T\ |\ \nu\in \mathbb M_n^r])$$
where the intersection is taken inside the ring 
$(S^r \otimes \mathbb Q)[[\phi_{\pi,\nu}T\ |\ \nu\in \mathbb M_n^r]]$.
 
  Let $S^*$ be an object in ${\bf Prol}_{\pi,\Phi}$, let $G$ have relative dimension $1$ over $S^0$ and assume we are given  an invariant $1$-form $\omega$ on $G/S^0$ and an  {\bf admissible coordinate} $T$ (defined in the obvious corresponding way) on $G/S^0$. (Note that $\omega$ and $T$ may not exist in general but they exist locally on $\textup{Spec}(S^0)$ in the Zariski topology.)
  Then, as 
  in Definition \ref{PFsymbol},
   one can attach to every 
 $\psi\in \mathbf{X}^r_{\pi,\Phi}(G/S^*)$ a {\bf Picard-Fuchs symbol} $\theta(\psi)\in S^r_{\pi,\Phi}\otimes \mathbb Q$. 
 
 The various results about $\delta_{\pi}$-characters obtained in the previous sections have (obvious) relative analogues (over objects $S^*$ in 
 ${\bf Prol}_{\pi,\Phi}$
 instead of over $R^*_{\pi}$)  that are proved using essentially identical arguments.  We note, however,  that the relative analogues over $S^*$ of 
 Corollary \ref{C1} and of 
 the  results
 in Section \ref{casenr2} need the hypothesis that the rings $S^r$ be integral domains; indeed 
 for Corollary \ref{C1} we need the concepts of rank and torsion of an $S^r$-module to be well defined while for the analysis in Section \ref{casenr2}  we need  the fact that linear dependence in torsion free $S^r$-modules is expressed via vanishings of corresponding determinants. 
 Rather than stating these relative analogues here we will use them freely in what follows with
  appropriate references to the corresponding ``absolute" results in the previous sections. 
    
\section{Partial $\delta$-modular forms}\label{modforms}

\subsection{Basic definitions}
We start with the standard extension of \cite{Bu00} or \cite[Sec. 8.4.1]{Bu05} to our setting of partial differential equations. In this section, $\pi \in \Pi$. Continue to set $\Phi = (\phi_1, \ldots, \phi_n)$ a fixed choice of Frobenius automorphisms of $K^{\textup{alg}}$. Consider the category of triples $(E/S^0, \omega, S^*)$ where $S^*$ is an object in ${\bf Prol}_{\pi,\Phi}$, $E/S^0$ is a elliptic curve, and $\omega \in H^0(E,\Omega_{E/S_0})$ is a basis. A morphism in this category is a map of tuples $(E/S^0, \omega, S^*) \to (E'/T^0, \omega', T^*)$ consists of a map of prolongation sequences $T^* \to S^*$ and a compatible map $E/S^0 \to E'/T^0$ of curves pulling back $\omega'$ to $\omega$. 

\begin{definition} A {\bf partial $\d_\pi$-modular function of order at most $r \geq 0$} is a rule $f$ assigning to each object $(E/S^0, \omega, S^*)$ an element $f(E/S^0, \omega, S^*) \in S^r$, depending only on the isomorphism class of $(E/S^0, \omega, S^*)$, such that $f$ commutes with base change of the prolongation sequence. Specifically, if $u \colon S^* \to T^*$ is a map of prolongation sequences, then $f((E \times_{S^0} T^0) / T^0,  u^* \omega, T^0) = u^r f(E/S^0, \omega, S^*)$. We denote by $M_{\pi,\Phi}^r$ the set of all partial $\d_\pi$-modular function of order at most $r \geq 0$; this set has an obvious structure of ring. An element of $M_{\pi,\Phi}^r$ is said to have order  $r$ if it is not in the image of the canonical  map $M_{\pi,\Phi}^{r-1}\rightarrow M_{\pi,\Phi}^r$. The latter map is, by the way, injective as can be seen from the next Remark.
\end{definition}

\begin{remark}\label{rmk:notabundle}
Consider two variables $a_4$ and $a_6$, let $\Delta:=4a_4^3+27a_6^2$, and let us consider the $R_{\pi}$-algebra
 $M_{\pi}:=R_{\pi}[a_4,a_6,\Delta^{-1}]$ and the affine scheme $B(1):=\textup{Spec}(M_{\pi})$. 
 (The scheme $B(1)$ has a natural $\mathbb G_m$-action and a morphism to the ``$j$-line" $Y(1)$ which is however not a $\mathbb G_m$-bundle; later we will consider level $\Gamma_1(N)$-structures, the modular curves $Y_1(N)$, and the corresponding $\mathbb G_m$-bundles $B_1(N)$.)
 For all $R_{\pi}$-algebras $S$ the set $B(1)(S)$ of $S$-points of $B(1)$ is in a natural bijection with the set of pairs
 $(E,\omega)$ consisting of an elliptic curve $E/S$ and a basis $\omega$ for the $1$-forms on $E/S$. Then, as in \cite{Bu00},   we have an identification of $R_{\pi}$-algebras
 $$M_{\pi,\Phi}^r\simeq J^r_{\pi,\Phi}(M_{\pi})=\mathcal O(J^r_{\pi,\Phi}(B(1)))\simeq R_{\pi}[\delta_{\mu} a_4,\delta_{\mu} a_6,\Delta^{-1}\ |\ \mu\in \mathbb M_n^r]^{\widehat{\ }}.$$
 In particular $M_{\pi,\Phi}^0\simeq \widehat{M_{\pi}}$.
\end{remark}

In what follows we discuss weights. In the case $\pi = p$ and $\Phi = \{\phi\}$, weights are  taken to be  elements of the polynomial ring $\mathbb{Z}[\phi]$ in the ``variable" $\phi$. In the partial differential setting considered here, we consider weights in the ring of integral symbols, $\mathbb Z_{\Phi}$.  As before, 
if $w = \sum m_\mu \phi_\mu\in \mathbb Z_{\Phi}$, $m_{\mu}\in \mathbb Z$, and if $S^*$ is an object of ${\bf Prol}_{\pi,\Phi}$ and $\lambda \in (S^0)^{\times}$, we write
$$\lambda^w = \prod_{\mu \in \mathbb M_n} (\phi_\mu(\lambda))^{m_\mu} \in (S^r)^\times.$$ 
We use a similar notation for $\lambda\in S^0$ in case all $\lambda_{\mu}\geq 0$. We have the formulae $\lambda^{w_1+w_2}=\lambda^{w_1}\lambda^{w_2}$ and $\lambda^{w_1w_2}=(\lambda^{w_2})^{w_1}$.
\bigskip

\begin{definition}\label{defofweighty}
A partial $\d_{\pi}$-modular function $f\in M^r_{\pi,\Phi}$ is called a {\bf partial $\delta_{\pi}$-modular form of weight} $w \in \mathbb Z^r_{\Phi}$ provided for all $(E/S^0,\omega,S^*)$ and for each $\lambda \in(S^0)^{\times}$ we have 
\begin{equation}
\label{defform}
f(E/S^0,\lambda\omega,S^*) = \lambda^{-w} f(E/S^0,\omega,S^*).\end{equation}
We denote by $M^r_{\pi,\Phi}(w)$ the $R_{\pi}$-module of partial $\d_{\pi}$-modular forms of weight $w$. 
\end{definition}

\begin{remark}
If $t$ is a variable we may consider the ring
\begin{equation}
\label{bababa}
J^r_{\pi,\Phi}(M_{\pi}[t,t^{-1}])\simeq R_{\pi}[\delta_{\mu} a_4,\delta_{\mu} a_6,\delta_{\mu} t, \Delta^{-1}, t^{-1}\ |\ \mu\in \mathbb M_n^r]^{\widehat{\ }}.\end{equation}
 Then for
 $f\in M^r_{\pi,\Phi}$ we have that $f\in M^r_{\pi,\Phi}(w)$ if and only if 
 $$f(\ldots,\delta_{\mu}(t^4a_4),\ldots,\delta_{\mu}(t^6a_6),\ldots)=t^{w}f(\ldots,\delta_{\mu}a_4,\ldots,\delta_{\mu}a_6,\ldots)$$
 in the ring (\ref{bababa}).
\end{remark}

\begin{remark}\label{brab}
The direct sum 
$\bigoplus_{w\in \mathbb Z_{\Phi}} M^r_{\pi,\Phi}(w)$
has a natural structure of $\mathbb Z_{\Phi}$-graded $R_{\pi}$-algebra. Moreover for every $i$
and $f\in M^r_{\pi,\Phi}(w)$ we have naturally defined form 
$f^{\phi_i}\in M^r_{\pi,\Phi}(\phi_i w)$. Consequently we have natural {\bf bracket} $\mathbb Z_p$-bilinear maps
$$\{\ ,\ \}_{\pi,i}:M^r_{\pi,\Phi}(w)\times M^r_{\pi,\Phi}(w)\rightarrow M^r_{\pi,\Phi}((\phi_i+p)w)$$
defined by 
$$\{f,g\}_{\pi,i}:=\frac{1}{\pi}(f^{\phi_i}g^p-g^{\phi_i}f^p)=g^p\delta_{\pi,i}f-f^p\delta_{\pi,i}g.$$
\end{remark}

\subsection{Jet construction}

Examples of $\d_{\pi}$-modular forms are provided by primary and secondary arithmetic Kodaira-Spencer classes as we shall explain in what follows. We begin with primary classes where each $f_{\mu}$ for $\mu \in \mathbb M_n$ ``comes from" a $\d_{\pi}$-modular function (which we denote by $f^{\textup{jet}}_{\pi, \mu}$). Indeed consider a prolongation sequence $S^*$ over $R$, an  elliptic curve $E/S^0$, and $\omega \in H^1(E,\Omega_{E/S^0})$  a basis.  For fixed $r$ and  $\mu \in {\mathbb M}_n^r$, replicating the arguments in the construction of $f_{\mu}$ in Remark~\ref{ovc} 
(with jet spaces over $R^*$ replaced by relative jet spaces over $S^*$ as in Section \ref{relativ}) and setting $\eta$ for the class of the corresponding  $\partial^r(\tilde{L}^{\mu}_{\pi,\Phi})$ in Definition \ref{primaryy} we define $f^{\textup{jet}}_{\pi,\mu}(E/S^0,\omega,S^*) = \langle \eta, \omega \rangle \in S^r$. In particular using the notation in Remark \ref{4parts}, Part 2, for every $(E,\omega)$ over $R_{\pi}$ we have 
\begin{equation}
\label{2use}
f_{\mu}(E,\omega)= f^{\textup{jet}}_{\pi, \mu}(E/R_{\pi},\omega,R_{\pi}^*).\end{equation}
Using the corresponding version over $S^*$ of Remark \ref{4parts}, Part 2 we get:

\begin{theorem}\label{ux}
 The rule $f^{\textup{jet}}_{\pi,\mu}$ defines  a partial $\delta_{\pi}$-modular form of weight $-1 - \phi_\mu$. \end{theorem}
 
 These forms are  generalizations of the forms  $f_{\textrm{jet}}$ constructed in \cite[Const. 4.1]{Bu00} or $f_{\textrm{jet}}^r$ in \cite[Sec. 8.4.2]{Bu05}.  
For $\pi=p$ we write
$$f^{\textup{jet}}_{p, \mu}=f^{\textup{jet}}_{\mu}.$$

 \begin{remark}\label{yuyi}
 Assume $\pi=p$ and let $E_{p-1}\in \mathbb Z_p[a_4,a_6]\subset M^0_{p,\Phi}$ be the polynomial that corresponds to the normalized Eisenstein series of weight $p-1$; we recall that the reduction mod $p$ of $E_{p-1}$ is the Hasse invariant.  Then, exactly as in \cite[Prop. 8.55]{Bu05} and using Remark \ref{ocong}, we get that for every
$i_1\ldots i_r\in \mathbb M_n^{r,+}$ we have
$$f^{\textup{jet}}_{i_1\ldots i_r}\equiv E_{p-1}^{1+p+\ldots + p^{r-2}}\cdot (f^{\textup{jet}}_{i_r})^{p^{r-1}}\ \ \textup{mod}\ \ p\ \ \textup{in}\ \ M^r_{p,\Phi}.$$
In particular for all $\mu\in \mathbb M_n^{r,+}$ we have
$$f^{\textup{jet}}_{\mu}\in M^r_{p,\Phi} \backslash p M^r_{p,\Phi}, \ \ \text{hence}\ \ \  f^{\textup{jet}}_{\mu}\neq 0.$$
 \end{remark}

  Next, for fixed $r$ and  $\mu,\nu  \in {\mathbb M}_n^{r,+}$ distinct, replicating the arguments in the construction of $f_{\mu,\nu}$ in Definition \ref{secondaryy}
(with jet spaces over $R^*$ replaced by relative jet spaces over $S^*$ as in Section \ref{relativ})  we define $f^{\textup{jet}}_{\pi, \mu,\nu}(E,\omega,S^*) \in S^r$. In particular using the notation in Remark \ref{weighttt} for every $E$ over $R_{\pi}$ we have
\begin{equation}
\label{oouu}
f_{\mu,\nu}(E,\omega)=f^{\textup{jet}}_{\pi, \mu,\nu}(E/R_{\pi},\omega,R_{\pi}^*).
\end{equation}

 Using the corresponding version over $S^*$  of Remarks \ref{weighttt} and \ref{antisymmetry}, we get:
 
 \begin{theorem}\label{uxx}
 The rule $f^{\textup{jet}}_{\pi, \mu,\nu}$, for $\mu, \nu \in {\mathbb M}_n^{r,+}$ distinct, defines a  partial $\d_{\pi}$-modular forms of weight $-\phi_\nu-\phi_\mu$. Moreover 
 $f^{\textup{jet}}_{\pi, \mu,\nu}+f^{\textup{jet}}_{\pi, \nu, \mu}=0$.
 \end{theorem}

For $\pi=p$ we write
$$f^{\textup{jet}}_{p, \mu,\nu}=:f^{\textup{jet}}_{\mu,\nu}.$$
 
  Our next goal is to show the above forms enjoy the special property of being ``isogeny covariant" which we now  define in our setting.

\begin{definition} \label{defofisocov}
For every weight $w = \sum_{\mu \in \mathbb M_n} m_\mu \phi_\mu$, set $\deg(w) := \sum_{\mu \in \mathbb M_n} m_\mu$. Let $f$ be a partial $\d_{\pi}$-modular form $f$ of weight $w\in \mathbb Z_{\Phi}^r$ where $\deg(w)$ is even. We say $f$ is {\bf isogeny covariant} provided for every tuple $(E/S^0,\omega,S^*)$ and every isogeny of degree prime to $p$, $u \colon E' \to E$ over $S^0$,  setting $\omega' = u^* \omega$, we have
$$f(E'/S^0,\omega',S^*) = [\deg(u)]^{-\deg(w)/2} f(E/S^0,\omega,S^*).$$ 
We denote by $I^r_{\pi,\Phi}(w)$  the $R_{\pi}$-module of isogeny covariant partial $\d_{\pi}$-modular forms of weight $w$.
\end{definition}

\begin{remark}
The direct sum 
$\bigoplus_{w\in \mathbb Z_{\Phi}} I^r_{\pi,\Phi}(w)$
is a $\mathbb Z_{\Phi}$-graded  $R_{\pi}$-subalgebra of $\bigoplus_{w\in \mathbb Z_{\Phi}} M^r_{\pi,\Phi}(w)$.
For every $f\in I^r_{p,\Phi}(w)$ and every $i$ 
we have $f^{\phi_i}\in I^r_{p,\Phi}(\phi_i w)$ and consequently the brackets in Remark \ref{brab}
induce brackets
$$\{\ ,\ \}_{\pi,i}:I^r_{\pi,\Phi}(w)\times I^r_{\pi,\Phi}(w)\rightarrow I^r_{\pi,\Phi}((\phi_i+p)w).$$
\end{remark}

\begin{theorem}\label{theyareisogcov}
The partial $\d_{\pi}$-modular forms $f^{\textup{jet}}_{\pi,\mu}$  and $f^{\textup{jet}}_{\pi,\mu,\nu}$ are isogeny covariant.
\end{theorem}

\begin{proof}
This follows by adapting the arguments in Remark \ref{4parts}, Part 5 and Remark \ref{tutu}
with $R^*_{\pi}$ replaced by an arbitrary prolongation sequence $S^*$. \end{proof}

Exactly as in \cite{BM20} the forms $f^{\textup{jet}}_{\pi,\mu},f^{\textup{jet}}_{\pi,\mu,\nu}$ induce totally $\delta$-overconvergent arithmetic PDEs on $B(1)$ and on certain natural bundles $B_1(N)$ over modular curves $Y_1(N)$. We explain this in what follows.

\begin{definition}\label{defofB}
Consider the modular curve $Y_1(N):=X_1(N)\setminus \{\textup{cusps}\}$ over $R_{\pi}$ where  $N\geq  4$, $N$ coprime to $p$,
and let $L$ be the line bundle on $Y_1(N)$ equal to the direct image of $\Omega_{E_{\textup{univ}}/Y_1(N)}$ where $E_{\textup{univ}}$ is the universal elliptic curve
over $Y_1(N)$.
Let $X\subset Y_1(N)$ be an affine open set, continue to denote by $L$ the restriction of $L$ to $X$, and consider the natural $\mathbb G_m$-bundle 
$$B_1(N):=\textup{Spec}\left( \bigoplus_{m\in \mathbb Z} L^m\right)\rightarrow X.$$
\end{definition}

The main example we have in mind  is the case
$X=Y_1(N)$. 

Recall that if $X$ is such that $L$ is free over $X=\textup{Spec}(A)$ with basis $\omega$ then we have a natural identification 
\begin{equation}
\label{wqe}
\mathcal O(J^r_{\pi,\Phi}(B_1(N)))\simeq J^r_{\pi,\Phi}(A)[x,x^{-1},\delta_{\mu}x\ |\ 
\mu\in \mathbb M_n^{r,+}]^{\widehat{\ }}\end{equation}
where $x$ is a variable identified with the section $\omega$.
We define the $R_{\pi}$-module
$$M^r_{\pi,\Phi,X}(w):=\widehat{\mathcal O(X)}\cdot x^w.$$
If $X$ is arbitrary
and $X=\cup X_i$ is an open cover such that $L$ is trivial on each $X_i$ 
 then we define $M^r_{\pi,\Phi,X}(w)$ to be the the submodule of 
$\mathcal O(J^r_{\pi,\Phi}(B_1(N)))$ of all elements whose restriction to every
$\mathcal O(J^r_{\pi,\Phi}(B_1(N)\times_X {X_i}))$ lies in $M^r_{\pi,\Phi,X_i}(w)$; this definition is independent of the covering considered.

Rercall the scheme $B(1)$  defined in Remark~\ref{rmk:notabundle}.
Exactly as in \cite[Sec. 5.2]{BM20}  every partial $\delta_{\pi}$-modular form 
$$f=f^{B(1)}\in M^r_{\pi,\Phi}(w)\subset \mathcal O(J^r_{\pi,\Phi}(B(1)))$$
   induces an element 
$$f^{B_1(N)}\in M^r_{\pi,\Phi,X}(w)\subset \mathcal O(J^r_{\pi,\Phi}(B_1(N))).$$  
We recall that for $X=\textup{Spec}(A)$ such that $L$ has a basis $x$ corresponding to a $1$-form $\omega$ we define
$$f^{B_1(N)}:=f(E_{\textup{univ}},\omega,J^*_{\pi,\Phi}(A))\cdot x^w;$$
for arbitrary $X$ we glue the elements just defined.
Exactly as in \cite[Sec. 5.3]{BM20} we have the following theorem.

 \begin{theorem}\label{theyareover}
 \label{beau} Let $B$ be either $B(1)$ or $B_1(N)$ with $N\geq 4$.
  For all $\mu,\nu\in \mathbb M_n^{r,+}$
 the elements $(f^{\textup{jet}}_{\pi,\mu})^B, (f^{\textup{jet}}_{\pi,\mu,\nu})^B\in \mathcal O(J^r_{\pi,\Phi}(B))$ are totally $\d$-overconvergent. So there are induced maps
 $$((f^{\textup{jet}}_{\pi,\mu})^B)^{\textup{alg}},((f^{\textup{jet}}_{\pi,\mu,\nu})^B)^{\textup{alg}}:B(R^{\textup{alg}})\rightarrow K^{\textup{alg}}.
 $$
 \end{theorem}
 
\begin{remark}
For $\pi'|\pi$ and every point $P\in B(R_{\pi'})$, denoting by $(E_P,\omega_P)$ the corresponding elliptic curve over $R_{\pi'}$ equipped with the induced $1$-form, we have that
\begin{equation}
\label{zipo}
((f^{\textup{jet}}_{\pi,\mu})^B)^{\textup{alg}}(P)=p^{-N(\pi')+N(\pi)}f_{\pi',\mu}(E_P/R_{\pi'},\omega_P,R^*_{\pi'}),\end{equation}   
\begin{equation}
\label{zipopo}
((f^{\textup{jet}}_{\pi,\mu,\nu})^B)^{\textup{alg}}(P)=p^{-2N(\pi')+2N(\pi)}f_{\pi',\mu,\nu}(E_P/R_{\pi'},\omega_P, R^*_{\pi'});\end{equation}
cf. Remark \ref{4parts}, Part 6 and Remark \ref{tamporal}.\end{remark}

\begin{definition}
For  every selection map $\epsilon$ with respect to $(\Phi',\Phi'',p)$ (cf. Definition \ref{permu})
and  every $f\in M^r_{p,\Phi'}$ 
we define $f_{\epsilon}\in M^r_{p,\Phi''}$ by letting
$$f_{\epsilon}(E/S^0,\omega,S^*):=f(E/S^0,\omega,S^*_{\epsilon})$$
for every prolongation sequence $S^*$, every elliptic curve $E/S^0$ and every basis $\omega$ for the $1$-forms on $E$. In particular we get the following special cases:

1) There is  a natural action  
$$\Sigma_n\times M^r_{p,\Phi}\rightarrow M^r_{p,\Phi},\ \ (\sigma,f)\mapsto \sigma f:=f_{\sigma}.$$ 

2) For every  $1\leq i_1<i_2<\ldots i_s\leq n$ there are natural {\bf face } homomorphisms 
$$M^r_{p,\phi_{i_1},\ldots,\phi_{i_s}}\rightarrow M^r_{p,\Phi}.$$

3) For every $\phi\in \mathfrak F(K^{\text{alg}}/\mathbb Q_p)$ there is a  natural {\bf degeneration} homomorphism 
$$M^r_{p,\Phi}\rightarrow M^r_{p,\phi}.$$

4) The composition of the face and degeneration maps below is the identity:
$$\textup{id}:M^r_{p,\phi_i}\rightarrow M^r_{p,\Phi}\rightarrow M^r_{p,\phi_i}.$$

5) The face and degeneration maps induce  maps between the corresponding modules $I^r_{p,\Phi}(w), I^r_{p,\phi_i}(w)$.

\end{definition}

On the other hand by Remark \ref{iubi} we have a natural action 
$$\Sigma_n \times \mathcal O(J^r_{p,\Phi}(B_1(N)))\rightarrow \mathcal O(J^r_{p,\Phi}(B_1(N))),\ \ 
(\sigma,u)\mapsto \sigma u.$$
 The following is trivially checked:

\begin{lemma}
The $\Sigma_n$-actions on $M^r_{p,\Phi}=\mathcal O(J^r_{p,\Phi}(B(1)))$ and $\mathcal O(J^r_{p,\Phi}(B_1(N)))$ are compatible in the sense that for every $\sigma\in \Sigma_n$ and every $f\in M^r_{p,\Phi}$ we have
$$\sigma (f^{B_1(N)})=(\sigma f)^{B_1(N)}.$$
\end{lemma}

\subsection{The case $n=2$, $\pi=p$}
In this subsection we assume $n=2, \pi=p$ and we may consider the forms
 \begin{equation}
 \label{neww}
 f^{\textup{jet}}_{\mu},  f^{\textup{jet}}_{\mu,\nu}\in M^2_{p,\phi_1,\phi_2},\ \ \mu,\nu\in \mathbb M_2^{2,+},\ \ \mu\neq \nu.
\end{equation}
 The forms (\ref{neww}) have  weights $-\phi_{\mu}-1$ and$ -\phi_{\mu}-\phi_{\nu}$, respectively.
 
 \bigskip

 Remark \ref{irir} can be amplified as follows.

\begin{remark} \label{lala}
It follows from \cite[Prop. 7.20, Cor. 8.84, and Rem 8.85]{Bu05}
that for  $i\in \{1,2\}$ we have the following equality in $M^2_{p,\phi_1,\phi_2}$:
\begin{equation}
\label{eqfi}
f^{\textup{jet}}_{ii,i}=p(f^{\textup{jet}}_{i})^{\phi_i}\end{equation}
It is also follows  from \cite[Prop. 8.55]{Bu05} 
that $f^{\textup{jet}}_{i},f^{\textup{jet}}_{ii}$ are non-zero in $M^2_{p,\phi_1,\phi_2}$; this also follows from Equation \ref{bgt} below.
Hence by Equation \ref{eqfi} it follows that $f^{\textup{jet}}_{ii,i}$ are non-zero in $M^2_{p,\phi_1,\phi_2}$. The fact that the the rest of the forms $f^{\textup{jet}}_{\mu}$ are non-zero was proved in Remark \ref{yuyi}; the fact that the forms $f^{\textup{jet}}_{\mu,\nu}$ (for $\mu\neq \nu$) are non-zero   is also true but more subtle and will be proved later; cf. Remark \ref{noz}.
\end{remark}

 \begin{theorem}\label{roo}
 The following relation holds in $M^2_{p,\phi_1,\phi_2}$:
  \begin{equation}
 f^{\textup{jet}}_{11}f^{\textup{jet}}_{22}f^{\textup{jet}}_{1,2}+ f^{\textup{jet}}_{2} f^{\textup{jet}}_{ 22}f^{\textup{jet}}_{ 11,1}-f^{\textup{jet}}_{ 11} f^{\textup{jet}}_{1}f^{\textup{jet}}_{ 22,2}-
  f^{\textup{jet}}_{1}f^{\textup{jet}}_{2}
f^{\textup{jet}}_{11,22}=0\end{equation}
\end{theorem}

{\it Proof}. 
 In order to check our relation over  a prolongation sequence $S^*$  it is enough  to check it after base change
 to a finite  \'{e}tale $S^0$-algebra. So we may assume $E/S^0$ has a  $\Gamma_1(N)$-level structure with $(N,p)=1$, $N\geq 4$.
 By functoriality it is then enough to check (\ref{defform}) for $S^0$ the ring of an affine open set of the modular curve $Y_1(N)$ over $R_{\pi}$ and $S^r:=J^r_{\pi,\Phi}(S^0)$. 
  Note that $f^{\textup{jet}}_{1}$ and $f^{\textup{jet}}_{2}$ evaluated at the universal curve over $S^0$  are nonzero mod 
  $p$ in $S^1$; cf. \cite[Lem. 4.4]{Bu05}
   (and also follows from Corollary \ref{corfrodo2} below).  As in Corollary \ref{cubic}
  the relation in our theorem holds with $(S^r)$ replaced by  $(T^r)$ where $T^0=S^0$ and  $T^r:=(S^r_{f^{\textup{jet}}_{1}f^{\textup{jet}}_{2}})^{\widehat{\ }}$ for $r\geq 1$; this is because 
  $T^r$ are integral domains in which $f^{\textup{jet}}_{1},f^{\textup{jet}}_{2}$ are invertible. Note now that  the homomorphisms
$S^r\rightarrow T^r$ are injective; this follows  because the homomorphisms
$S^r/p S^r\rightarrow T^r/p T^r$  are injective as $S^r/pS^r$ are integral domains and $f^{\textup{jet}}_{1}f^{\textup{jet}}_{2}$ is not zero in $S^r/p S^r$.
We conclude that the relation in our theorem holds for $(S^r)$.\qed

\bigskip

By an argument similar to the one in the proof of Theorem \ref{roo}, using the corresponding version of  Lemmas  \ref{quadraticc}, \ref{quadraticccc}, \ref{hotzu}, \ref{whyawake} over an arbitrary prolongation sequence we have:

\begin{theorem}\label{quadocc}
The following relations hold in $M^2_{p,\phi_1,\phi_2}$:
\begin{equation}\label{gogu1}
f^{\textup{jet}}_{12,1}(f^{\textup{jet}}_{1})^{\phi_1}-f^{\textup{jet}}_{11,1}(f^{\textup{jet}}_{2})^{\phi_1}=0,
\end{equation}
\begin{equation}\label{gogu2}
f^{\textup{jet}}_{12}(f^{\textup{jet}}_{1})^{\phi_1}-f^{\textup{jet}}_{11}(f^{\textup{jet}}_{2})^{\phi_1} -f^{\textup{jet}}_{1}(f^{\textup{jet}}_{1,2})^{\phi_1}=0.\end{equation}
\begin{equation}\label{gogu3}
 f^{\textup{jet}}_{12,21}f^{\textup{jet}}_{1}f^{\textup{jet}}_{2}-f^{\textup{jet}}_{2} f^{\textup{jet}}_{21} f^{\textup{jet}}_{12,1}+f^{\textup{jet}}_{1}f^{\textup{jet}}_{12}f^{\textup{jet}}_{21,2}-
 f^{\textup{jet}}_{12}f^{\textup{jet}}_{21}f^{\textup{jet}}_{1,2}=0,\end{equation}
 \begin{equation}\label{gogu4}
 f^{\textup{jet}}_{1}f^{\textup{jet}}_{11,2}-f^{\textup{jet}}_{2}f^{\textup{jet}}_{11,1}-f^{\textup{jet}}_{11}f^{\textup{jet}}_{1,2}=0,
 \end{equation}
 \begin{equation}\label{gogu5}
 f^{\textup{jet}}_{1}f^{\textup{jet}}_{11,12}-f^{\textup{jet}}_{12}f^{\textup{jet}}_{11,1}+f^{\textup{jet}}_{11}f^{\textup{jet}}_{12,1}=0,
 \end{equation}
 \begin{equation}\label{gogu6}
 f^{\textup{jet}}_{1}f^{\textup{jet}}_{12,2}-f^{\textup{jet}}_{2}f^{\textup{jet}}_{12,1}-f^{\textup{jet}}_{12}f^{\textup{jet}}_{1,2}=0,
 \end{equation}
 \begin{equation}\label{gogu7}
 f^{\textup{jet}}_{2}f^{\textup{jet}}_{11,21}-f^{\textup{jet}}_{21}f^{\textup{jet}}_{11,2}+f^{\textup{jet}}_{11}f^{\textup{jet}}_{21,2}=0.
 \end{equation}
 Moreover the relations obtained from the above relations by switching the indices $1$ and $2$ also hold.
 \end{theorem}

\subsection{$\delta$-Serre-Tate expansions}\label{DST}
In this subsection we assume $\pi=p$ and $n$ is arbitrary.

For the discussion of formal moduli 
in this paragraph we refer to \cite{Ka81}; we will the notation in \cite[Sect. 8.2]{Bu05}.
Throughout our discussion we fix   an ordinary elliptic curve $E_0$ over $k=R/pR$, a basis $b$ of the Tate module of $T_p(E_0)$, and  a basis $\check{b}$ of the Tate module of the dual $T_p(\check{E}_0)$. 
Let
$S^0_{\text{for}}=R[[T]]$,
 with $T$ a variable, and consider the Serre-Tate universal deformation space (identified with $\text{Spf}(R[[T]])$)  of  $E_0/k$ \cite{Ka81}. Let $E_{\textup{for}}/S^0_{\textup{for}}$ be the universal elliptic curve over $R[[T]]$. For all Noetherian complete local ring $(A,\mathfrak m(A))$ with residue field $k$  and every elliptic curve $E/A$ lifting $E_0/k$ we let $q(E)=q(E/A)\in 1+\mathfrak m(A)$ be  the {\bf Serre-Tate parameter} of $E$, i.e., the value of the Serre-Tate pairing 
 $q_{E/A}:T_p(E_0)\times T_p(\check{E}_0)\rightarrow 1+\mathfrak m(A)$ at the pair $(b,\check{b})$. Then $q(E)$ is the image  of $1+T$ via the classifying map $R[[T]]\rightarrow A$ corresponding to $E/A$.
  We denote by $\omega_{\textup{for}}$  the canonical $1$-form on $E_{\textup{for}}$ attached to $\check{b}$; cf. \cite[Eq. 8.67]{Bu05} and the discussion before it.
  
 Let now
 $$S^r_{\text{for}}:=R[[T]][\delta_{p,\mu} T\ |\ \mu\in \mathbb M_n^r]^{\widehat{\ }}.$$
 Clearly $S^*_{\textup{for}}=(S^r_{\textup{for}})$ is naturally an object of ${\bf Prol}_{p,\Phi}$.
 We define a ring homomorphism
 $$\mathcal E=\mathcal E_{E_0,b,\check{b}}:M^r_{p,\Phi}\rightarrow S^r_{\textup{for}}$$
 by attaching to 
 every $f\in M^r_{p,\Phi}$ its {\bf $\delta$-Serre-Tate expansion} given by
 $$\mathcal E(f):=f(E_{\textup{for}}/S^0_{\textup{for}},\omega_{\textup{for}},S^*_{\textup{for}})\in S^r_{\textup{for}}.$$
 Note that we have a natural action 
$$\Sigma_n \times S^r_{\textup{for}} \rightarrow S^r_{\textup{for}},\ \ 
(\sigma,\delta_{p,i} T)\mapsto \delta_{p,\sigma(i)}T.$$
 The following is trivially checked:

\begin{lemma}\label{compii}
The $\delta$-Serre-Tate expansion map $\mathcal E$ is compatible with the $\Sigma_n$-actions on $f\in M^r_{p,\Phi}$ and $S^r_{\textup{for}}$ in the sense that for every $\sigma\in \Sigma_n$ and every $f\in M^r_{p,\Phi}$ we have
$$\sigma (\mathcal E(f))=\mathcal E(\sigma f).$$
\end{lemma}

As in the arithmetic ODE case \cite{Bu00}, one has the following {\bf Serre-Tate expansion principle}:

\begin{theorem}\label{STE}
For every $w\in \mathbb Z_{\Phi}$ the homomorphism
$$\mathcal E:M^r_{p,\Phi}(w) \rightarrow S^r_{\textup{for}},\ \ f\mapsto \mathcal E(f)$$
is injective with torsion free cokernel.
\end{theorem}

{\it Proof}. Let $f\in M^r_{p,\Phi}(w)$ be such that $\mathcal E(f)=0$ (respectively $\mathcal E(f)\in pS^r_{\textup{for}}$); we want to show that for all $S^*$ and $E/S^0$ we have 
$f(E/S^0,\omega,S^*)=0$ (respectively $f(E/S^0,\omega,S^*)\in pS^r$). 
As in the proof of Theorem \ref{roo} it is enough to check this for $S^0$ the ring of an affine open set of the modular curve $Y_1(N)$ over $R$, $E/S^0$ the universal elliptic curve,  and $S^r:=J^r_{p,\Phi}(S^0)$. We conclude by the injectivity of the homomorphisms $S^r\rightarrow S^r_{\textup{for}}$ (respectively $S^r/pS^r\rightarrow S^r_{\textup{for}}/pS^r_{\textup{for}}$).
\qed

\begin{corollary}\label{intdomu}
The ring 
$\bigoplus_{w\in \mathbb Z_{\Phi}} M^r_{p,\Phi}(w)$
is an integral domain.
\end{corollary}

 \begin{remark}\label{plustheequality}
 Write 
 $$f^{\textup{jet}}_{i^r}=f^{\textup{jet}}_{i\ldots i}\ \ \ \text{with $i$ repeated $r$ times}.$$
   By \cite[Prop. 8.22, 8.61, 8.84]{Bu05} plus the equality $\epsilon=1$ in \cite[p. 236]{Bu05} we have
  \begin{equation}
  \label{luli}
  \mathcal E(f^{\textup{jet}}_{i^r})=c_r \Lambda_i^{r-1} \Psi_i,\end{equation}
  where $c_r\in R^{\times}$,
  $$\Lambda_i^{r-1}:=\sum_{j=0}^{r-1} p^j \phi_i^{r-1-j},$$
  and
  $$\Psi_i:=
  \frac{1}{p}\sum_{n\geq 1} (-1)^n\frac{p^n}{n}\left( \frac{\delta_{p,i}(1+T)}{(1+T)^p}
  \right)^n\in S^1_{\text{for}}=R[[T]][\delta_{p,i}T]^{\widehat{\ }}.$$
  In fact, by the theory over $\mathbb Z_p$ in \cite{Bu00, BaBu02} (instead of over $R$ as in \cite{Bu05})  one gets that 
  \begin{equation}
  c_r\in\mathbb Z_p^{\times}.\end{equation}
  Note that one has the following equality in $K[[T,\delta_{p,i}T,\ldots,\delta_{p,i}^rT]]$:
  \begin{equation}
  \label{lclc}
  \mathcal E(f^{\textup{jet}}_{i^r})=c_r \frac{1}{p}(\phi_i^r-p^r)\log(1+T).
  \end{equation}
  Now recall that the Serre-Tate expansion of $E_{p-1}$ satisfies
  \begin{equation}
  \mathcal E(E_{p-1})\equiv 1 \ \text{mod}\ \ p\ \ \text{in}\ \ R[[T]];\end{equation}
  cf., for instance, \cite[Prop. 8.57 and 8.59]{Bu05}. (In loc.cit $\overline{H}$
  denotes the Hasse invariant which is the reduction mod $p$ of $E_{p-1}$.) 
  Taking $\mathcal E$ in the congruence
  $$f^{\textup{jet}}_{ii}\equiv E_{p-1} \cdot (f^{\textup{jet}}_i)^p\ \ \text{mod}\ \ p\ \ \text{in}\ \  M^2_{p,\Phi},$$
 cf. Remark \ref{yuyi}, and using Fermat's Little Theorem we get
 $$c_2 \Psi_i\equiv c_1 \Psi_i^p\ \ \text{mod}\ \ p\ \ \text{in}\ \ R[[T]]$$
 hence
 \begin{equation}
 \label{c2c1p}
 c_2\equiv c_1\ \ \ \text{mod}\ \ p\ \ \text{in}\ \ \mathbb Z_p.
 \end{equation}
 We claim that we  have:
  \begin{equation}
  \label{c1equalsc2}
  c_1=c_2=c_3\end{equation}
  and hence we set, in what follows, 
  \begin{equation}
  \label{defofccc}
  c:=c_1=c_2=c_3.\end{equation}
  To check the  equality (\ref{c1equalsc2}) consider  a prolongation sequence $S^*$ with $S^r$ integral domains and an aribitrary elliptic curve $E/S^0$ such that $f_1\neq 0$.  By the relative case  of Corollary \ref{C1} (with $n=1)$  the $S^3$-module $\mathbf{X}_{p,\phi}^3(E/S^*)$ has rank $\leq 2$ so the the following $\delta$-characters of $E$ are $S^3$-linearly dependent:
  $$\psi_{11,1}, \phi_1 \psi_{11,1}, \psi_{111,1}.$$
  Therefore the Picard-Fuchs symbols of these $\delta$-characters,
  $$f_1\phi^2-f_{11}\phi+pf_1^{\phi},\ f_1^{\phi}\phi^3-f_{11}^{\phi}\phi^2+pf_1^{\phi^2}\phi,\ \ 
  f_1\phi^3-f_{111}\phi+f_{111,1}$$
  are $S^3$-linearly  dependent. We deduce that the determinant of the matrix
  $$\left(\begin{array}{ccc}
  0 & f_1 & -f_{11}\\
  \ & \ &  \\
  f_1^{\phi} & -f_{11}^{\phi} & pf_1^{\phi^2}\\
  \ & \ & \ \\
  f_1 & 0 & -f_{111}
  \end{array}\right)
  $$
  vanishes, hence
  $$pf_1^{\phi^2}f_1-f_{11}^{\phi}f_{11}+f_1^{\phi}f_{111}=0.$$
  Since $S^*$ was arbitrary (with $S^r$ integral domains) we get
  $$p(f_1^{\textup{jet}})^{\phi^2}f_1^{\textup{jet}}-(f_{11}^{\textup{jet}})^{\phi}f_{11}^{\textup{jet}}+(f_1^{\textup{jet}})^{\phi}f_{111}^{\textup{jet}}=0.$$
  Using the formulae \ref{luli} we get
  $$pc_1^2\Psi^{\phi^2}\Psi-c_2^2(\Psi^{\phi}+p\Psi)(\Psi^{\phi^2}+p\Psi^{\phi})+c_1c_3\Psi^{\phi}(\Psi^{\phi^2}+p\Psi^{\phi}+p^2\Psi)=0.
  $$
  Identifying the coefficients we get $c_1^2=c_2^2$ and $c_2^2=c_3c_1$. So $c_3=c_1$ and $c_2$ is either $c_1$ or $-c_1$.
  The equality $c_2=-c_1$ together with the equality (\ref{c2c1p}) leads to a congruence $c_1\equiv -c_1$ mod $p$ which is impossible for $p$ odd. We conclude that $c_1=c_2=c_3$.
  \end{remark}

 \begin{theorem}\label{mathcalEf}
For every weight $w\in \mathbb M^r_n$ of degree $\deg(w)=-2$ and every $f\in I^r_{p,\Phi}(w)$ we have that $\mathcal E(f)$ is a $K$-linear combination of elements in the set 
$$\{\Psi_i^{\phi_{\mu}}\ |\ \mu\in \mathbb M_n^{r-1},\ i\in \{1,\ldots, n\}\}.$$
\end{theorem}

{\it Proof}.
The proof is entirely similar to that of the statement made in \cite[par. after Prop. 8.30]{Bu05}.
Here is a  rough guide to the argument. By \cite[Prop. 8.22]{Bu05} there exists
a prime $l\neq p$ and an endomorphism $a:=u_0\in \text{End}(E_0)\subset \mathbb Z_p$  of degree $l$ such that
the quotient $\check{u}_0/u_0\in \mathbb Z_p$ is not a root of unity. By standard Serre-Tate theory  $u_0$ lifts to
 an isogeny of degree $l$ between $E_{R[[T]]}$ and the curve $E'_{R[[T]]}$ obtained from
 $E_{R[[T]]}$
 by base change via 
the homomorphism $R[[T]]\rightarrow R[[T]]$ given by $T\mapsto (1+T)^a-1$.
This forces the series $F:=\mathcal E(f)$ to satisfy the equation 
\begin{equation}
\label{mathcalEff}
F(\ldots, \delta_{p,\mu}((1+T)^a-1),\ldots)=a\cdot F(\ldots,\delta_{p,\mu}T,\ldots);\end{equation}
cf. \cite[Prop. 8.30]{Bu05}. We conclude exactly as in \cite[Prop. 4.36]{Bu05} that $F$ is a $K$-linear combination of series $\Psi^{\phi_{\mu}}$.
\qed

\bigskip

By Theorem \ref{mathcalEf} and the Serre-Tate expansion principle in Theorem \ref{STE} we get:

\begin{corollary}
For every $w\in \mathbb Z_{\Phi}^r$ of degree $\deg(w)=-2$ 
the $R$-module $I^r_{p,\Phi}(w)$ has rank at most  $$D(n,r)-1=n+\ldots+n^r.$$
\end{corollary}

\begin{definition}
Let $w\in \mathbb Z^r_{\Phi}$ be a weight of degree $\deg(w)=-2$ and $f\in I^r_{p,\Phi}(w)$. Write
$$\mathcal E(f)=\sum_{i=1}^n \sum_{\mu\in \mathbb M_n^{r-1}} \lambda_{i,\mu} \Psi_i^{\phi_{\mu}},\ \ \lambda_{i,\mu}\in K_{\pi};$$
cf. Theorem \ref{mathcalEf}. Note that
$$
\Psi_i^{\phi_{\mu}}=\frac{1}{p}\phi_i^{\mu}(\phi_i-p)\log(1+T).$$
Define the {\bf symbol} $\theta(f)\in K_{\pi,\Phi}$ by
$$\theta(f)=\sum_{i=1}^n \sum_{\mu\in \mathbb M_n^{r-1}} \lambda_{i,\mu} \phi_{\mu}(\phi_i-p).$$
Hence the following holds in the ring $S^r_{\textup{for}}$:
$$\mathcal E(f)=\frac{1}{p}\theta(f) \log(1+T).$$
\end{definition}

\begin{remark}\label{remfrodo1}
 For $n=2$, using  Equation \ref{luli} and Remarks \ref{lala} and \ref{plustheequality}, we have:
   \begin{equation}
   \label{bgt}
   \begin{array}{rcl}
\mathcal E(f^{\textup{jet}}_{i}) & = & c\Psi_i,\\
\mathcal E(f^{\textup{jet}}_{ii}) & = & c(\Psi_i^{\phi_i}+p\Psi_i),\\
\mathcal E(f^{\textup{jet}}_{ii,i}) & = & pc\Psi_i^{\phi_i}.
\end{array}\end{equation}
In particular the symbols of these forms are:
\begin{equation}
\begin{array}{rcl}
\theta(f^{\textup{jet}}_{i}) & = & c(\phi_i-p),\\
\theta (f^{\textup{jet}}_{ii}) & = & c(\phi^2_i-p^2),\\
\theta(f^{\textup{jet}}_{ii,i}) & = & pc(\phi^2_i-p^2\phi_i).
\end{array}\end{equation}
\end{remark}

For the rest of this subsection we assume that $n=2$.

 \begin{theorem}\label{nonzzero}
 We have  the following Serre-Tate expansions: 
  $$
 \begin{array}{rcl}
 \mathcal E(f^{\textup{jet}}_{1,2}) & =  &pc(\Psi_1-\Psi_2),\\
\mathcal E(f^{\textup{jet}}_{11,22}) & =  & p^2c(\Psi_1^{\phi_1}+p\Psi_1-\Psi_2^{\phi_2}-p\Psi_2),\\
\mathcal E(f^{\textup{jet}}_{12,1}) & = & pc\Psi_2^{\phi_1},\\
 \mathcal E(f^{\textup{jet}}_{12}) & = & c(\Psi_2^{\phi_1}+p\Psi_1),\\
\mathcal E(f^{\textup{jet}}_{12,21}) & = & p^2c(\Psi_2^{\phi_1}+p\Psi_1-\Psi_1^{\phi_2}-p\Psi_2),\\
\mathcal E(f^{\textup{jet}}_{11,2}) & = & pc(\Psi_1^{\phi_1}+p\Psi_1-p\Psi_2),\\
\mathcal E(f^{\textup{jet}}_{11,12}) & = & p^2c(\Psi_1^{\phi_1}-\Psi_2^{\phi_1}),\\
\mathcal E(f^{\textup{jet}}_{12,2}) & = & pc(\Psi_2^{\phi_1}+p\Psi_1-p\Psi_2),\\
\mathcal E(f^{\textup{jet}}_{11,21}) & = & p^2c(\Psi_1^{\phi_1}-\Psi_1^{\phi_2}+p\Psi_1-p\Psi_2).
 \end{array}
 $$
 Moreover the relations obtained from the above relations by switching the indices $1$ and $2$ also hold.
 \end{theorem}

{\it Proof}. 
We begin by proving the first $2$ equalities.

Set
$G_1:=\mathcal E(f^{\textup{jet}}_{1,2}), \ \ G_2:=\mathcal E(f^{\textup{jet}}_{11,22})$.
Taking $\mathcal E$ in the cubic relation of Theorem 
\ref{roo}  and using the formulae (\ref{bgt}) we get
\begin{equation}
\label{bbb}
\begin{array}{l}
 c^2(\Psi_1^{\phi_1}+p\Psi_1) (\Psi_2^{\phi_2}+p\Psi_2) G_1+
 pc^3\Psi_2 (\Psi_2^{\phi_2}+p\Psi_2) \Psi_1^{\phi_1}=\\
 \ \\
 =pc^3(\Psi_1^{\phi_1}+p\Psi_1) \Psi_1 \Psi_2^{\phi_2}+
 c^2\Psi_1\Psi_2 G_2.\end{array}
\end{equation}
By Theorem \ref{mathcalEf} we can write
$$G_1=\gamma_1\Psi_1+\gamma_2\Psi_2$$
$$G_2=\sum_{\mu} \gamma'_{\mu}\Psi_1^{\phi_{\mu}}+ \sum_{\mu} \gamma''_{\mu}\Psi_2^{\phi_{\mu}}$$
with $\gamma_i,\gamma'_{\mu},\gamma''_{\mu}\in K$. 
Plugging these expressions    into the equation (\ref{bbb})
and using the fact that the set 
$$\{\Psi_i^{\phi_{\mu}}\ |\ i=1,2;\ \mu\in \mathbb M_2\}$$ is algebraically independent over $K$,
 we see 
 that there is a unique tuple $(\gamma_i,\gamma'_{\mu},\gamma''_{\mu})$ satisfying the resulting equation, which leads to the desired formulae for $\mathcal E(f^{\textup{jet}}_{1,2})$
 and  $\mathcal E(f^{\textup{jet}}_{11,22})$.
Taking $\mathcal E$ in Equations (\ref{gogu1}), (\ref{gogu2}), (\ref{gogu3}), (\ref{gogu4}), (\ref{gogu5}), (\ref{gogu6}), (\ref{gogu7}) in this order 
 we get the desired formulae for $\mathcal E(f^{\textup{jet}}_{12,1})$, $\mathcal E(f^{\textup{jet}}_{12})$,
 $\mathcal E(f^{\textup{jet}}_{12,21})$,
$\mathcal E(f^{\textup{jet}}_{11,2})$, $\mathcal E(f^{\textup{jet}}_{11,12})$, $\mathcal E(f^{\textup{jet}}_{12,2})$,  $\mathcal E(f^{\textup{jet}}_{11,21})$, respectively.
   \qed
   
   \begin{corollary}\label{ila}
   The following equalities hold:
   $$f^{\textup{jet}}_{12,1}=p(f^{\textup{jet}}_2)^{\phi_1};$$
    $$f^{\textup{jet}}_{11,12}=p(f^{\textup{jet}}_{1,2})^{\phi_1};$$
     $$f^{\textup{jet}}_{1,2}f^{\textup{jet}}_{12,21}-p^2(f^{\textup{jet}}_{2})^{\phi_1} (f^{\textup{jet}}_1)^{\phi_2}+f^{\textup{jet}}_{21,1}f^{\textup{jet}}_{12,2}=0.
 $$
  Moreover the equalities  obtained from the above equalities by switching the indices $1$ and $2$ also hold.

   \end{corollary}
   
   {\it Proof}.
   The forms $f^{\textup{jet}}_{12,1}$ and $p(f^{\textup{jet}}_2)^{\phi_1}$ have the same weight
   and the same Serre-Tate expansion so by the Serre-Tate expansion principle they must be equal. The same argument holds for the other equalities.
   \qed

\begin{corollary}\label{corfrodo2}
 Let $\mu,\nu\in \mathbb M_2^{2,+}$, $\mu\neq \nu$, of length $r,s$ respectively,
with $r\geq s$.

1) The following non-divisibility, respectively divisibility conditions hold:
\begin{equation}
\begin{array}{rcl}
f^{\textup{jet}}_{\mu} & \in & I^r_{p,\phi_1,\phi_2}(-\phi_{\mu}-1)\setminus p I^r_{p,\phi_1,\phi_2}(-\phi_{\mu}-1);\\
f^{\textup{jet}}_{\mu,\nu} & \in & p^s I^r_{p,\phi_1,\phi_2}(-\phi_{\mu}-\phi_{\nu}).\end{array}\end{equation}

2) The symbols of $f^{\textup{jet}}_{\mu},f^{\textup{jet}}_{\mu,\nu}$ are given by 
\begin{equation}
\label{tommy}
\begin{array}{rcl}
\theta(f^{\textup{jet}}_{\mu}) & = & c(\phi_{\mu}-p^r),\\
\theta(f^{\textup{jet}}_{\mu,\nu}) & = & c(p^s\phi_{\mu}-p^r\phi_{\nu}).
\end{array}\end{equation}
\end{corollary}

{\it Proof}.
Part 1 follows from Theorem \ref{nonzzero} plus the torsion freeness part of the Serre-Tate expansion principle. (The first equation in Part 1 also follows from Remark \ref{yuyi}.)
Part 2 follows from a direct computation using our definitions.
\qed

\begin{remark}\label{noz}
In particular we have that the forms $f^{\textup{jet}}_{\mu}, f^{\textup{jet}}_{\mu,\nu}$ for $\mu,\nu\in \mathbb M_2^{2,+}$, $\mu\neq \nu$ are non-zero in $M^2_{p,\phi_1,\phi_2}$.
 Note that all these forms, with the exception of the ``ODE forms" $f^{\textup{jet}}_{i}, f^{\textup{jet}}_{ii}, f^{\textup{jet}}_{ii,i}$, are  ``genuinely PDE" in the sense that they are not sums of products of ODE forms and their images by $\phi_1,\phi_2$
 (as one can see by looking at weights).\end{remark}
 
 \begin{remark}\label{qq1}
 We expect that Corollary \ref{corfrodo2} remains true for $\mathbb M_2^{2,+}$ replaced by $\mathbb M^{r,+}_n$ for $n$ and $r$ arbitrary. Our method of proof for $\mathbb M_2^{2,+}$ was based on ``solving" a rather complicated system of quadratic and cubic equations satisfied by the arithmetic Kodaira-Spencer classes; extending this method to the case of 
 $\mathbb M^{r,+}_n$ seems tedious. It would be interesting to find another approach to the proof of Corollary \ref{corfrodo2} that easily extends to arbitrary $n$ and $r$. 
 \end{remark}

\begin{remark}
Let $\sigma\in \Sigma_2$ be the transposition $(12)$. Then
by Lemma \ref{compii} and by the $\delta$-Serre-Tate expansion principle   we have:
$$\sigma f^{\textup{jet}}_{\mu} =f^{\textup{jet}}_{\sigma{\mu}},\ \ 
\sigma f^{\textup{jet}}_{\mu,\nu}= f^{\textup{jet}}_{\sigma{\mu},\sigma \nu}.$$
\end{remark}

\begin{theorem}\label{mult1}
The form $f^{\textup{jet}}_{1,2}$ is a basis modulo torsion of the $R$-module 
$$I^1_{p,\phi_1,\phi_2}(-\phi_1-\phi_2)$$ of  isogeny covariant $\delta_p$-modular forms of order $1$ and weight $-\phi_1-\phi_2$.
\end{theorem}

{\it Proof}.
Assume this $R$-module has rank $\geq 2$ and seek a contradiction. Let $f$ be an element of this module that is $R$-linearly independent from $\psi_{p,1,2}$. By Theorem \ref{mathcalEf} we have $\mathcal E(f)=\gamma_1\Psi_1+\gamma_2\Psi_2$ for some $\gamma_1,\gamma_2\in K$. 
By the Serre-Tate expansion principle (Theorem \ref{STE}) $\mathcal E(f)$ is $K$-linearly independent from $\mathcal E(f^{\textup{jet}}_{1,2})$. By Theorem \ref{nonzzero} the latter equals
$p\Psi_1-p\Psi_2$. 
Hence $\gamma_1+\gamma_2\neq 0$.
 Consider  the degeneration
morphism
$d:M^1_{p,\phi_1,\phi_2}\rightarrow M^1_{p,\phi}$
where $\phi\in \mathfrak F(K^{\text{alg}}/\mathbb Q_p)$ and let $\Psi$ be the series corresponding to $f$.
Then 
$d(f)\in I^1_{p,\phi}(-2\phi)$. By \cite[Thm. 8.83, Part 2]{Bu05} we have
$I^1_{p,\phi}(-2\phi)=0$; hence $f=0$. But on the other hand $\mathcal E(f)=(\gamma_1+\gamma_2)\Psi\neq 0$, a contradiction. 
\qed

\subsection{$\delta$-period maps}

In this subsection we revert to the case of an arbitrary $n$.
For the next result let us consider 
a weight $w\in \mathbb Z_{\Phi}^r$ and the set $P^r(w)$ of all polynomials $F(\ldots,y_{\eta,\mu},\ldots,y_{\eta,\mu,\nu},\ldots)$  with $R$-coefficients in the variables $y_{\eta,\mu},y_{\eta,\mu,\nu}$, 
$\eta,\mu,\nu\in \mathbb M_n$, that are homogeneous of weight $w$ when $y_{\eta,\mu},y_{\eta,\mu,\nu}$ are given weights $-\phi_{\eta}-\phi_{\eta\mu}$ and $ -\phi_{\eta\mu}-\phi_{\eta\nu}$, respectively. Moreover we denote by 
$KSI^r_{p,\Phi}(w)$ the $R$-submodule of $M^r_{p,\Phi}$ of all elements $f$ of the form
\begin{equation}
\label{fequal}
f=F(\ldots,\phi_{\eta}f^{\textup{jet}}_{\mu},\ldots,\phi_{\eta}f^{\textup{jet}}_{\mu,\nu},\ldots),\ \ \ F\in P^r(w).\end{equation}
Clearly the elements of $KSI^r_{p,\Phi}(w)$ have weight $w$ and are isogeny covariant i.e., 
\begin{equation}
\label{ksi}
KSI^r_{p,\Phi}(w)\subset I^r_{p,\Phi}(w).\end{equation}
Alternatively $KSI^r_{p,\Phi}(w)$ is the $R$-span of all the products of the form
\begin{equation}
\prod_{\mu} (f^{\textup{jet}}_{\mu})^{w_{\mu}}\cdot \prod_{\mu,\nu}(f^{\textup{jet}}_{\mu,\nu})^{w_{\mu,\nu}}
\end{equation}
where $\mu,\nu$ run through $\mathbb M_n$, $w_{\mu},w_{\mu,\nu}\in \mathbb Z_{\Phi}$ are $\geq 0$ and
$$\sum_{\mu} w_{\mu}(1+\phi_{\mu})+\sum_{\mu,\nu}w_{\mu,\nu}(\phi_{\mu}+\phi_{\nu})=-w.$$
One should view the ring
$$KSI^r_{p,\Phi}:=\bigoplus_{w\in \mathbb Z^r_{\Phi}} KSI^r_{p,\Phi}(w)\subset \bigoplus_{w\in \mathbb Z^r_{\Phi}} I^r_{p,\Phi}(w)$$
as the ``$\Phi$-stable $R$-subalgebra generated by the arithmetic Kodaira-Spencer classes"; whence our notation.
By the way we do not know if the inclusion (\ref{ksi}) is an equality (or an equality modulo torsion). 
  
\begin{definition}\label{defpermap}
 Let $B=B_1(N)$, $N\geq 4$, $\pi=p$ and assume as before that the reduction mod $p$ of $X\subset Y_1(N)$  is non-empty. Fix an order $r\geq 1$ and  weight $w\in \mathbb Z_{\Phi}^r$
 and consider a basis
 $$f_{(0)},\ldots,f_{(N_w)}$$ 
 of the $R$-module $KSI^r_{p,\Phi}(w)$ (so $N_w+1$ is the rank of this module).
 Consider the map
 $$\mathfrak P_w^B:B(R^{\textup{alg}})\rightarrow 
 \mathbb A^{N_w}(K^{\textup{alg}})=(K^{\textup{alg}})^{N_w}$$
 defined by
 $$\mathfrak P_w^B(P):=(((f_{(0)})^B)^{\textup{alg}}(P),\ldots,((f_{(N_w)})^B)^{\textup{alg}}(P))$$
 and the induced map to the set of points of the  projective space:
 $$\mathfrak p_w^B:B(R^{\textup{alg}})_w^{\textup{ss}}:=B(R^{\textup{alg}})\setminus ((\mathfrak P_w^B)^{-1}(0))\rightarrow \mathbb P^{N_w}(K^{\textup{alg}})=
 \mathbb P^{N_w}(R^{\textup{alg}}).$$
 The ``ss" superscript stands for ``semistable" (in analogy with geometric invariant theory; here instead of group actions we have an action of the Hecke correspondences and isogeny covariant forms are viewed as analogues of invariant sections of line bundles in geometric invariant theory).
 Assuming, for a moment, that the universal elliptic curve over $X$ possesses a global invertible relative $1$-form $\omega$ we get an induced section $\sigma:X\rightarrow B$ of the projection
 $B\rightarrow X$ and hence an induced map
 \begin{equation}
 \label{ainceputt}
 \mathfrak p_w:=\mathfrak p_w^X:X(R^{\textup{alg}})_w^{\textup{ss}}:=\sigma^{-1}(B(R^{\textup{alg}})_w^{\textup{ss}})
 \stackrel{\sigma}{\longrightarrow} B(R^{\textup{alg}})_w^{\textup{ss}}\stackrel{\mathfrak P_w^B}{\longrightarrow} 
 \mathbb P^{N_w}(R^{\textup{alg}})\end{equation}
 The  map (\ref{ainceputt}) 
 does not depend on the choice of $\omega$ (due to the fact that $f_{(i)}$ have the same weight) and hence this map   is well defined for any $X$ (not only for $X$ such that an $\omega$ as above exists) and only depends on $X$ and $w$ (up to a projective transformation). The map (\ref{ainceputt}) 
 will be referred to as the {\bf $\delta$-period map} (of weight $w$) and the set
 $X(R^{\textup{alg}})_w^{\textup{ss}}$ will be referred to as the set of {\bf semistable} points (relative to $w$).
 \end{definition}
 
 Note that for $w,w'\in \mathbb M_n$ the composition
 $$X(R^{\textup{alg}})_w^{\textup{ss}}\cap X(R^{\textup{alg}})_{w'}^{\textup{ss}}
 \stackrel{\mathfrak p_w\times \mathfrak p_{w'}}{\longrightarrow}
 \mathbb P^{N_w}(R^{\textup{alg}})
\times \mathbb P^{N_{w'}}(R^{\textup{alg}})\stackrel{\textup{Segre}}{\longrightarrow}
 \mathbb P^{N_wN_{w'}+N_w+N_{w'}}(R^{\textup{alg}})
$$
is obtained by composing the map
$$ \mathfrak p_{w+w'}:X(R^{\textup{alg}})_w^{\textup{ss}}\rightarrow  \mathbb P^{N_{w+w'}}(R^{\textup{alg}})
 $$ 
 with a projection. And similarly $\mathfrak p_w$  is obtained from any of the maps
 $\mathfrak p_{\phi_i w}$ by composing with  a projection followed by the $\phi_i$ map.

 \begin{theorem}\label{thmpermap}
 The $\delta$-period map 
 $$ \mathfrak p_w:X(R^{\textup{alg}})_w^{\textup{ss}}\rightarrow  \mathbb P^{N_w}(R^{\textup{alg}})
 $$ 
 is constant on prime to $p$ isogeny classes in the following sense: for every two points $P,Q\in X(R^{\textup{alg}})_w^{\textup{ss}}$ if there exists an isogeny of degree prime to $p$ between the elliptic curves
 over $R^{\textup{alg}}$ corresponding to $P$ and $Q$ then 
 $$ \mathfrak p_w(P)= \mathfrak p_w(Q).$$
 \end{theorem}
 
 {\it Proof}.
 By the density of prime to $p$ ordinary isogeny classes \cite{Cha95} we may assume $X$ is such that the universal elliptic curve over $X$ possesses an invertible $1$-form $\omega$.
 Assume $P$ and $Q$ are in the statement of the theorem and let $u:E_P\rightarrow E_Q$
 be an isogeny of degree prime to $p$ between the corresponding elliptic curves over $R^{\textup{alg}}$. 
 Let $\omega_P$ and $\omega_Q$ be the $1$-forms on $E_P$ and $E_Q$ induced by $\omega$, respectively.
 We may view both elliptic curves and the isogeny as being defined over some $R_{\pi}$. 
 For each $i$ let $f^X_{(i)}$ be the composition
 $$f^X_{(i)}:J^r_{\pi,\Phi}(X)\stackrel{J^r(\sigma)}{\longrightarrow} J^r_{\pi,\Phi}(B)
 \stackrel{f^B_{(i)}}{\longrightarrow} \widehat{\mathbb A^1}.$$
 Write for simplicity $f_{(i)}(P):=((f_{(i)})^X)^{\textup{alg}}(P)$ for $P\in X(R^{\textup{alg}})$.
 By isogeny covariance, the weight condition,  and the equalities (\ref{2use}), (\ref{oouu}), (\ref{zipo}), (\ref{zipopo})
 we get that 
 $$f_{(i)}(P)=\lambda \cdot f_{(i)}(Q),\ \ i\in \{0,\ldots,N_w\}$$
 for some  $\lambda\in R_{\pi}^{\times}$ depending on $E_P,E_Q,\omega_P,\omega_Q,u$ but not on $i$. This implies that $\mathfrak p_w(P)=\mathfrak p_w(Q)$.
 \qed

\begin{example}\label{opopop}
Assume $n=r=2$. Then one can explicitly describe the algebra $KSI^2_{p,\Phi}\otimes_R K$ as follows. 
Consider the ring of polynomials
$$\mathcal P:=K[\Psi_i,\Psi_i^{\phi_j}\  |\ i,j\in \{1,2\}],$$
where $\Psi_i^{\phi_j}$ are viewed as variables and view this ring as graded by giving the variables the degree $1$. We denote by $\mathcal P(i)$ the graded piece of degree $i\in \mathbb Z_{\geq 0}$, we set
$$t_0:=\frac{\Psi_2}{\Psi_1},\ \ t_{ij}:=\frac{\Psi_i^{\phi_j}}{\Psi_1}\ \ i,j\in \{1,2\},$$
 and we consider the field of ``homogeneous fractions of degree $0$:"
$$
\mathcal F:=\{\frac{F}{G}\ |\ F,G\in \mathcal P(i),\ i\geq 1,\ G\neq 0\}=K(t_0,t_{11},t_{12},t_{21},t_{22});
$$
this is a field of rational functions over $K$ in five variables.
Consider new variables $\Lambda^{\phi_{\mu}}$ for $\mu\in \mathbb M_2^2$ and the algebra of polynomials
$$\mathcal P[\Lambda^{\phi_{\mu}}\ |\ \mu\in \mathbb M^2_2].$$
Inside the latter  consider the $K$-subalgebra $\mathcal K\mathcal S \mathcal I^2$ 
generated by all the elements of the form
$$\mathcal E(f_{\mu}^{\textup{jet}})\Lambda^{1+\phi_{\mu}},\ \ \mathcal E(f_{\mu,\nu}^{\textup{jet}})\Lambda^{\phi_{\mu}+\phi_{\nu}},$$
which have been explicitly computed in Remark \ref{remfrodo1} and Theorem \ref{nonzzero}.
We then write
$$\mathcal K\mathcal S \mathcal I^2=\bigoplus_{w\in \mathbb Z_{\Phi}^2} \mathcal K\mathcal S \mathcal I^2(w)\Lambda^w,\ \ \ \mathcal K\mathcal S \mathcal I^2(w)\subset \mathcal P(\frac{\deg(w)}{2})
$$
(which makes sense since $\mathcal K\mathcal S \mathcal I^2(w)=0$ for $\deg(w)$ odd).
On the other hand,  by the Serre-Tate expansion principle we get isomorphisms of $K$-vector spaces
$$KSI^2_{p,\Phi}\otimes_R K\simeq K\mathcal S \mathcal I^2(w).$$
One may consider the subfield 
$\mathcal F_w\subset \mathcal F$
generated by the set 
$$\left\{\frac{F}{G}\ |\ F,G\in \mathcal K\mathcal S \mathcal I^2(w),\ m\geq 0,\ G\neq 0\right\}.$$
This field could be intuitively interpreted as the ``field of rational functions on the image
of the $\delta$-period map $\mathfrak p_w$." If $\mathcal K\mathcal S \mathcal I^2(w)\neq 0$
and $\mathcal K\mathcal S \mathcal I^2(w')\neq 0$ then, clearly, $\mathcal F_w$ and $\mathcal F_{w'}$ are subfields of $\mathcal F_{w+w'}$. 

As a special case of the above discussion let $w=-(1+\phi_1+\phi_2+\phi_1^2)$. Then $KSI^r_{p,\Phi}(w_1)$ is spanned by
$$
 f^{\textup{jet}}_{1}f^{\textup{jet}}_{11,2},\ \ f^{\textup{jet}}_{2}f^{\textup{jet}}_{11,1},\ \ f^{\textup{jet}}_{11}f^{\textup{jet}}_{1,2},
 $$
 and has rank $2$ 
 with the ``only relation" (\ref{gogu4}). The fact that this module has rank $2$ follows by looking at the Serre-Tate expansions of the generators; cf. Theorem \ref{nonzzero}. Note that
 \begin{equation}
 \label{ggt1}
 \tau:=\frac{\Psi_1 (\Psi_1^{\phi_1}+p\Psi_1-p\Psi_2)}{\Psi_2\Psi_1^{\phi_1}}
 = \frac{t_{11}+p-pt_0}{t_0t_{11}}\in \mathcal F_{w}.
 \end{equation}
 Similarly
  let $w'=-(1+\phi_1+\phi_2+\phi_1\phi_2)$. Then $KSI^r_{p,\Phi}(w')$ is spanned by
$$
 f^{\textup{jet}}_{1}f^{\textup{jet}}_{12,2},\ \ f^{\textup{jet}}_{2}f^{\textup{jet}}_{12,1},\ \ f^{\textup{jet}}_{12}f^{\textup{jet}}_{1,2},
 $$
 and has rank $2$ 
 with the ``only relation" (\ref{gogu6}).  Note that 
 \begin{equation}
 \label{ggt2}
 \tau':=\frac{\Psi_1 (\Psi_2^{\phi_1}+p\Psi_1-p\Psi_2)}{\Psi_2\Psi_2^{\phi_1}}
 = \frac{t_{21}+p-pt_0}{t_0t_{21}}\in \mathcal F_{w'}.
 \end{equation}
 One has a similar discussion for the weights $w'',w'''$ obtained by switching the indices $1$ and $2$ in the weights $w,w'$ and we have corresponding elements
 \begin{equation}
 \label{ggt3}
 \tau'':=\frac{\Psi_2 (\Psi_2^{\phi_2}+p\Psi_2-p\Psi_1)}{\Psi_1\Psi_2^{\phi_2}}
 = \frac{t_0(t_{22}+pt_0-p)}{t_{22}}\in \mathcal F_{w''}.
 \end{equation}
  \begin{equation}
 \label{ggt4}
 \tau''':=\frac{\Psi_2 (\Psi_1^{\phi_2}+p\Psi_2-p\Psi_1)}{\Psi_1\Psi_1^{\phi_2}}
 = \frac{t_0(t_{12}+pt_0-p)}{t_{12}}\in \mathcal F_{w'''}.
 \end{equation}
It is trivial to check that
$$K(t_0,\tau,\tau',\tau'',\tau''')=K(t_0,t_{11}, t_{12}, t_{21},t_{22}).$$
 So in particular 
 the field $K(\tau,\tau',\tau'',\tau''')$ is a field of rational functions in $4$ variables.
 The union $\bigcup_w\mathcal F_w$ for all $w$'s of order $2$ is a subfield of $\mathcal F$ so it is finitely generated, hence equal to one of the fields $\mathcal F_{w_0}$. The field $\mathcal F_{w_0}$ has then transcendence degree $4$ or $5$ over $K$. It would be interesting to compute this field and in particular to compute its transcendence degree over $K$. The heuristic we are employing is that the order $2$ partial $p$-jet space of $Y_1(N)$ has relative dimension $7=|\mathbb M^2_2|$ over $K$ and the field $\mathcal F_{w_0}$ (which has transcendence degree $4$ or $5$ over $K$) plays the role of ``field of rational functions" of the quotient ``of order $2$" of $Y_1(N)$ by the actiuon of the Hecke correspondences. The difference between the dimensions (which is either $3$ or $2$) should play the role of ``dimension of the fibers" of the ``order $2$" projection from 
 $Y_1(N)$ to this ``quotient". All of this can be made rigorous in a ``partial $\delta$-geometry" 
 which is a PDE analogue of the ODE $\delta$-geometry in \cite{Bu05}; we will not pursue this in the present paper. Suffices to say that, for every $P\in X(R^{\textup{alg}})_{w_0}^{\textup{ss}}$,  we get in this way, a ``large" system of arithmetic differential equations of order $2$ satisfied by the points $P'\in X(R^{\textup{alg}})_{w_0}^{\textup{ss}}$ in the prime to $p$ isogeny class of $P$. Roughly speaking these equations have the form
 $$f_{(i)}(P)f_{(j)}(P')- f_{(j)}(P)
 f_{(i)}(P')=0,$$
 where we are using the notation in the proof of Theorem \ref{thmpermap}.
 Note that the analogous ODE $\delta$-period maps of minimum order (implicit in \cite[Sect. 8.6]{Bu05}) have order $3$ rather than $2$; this is an instance of the principle, already encountered in this paper in the case of $\delta$-characters, that replacing ODEs by PDEs reduces the order of the interesting arithmetic differential equations.
 \end{example}

We end by discussing the maps on points on the ordinary locus.

\begin{proposition}\label{zaza} Let $B=B_1(N)$, $N\geq 4$. Assume the reduction mod $p$ of $X$ is contained in the ordinary locus of the modular curve. Then the following hold:

1) 
For all distinct $\mu,\nu\in \mathbb M_n$ and all $\eta\in \mathbb M_n$ the maps
 $$((\phi_{\eta}f^{\textup{jet}}_{\mu})^B)^{\textup{alg}},((\phi_{\eta}f^{\textup{jet}}_{\mu,\nu})^B)^{\textup{alg}}:B(R^{\textup{alg}})\rightarrow K^{\textup{alg}}
 $$
 (cf. Theorem \ref{theyareover})
 extend to continuous maps
  $$((\phi_{\eta}f^{\textup{jet}}_{\mu})^B)^{\mathbb C_p},((\phi_{\eta}f^{\textup{jet}}_{\mu,\nu})^B)^{\mathbb C_p}:B(\mathbb C_p^{\circ})\rightarrow \mathbb C_p.
 $$
 
2) Assume $\Phi$ is monomially independent  in $\mathfrak G(K^{\textup{alg}}/\mathbb Q_p)$. Then for every $w\in \mathbb Z_{\Phi}$ the $R$-module homomorphism
$$KSI^r_{p,\Phi}(w)\rightarrow \textup{Fun}(B(R^{\textup{alg}}),K^{\textup{alg}}),\ \ \ 
f\mapsto (f^B)^{\textup{alg}}$$
is injective.
   \end{proposition}

{\it Proof}.
Part 1 follows exactly as in \cite[Prop. 5.17]{BM20}. Here is a guide to the argument. 
Let  $\text{pr}:B\rightarrow X$ be the projection and let $P\in B(\mathbb C_p^{\circ})$. 
Since $\text{pr}(P)\in X(\mathbb C_p^{\circ})$ corresponds to an elliptic curve with ordinary reduction $E_0$ we have at our disposal the $\delta$-Serre-Tate expansion at $E_0$.
Similar to  \cite[Eqn. 5.29]{BM20}
 one gets that there exists  a $p$-adic ball $\mathbb B\subset B(\mathbb C_p^{\circ})$ containing $P$ 
and a nowhere vanishing analytic map  $u:\mathbb B\rightarrow \mathbb C_p$ 
with the following property: for all $\mu$ and $\nu$,  $((f^{\textup{jet}}_{\mu})^B)^{\textup{alg}}$ and $((f^{\textup{jet}}_{\mu,\nu})^B)^{\textup{alg}}$ may be expressed on $\mathbb B\cap B(R^{\textup{alg}})$ as:
\begin{equation}
\label{aragorn1}
((f^{\textup{jet}}_{\mu})^B)^{\textup{alg}}(P)=u(P)^{-1-\phi_{\mu}}\cdot \theta(f^{\textup{jet}}_{\mu})^{\textup{alg}}[\log (q(E_P))]\end{equation}
\begin{equation}
\label{aragorn2}
((f^{\textup{jet}}_{\mu,\nu})^B)^{\textup{alg}}(P)=u(P)^{-\phi_{\mu}-\phi_{\nu}}\cdot \theta(f^{\textup{jet}}_{\mu,\nu})^{\textup{alg}}[\log  (q(E_P))]\end{equation}
 where $\log$ is the usual logarithm defined on the open ball of $\mathbb C_p$ of radius $1$, $E_P$ is the elliptic curve corresponding to $P$, and $\theta(f^{\textup{jet}}_{\mu})^{\textup{alg}},\theta(f^{\textup{jet}}_{\mu,\nu})^{\textup{alg}}$ are, as usual, the induced maps $K^{\textup{alg}}\rightarrow K^{\textup{alg}}$. Cf. loc. cit. for details on this representation.  Extension by continuity then follows. A similar argument works for the above maps composed with $\phi_{\eta}$.
 
 By the way, if $P$ corresponds to a pair $(E_P,\omega_P)$ then by the construction in loc. cit. we have
 \begin{equation}
 \label{uP}
 u(P)=1\ \ \text{if $\omega_P$ is induced by $\omega_{\textup{for}}$ via the classifying map}.\end{equation}
 
 To check Part 2  let $f$ be as in Equation \ref{fequal} and assume
$(f^B)^{\textup{alg}}=0$. 
Then, by the homogeneity of $F$, the  map $R^{\textup{alg}}\rightarrow 
K^{\textup{alg}}$ given by
$$\lambda\mapsto F(\ldots, \theta(\phi_{\eta}f^{\textup{jet}}_{\mu})^{\textup{alg}}(\lambda),\ldots, \theta(\phi_{\eta}f^{\textup{jet}}_{\mu,\nu})^{\textup{alg}}(\lambda),\ldots)$$
must vanish on the additive group $p^NR^{\textup{alg}}$ for some $N$ 
hence, again, by the homogeneity of $F$,  this map vanishes on $R^{\textup{alg}}$. Now we proceed as in the proof of Proposition \ref{injonchar}. Indeed write $\theta(\phi_{\eta}f_{\eta,\mu})=\sum_{\epsilon} \lambda_{\eta,\mu,\epsilon} \phi_{\epsilon}$ and 
$\theta(\phi_{\eta}f_{\mu,\nu})=\sum_{\epsilon} \lambda_{\eta,\mu,\nu,\epsilon} \phi_{\epsilon}$. Let $x_{\epsilon}$ be variables indexed by $\epsilon\in \mathbb M_n$.
By Lemma \ref{preartin} 
we get that the following polynomial vanishes:
\begin{equation}
\label{lhs}
F(\ldots,\sum_{\epsilon} \lambda_{\eta,\mu,\epsilon} x_{\epsilon},\ldots,\sum_{\epsilon} \lambda_{\eta,\mu,\nu,\epsilon} x_{\epsilon},\ldots)=0.\end{equation}
But $f$ is obtained from the left hand side of (\ref{lhs}) by replacing $x_{\epsilon}\mapsto \frac{1}{p}\phi_{\epsilon}T$. Hence $f=0$.
\qed

\bigskip

At this point we can give the following proof.	

\bigskip

{\it Proof of Theorem \ref{fat2}}.
By Remark \ref{noz} 
 the forms corresponding to the classes in Theorem \ref{fat2} are all non-zero.
Let $f$ be the product of all these forms. By Corollary \ref{intdomu} $f\neq 0$. By Proposition \ref{zaza}, Part 2, for $B=B_1(N)$, $N\geq 4$, we get that $(f^B)^{\textup{alg}}\neq 0$ so there exists $\pi\in \Pi$ and a 
point $P\in B(R_{\pi})$ such that the
pair
$(E_P,\omega_P)$ over $R_{\pi}$ corresponding to $P$ satisfies  $(f^B)^{\textup{alg}}(P)\neq 0$. We conclude by Equations \ref{zipo} and \ref{zipopo}.
\qed

\bigskip
Here is a characterization of ordinary elliptic curves with vanishing arithmetic Kodaira-Spencer classes; it is an improvement (and generalization) of \cite[Prop. 5.10]{BM20}.

\begin{proposition}\label{chara}
Let $E/R_{\pi}$ be an elliptic curve with ordinary reduction and $\omega$ a basis for its $1$-forms. The following are equivalent,

1) $f_{\pi,i}(E,\omega)=0$ for some $i\in \{1,\ldots, n\}$.

2) $f_{\pi,\mu}(E,\omega)=f_{\pi,\mu,\nu}(E,\omega)=0$ for all $\mu,\nu\in \mathbb M_n$.

3) The Serre-Tate parameter $q(E)$ is a root of unity.
\end{proposition}

{\it Proof}.
Assume condition  1) holds. Let $P\in B(R_{\pi})$ represent $(E,\omega)$. By formula \ref{zipo} we have $((f^{\textup{jet}}_{i})^B)^{\textup{alg}}(P)=0$ for $B=B_1(N)$, $N\geq 4$. By Equation \ref{aragorn1} (with $\mu=i$) since the map
$\theta(f^{\textup{jet}}_{i})^{\textup{alg}}:R^{\textup{alg}}\rightarrow K^{\textup{alg}}$, $\beta\mapsto \phi_i(\beta)-p\beta$,  is injective it follows
that $\log(q(E))=0$. Hence $q(E)$ is a root of unity, i.e. condition 3) holds.
Similarly condition 3) implies condition 2) due to Equations \ref{zipo} and \ref{aragorn1}, \ref{aragorn2} (applied to  arbitrary $\mu,\nu$). Finally Condition 2) trivially implies condition 1).
\qed

\subsection{Theorem of the Kernel and Reciprocity Theorem} Recall from the Introduction the following pairing.

\begin{definition}\label{exx}
Let $\mu,\nu\in \mathbb M^2_2$ have length $r,s\in \{1,2\}$ respectively.  Define the  $\mathbb Q_p$-bilinear map
$$\langle\ \ ,\ \ \rangle_{\mu,\nu}:K^{\textup{alg}}\times K^{\textup{alg}}\rightarrow K^{\textup{alg}}$$
by the formula
\begin{equation}
\label{tzotzu}
\langle \alpha, \beta \rangle_{\mu,\nu}=\beta^{\phi_{\nu}}\alpha^{\phi_{\mu}}-\beta^{\phi_{\mu}}\alpha^{\phi_{\nu}}
+p^s(\alpha \beta^{\phi_{\mu}}-\beta\alpha^{\phi_{\mu}})+p^r(\beta\alpha^{\phi_{\nu}}-\alpha\beta^{\phi_{\nu}}).\end{equation}
\end{definition}

Note that the above expression is antisymmetric in $\alpha,\beta$:
$$\langle \alpha ,\beta \rangle_{\mu,\nu}=-\langle\beta,\alpha \rangle_{\mu,\nu}$$
and also in $\mu,\nu$:
$$\langle \alpha ,\beta \rangle_{\mu,\nu}=-\langle\alpha,\beta \rangle_{\nu,\mu}.$$

Consider in what follows the following data. Fix  $\pi\in \Pi$ and an elliptic curve $E$ over $R_{\pi}$ with ordinary reduction $E_0$. Choose bases $b,\check{b}$ of $T_p(E_0),T_p(\check{E}_0)$ as in the beginning of Section \ref{DST}. Now set $S^0_{\textup{for}}=R[[T]]\rightarrow R_{\pi}$ the classifying homomorphism given by Serre-Tate theory. Choose the $1$-form $\omega$ on $E$ induced by the canonical form $\omega_{\textup{for}}$ on the universal elliptic curve $E_{\textup{for}}/S^0_{\textup{for}}$ via the classifying homomorphism and set $q(E)\in R_{\pi}$ the Serre-Tate parameter  of $E$, i.e. the image of $1+T$ via the classifying homomorphism. Denote also by $\beta:=\beta(E):=\log(q(E))\in K_{\pi}$, where $\log \colon 1+\pi R_{\pi}\rightarrow K_{\pi}$ is the usual logarithm. When $n=2$, set $\Phi=\{\phi_1,\phi_2\}$, $\mu\neq \nu\in \mathbb M_2^{2,+}$  of length $r, s$ respectively, with $r\geq s$. Consider $\psi_{\mu,\nu}:=\psi_{\pi,\mu,\nu}(E,\omega)\in \mathbf{X}^2_{\pi,\Phi}(E)$  the $\delta_{\pi}$-character  attached to $(E,\omega)$ and $c\in \mathbb Z_p^{\times}$  the constant in Equation (\ref{defofccc}). Recall this constant depends only on $p$.

\begin{proposition}\label{frodoo}
The  Picard-Fuchs symbol of $\psi_{\mu,\nu}$ is given by the following formula:
$$\theta(\psi_{\mu,\nu})=p^{N(\pi)+1}c[(\beta^{\phi_{\nu}}-p^s\beta)\phi_{\mu}-(\beta^{\phi_{\mu}}-p^r\beta)\phi_{\nu}+(p^s\beta^{\phi_{\mu}}-p^r\beta^{\phi_{\nu}})].$$
\end{proposition}

{\it Proof}.
A direct computation using (in the following order) the  Equations (\ref{missing1}), (\ref{tildef}), (\ref{zipo}), (\ref{zipopo}), (\ref{aragorn1}), (\ref{aragorn2}), (\ref{uP}), (\ref{tommy}).
\qed

\bigskip

 In view of Definition \ref{exx}  Proposition \ref{frodoo} yields then the formula
\begin{equation}
\label{shacal}
\theta(\psi_{\mu,\nu})^{\textup{alg}}(\alpha)=p^{2N(\pi)+1}c \langle \alpha, \beta\rangle_{\mu,\nu}.
\end{equation} 
Hence
$$\text{Ker}(\theta(\psi_{\mu,\nu})^{\textup{alg}})=\{\alpha\in K^{\textup{alg}}\ |\ \langle \alpha,\beta\rangle_{\mu,\nu} =0\}.$$
The above is a $\mathbb Q_p$-linear space;  this space contains $\beta$ which is non-zero if $q(E)$ is not a root of unity. So by  Corollary \ref{ToftheK} if $q(E)$ is not a root of unity  then the group  $\text{Ker}(\psi_{\mu,\nu}^{\textup{alg}})$ in {\it not torsion}. More generally by Corollary \ref{ToftheK} we get:

\begin{theorem}\label{frodooo} (Theorem of the Kernel)
We have a natural group isomorphism
$$\text{Ker}(\psi_{\mu,\nu}^{\textup{alg}})\otimes_{\mathbb Z}\mathbb Q\simeq
\{\alpha\in K^{\textup{alg}}\ |\ \langle \alpha,\beta\rangle_{\mu,\nu} =0\}.$$
\end{theorem}

\begin{remark}\label{cucu}
In fact, in view of Corollary \ref{strongger} a stronger result holds as follows. Let $L$ be a filtered union of complete 
subfields of $K^{\textup{alg}}$ and let $\mathcal O$ be the valuation ring of $L$.
Assume $E$ comes via base change from an elliptic curve over $\mathcal O$. Then
\begin{equation}
(\textup{Ker}(\psi_{\mu,\nu}^{\textup{alg}})\cap E(\mathcal O))\otimes_{\mathbb Z} \mathbb Q\simeq 
\{\alpha\in L\ |\ \langle \alpha,\beta\rangle_{\mu,\nu} =0\}.
\end{equation}
\end{remark}

\bigskip

In order to state our next result 
 fix  an ordinary elliptic curve $E_0$ over $k$ and 
 bases $b,\check{b}$ of $T_p(E_0),T_p(\check{E}_0)$ as in the beginning of Section \ref{DST}.
 Let $\pi\in \Pi$ and let
 \begin{equation}
 \label{dattay}
 \alpha,\beta\in K_{\pi},\ \ |\alpha|,|\beta|< p^{-\frac{1}{p-1}}.\end{equation}
 By \cite[Thm. 6.4]{Sil86} there exists  $q\in 1+\pi R_{\pi}$
 such that $\log q=\beta$. Then the homomorphism $R[[T]]\rightarrow R_{\pi}$, $T\mapsto q-1$,
 defines an elliptic curve $E_{\beta}$ over $R_{\pi}$ with logarithm of the Serre-Tate parameter satisfying $\log(q(E_{\beta}))=\beta$. Let $\ell_{E_{\beta}}^{\textup{alg}}:E_{\beta}(\pi R_{\pi})\rightarrow K_{\pi}$ be the logarithm of $E_{\beta}$ and let $\omega_{\beta}$ be the $1$-form on $E_{\beta}$ induced from the canonical $1$-form $\omega_{\textup{for}}$ defined by $\check{b}$ on the universal deformation $E_{\textup{for}}/R[[T]]$ of $E_0$.  Again, by \cite[Thm. 6.4]{Sil86}
 there exists a point $P_{\alpha,\beta}\in E_{\beta}(\pi R_{\pi})$ such that $\ell_{E_{\beta}}^{\textup{alg}}(P_{\alpha,\beta})=\alpha$. Since the roles of $\alpha$ and $\beta$ can be interchanged we also
 have at our disposal an elliptic curve $E_{\alpha}$ over $R_{\pi}$, a $1$-form $\omega_{\alpha}$ on $E_{\alpha}$,  and a point $P_{\beta,\alpha}\in E_{\alpha}(\pi R_{\pi})$. Let $\mu,\nu\in \mathbb M^2_n$ be distinct and let $\psi_{\mu,\nu,\beta}$ and $\psi_{\mu,\nu,\alpha}$ be the corresponding $\delta_{\pi}$-characters attached to $(E_{\beta},\omega_{\beta})$ and $(E_{\alpha},\omega_{\alpha})$ over $R_{\pi}$, respectively. Then formula (\ref{shacal}), the antisymmetry of $\langle\ ,\ \rangle_{\mu,\nu}$,  and the commutative diagram (\ref{patrat}) imply the following theorem.
 
 \begin{theorem}\label{reciporc} (Reciprocity Theorem)
 For every $\alpha, \beta$ as in (\ref{dattay})  and every distinct $\mu,\nu\in \mathbb M^2_2$  we have
 $$\psi_{\mu,\nu,\beta}^{\textup{alg}}(P_{\alpha,\beta})=\psi_{\nu,\mu,\alpha}^{\textup{alg}}(P_{\beta,\alpha}).$$
 \end{theorem}
 
 \begin{remark} Note that by antisymmetry in $\mu,\nu$ we get:
  \begin{equation}
  \label{vann}
  \psi_{\mu,\nu,\alpha}^{\textup{alg}}(P_{\alpha,\alpha})=0.\end{equation}
In case $\pi=p$ and $\alpha \in pR$  Equation \ref{vann} can also be derived as follows. Recall from \cite[Sect. (4.1)]{Bu95}  that if $\pi=p$ and $\alpha \in pR$ then the point $P_{\alpha,\alpha}$ belongs to the group $\cap_{m=1}^{\infty}p^m E(R)$ of infinitely $p$-divisible points of $E(R)$. On the other hand for every partial $\delta_p$-character $\psi$ of $E_{\alpha}$
 the homomorphism $\psi^{\textup{alg}}:E_{\alpha}(R^{\textup{alg}})\rightarrow K^{\textup{alg}}$ sends $E_{\alpha}(R)$ into $R$. Since $\cap_{m=1}^{\infty} p^mR=0$ we get $\psi^{\textup{alg}}(P_{\alpha,\alpha})=0$. We expect that Equation (\ref{vann}) can be derived along similar lines in the general case when $\pi$ is arbitrary and $\alpha$ is arbitrary, satisfying $|\alpha|< p^{-\frac{1}{p-1}}$.
 \end{remark}
  
 \begin{example}\label{illusrt}
 Here is an illustration of the Theorem of the Kernel; the example below can be easily generalized.
 
  Let $\ell\leq p-1$ be a prime
 and consider the field $K^{(l)}$ in Equation (\ref{Kelll}) and the notation of the paragraph containing that equation; in particular recall the elements $\pi_m$, $\zeta_{l^m}$, and the automorphisms $\phi^{(\gamma)}\in \mathfrak F(K^{\text{alg}}/\mathbb Q_p)$. Set $\phi_1:=\phi^{(0)}$ and $\phi_2:=\phi^{(1)}$; hence $\phi_1$ and $\phi_2$ are Frobenius automorphisms whose restriction to 
 $K_{\pi_m}$ satisfy $\phi_1 \pi_m=\pi_m$, $\phi_2 \pi_m=\zeta_{l^m} \pi_m$.
 Let $\beta=\pi_1$ (hence $|\beta|<p^{-\frac{1}{p-1}}$), let $\mu,\nu\in \mathbb M_2^{2,+}$ be distinct of length $2$ (a similar computation holds for length $1$) and consider the $\mathbb Q_p$-bilinear map
 $\langle\ ,\  \rangle_{\mu,\nu}:K^{\textup{alg}}\times K^{\textup{alg}}\rightarrow K^{\textup{alg}}$ in Equation \ref{tzotzu}.

\bigskip

{\it Claim 1.  For every $\alpha\in K^{(l)}$ we have $\langle \alpha,\beta\rangle_{\mu,\nu}=0$ if and only if there exists $\lambda\in \mathbb Q_p$ such that $\alpha=\lambda \beta$.}

\bigskip

The ``if" part is clear. To check the ``only if" part note that we may assume $\alpha\in R_{\pi_m}$ for some $m$ so we may write 
$$\alpha=\sum_{i=0}^{l^m-1}\alpha_i\pi_m^i,\ \ \alpha_i\in R.$$
Let $\phi$ be the restriction of $\phi_1,\phi_2$ to $R$.
Picking out the coefficient of $\pi_m^i$  in the right hand side of the equality (\ref{tzotzu}) we get 
from the equality $\langle \alpha,\beta\rangle_{\mu,\nu}=0$ that
$$\alpha_i^{\phi}(\zeta_l-\zeta_{l^m}^i-p^2+p^2\zeta_{l^m}^i)+\alpha_i(p^2-p^2\zeta_l)=0,\ \ i\in \{0,\ldots,l^m-1\}.$$
If $i$ is such that $\alpha_i\neq 0$, since $|\phi(\alpha_i)|=|\alpha_i|$, it follows that
$p$ divides $\zeta_l-\zeta_{l^m}^i$ in $R$ which forces $i=l^{m-1}$. This in turn implies $\phi(\alpha_{l^{m-1}})=
\alpha_{l^{m-1}}$, hence $\alpha_{l^{m-1}}\in \mathbb Z_p$ and our Claim 1 is proved.

Consider now the data $E_0,b,\check{b}$ in the paragraph before Theorem \ref{reciporc} and consider the 
elliptic curve $E_{\beta}$ over $R_{\pi_1}$ whose Serre-Tate parameter has logarithm equal to $\beta$. Consider, as in that paragraph, the partial $\delta_{\pi_1}$-character $\psi_{11,22,\beta}$.
Then by Claim 1 above and by the strengthening in Remark \ref{cucu} of our Theorem of the Kernel the following follows:

\bigskip

{\it Claim 2.  The group
$(\textup{Ker}(\psi_{\mu,\nu,\beta}^{\textup{alg}})\cap E_{\beta}(K^{(l)}))\otimes_{\mathbb Z} \mathbb Q$ is a one dimensional $\mathbb Q_p$-linear space with basis any  point $P_{\beta,\beta}$ whose elliptic logarithm is $\beta$.}

\bigskip

On the other hand, if instead of the elliptic curve $E_{\beta}=E_{\pi_1}$ over $R_{\pi_1}$ above we consider an elliptic curve $E_{\gamma}$ over $R$ with $\gamma\in pR$ then we get

\bigskip

{\it Claim 3.  The group
$(\textup{Ker}(\psi_{\mu,\nu,\gamma}^{\textup{alg}})\cap E_{\gamma}(K))\otimes_{\mathbb Z} \mathbb Q$ is naturally isomorphic to $K$ hence this group is an infinite  dimensional $\mathbb Q_p$-linear space.}

\bigskip

Indeed, in this case, one has $\langle \alpha,\gamma\rangle_{\mu,\nu}=0$ for every $\alpha\in K$. 

\bigskip

The above is a ``genuine PDE" example of an explicit computation for the kernel of a $\delta$-character,  in a tamely ramified situation. For the ODE case one has a complete description of the kernels of $\delta$-characters if one restricts to the  unramified situation; cf. \cite[Introd.]{Bu97}.
 \end{example}

 \subsection{Crystalline construction}
 For background here we refer to \cite[Sect. 8.4.3]{Bu05}.
 Let $S^*$ be an object of ${\bf Prol}_{p,\Phi}$ and $E/S^0$ an elliptic curve.
 Let $H^1_{DR}(E/S^0)$ be the deRham $S^0$-module.
 By crystalline theory  for every $i\in \{1,\ldots,n\}$ and every $r\geq 0$ we have $\phi_i$-linear maps
 $$H^1_{DR}(E/S^0)\otimes_{S^0}S^r\stackrel{\phi_i}{\longrightarrow} H^1_{DR}(E/S^0)\otimes_{S^0}S^{r+1}.$$
 Also one has the deRham pairing
 $$\langle\ ,\ \rangle_{DR}:H^1_{DR}(E/S^0)\otimes_{S^0}S^r\times H^1_{DR}(E/S^0)\otimes_{S^0}S^r\rightarrow S^r.$$
 
 \begin{definition}\label{defoffcrys}
 For every basis $\omega$ of $1$-forms on $E/S^0$ and every distinct $\mu,\nu\in \mathbb M_n^{r,+}$ we set
 $$f^{\textup{crys}}_{\mu}(E/S^0,\omega,S^*)=\frac{1}{p}\langle \phi_{\mu}\omega,\omega\rangle_{DR} \in S^r,$$
 $$f^{\textup{crys}}_{\mu,\nu}(E/S^0,\omega,S^*)=\frac{1}{p}\langle \phi_{\mu}\omega,\phi_{\nu}\omega\rangle_{DR} \in S^r.$$
 \end{definition}
 
 It is trivial to check that the following proposition holds.
 
 \begin{proposition}\label{cucu65}
 The rules $f^{\textup{crys}}_{\mu}$ and $f^{\textup{crys}}_{\mu,\nu}$ define isogeny covariant partial $\delta_p$-modular forms of order $\leq r$ and weights $-1-\phi_{\mu}$ and $-\phi_{\mu}-\phi_{\nu}$, respectively.
 \end{proposition}
 
 On the other hand we have the next proposition.
 
 \begin{proposition}\label{corfrodo2crys}
 Let $\mu,\nu\in \mathbb M_n^r$, $\mu\neq \nu$, of length $r,s$ respectively,
with $r\geq s$. The  symbols of $f^{\textup{crys}}_{\mu},f^{\textup{crys}}_{\mu,\nu}$ are given by 
\begin{equation}
\begin{array}{rcl}
\theta(f^{\textup{crys}}_{\mu}) & = & \phi_{\mu}-p^r,\\
\theta(f^{\textup{crys}}_{\mu,\nu}) & = & p^s\phi_{\mu}-p^r\phi_{\nu}.
\end{array}\end{equation}
In particular the Serre-Tate expansions $\mathcal E(f^{\textup{crys}}_{\mu})$  are $R$-linearly independent and not divisible by $p$ in $S^r_{\textup{for}}$.
 \end{proposition}
 
 {\it Proof}. This follows exactly as in the proof of \cite[Prop. 8.61]{Bu05}.\qed
 
 \bigskip

 \begin{remark}\label{iona33}
 Following the lead of the ODE case  we expect that the forms
 $f^{\textup{crys}}_{\mu}$ and $f^{\textup{crys}}_{\mu,\nu}$ coincide  up to a multiplicative constant in $\mathbb Z_p^{\times}$ with the forms
 $f^{\textup{jet}}_{\mu}$ and $f^{\textup{jet}}_{\mu,\nu}$. Cf. also Remark \ref{iubity} for more on this.  In any case, by \cite[Cor. 8.84]{Bu05} (adapted to the theory over $\mathbb Z_p$ instead of over $R$, as explained in Remark \ref{plustheequality}) we have the next corollary. \end{remark}

 \begin{corollary} Assume $n=1$. Then for all $\mu$ we have
 $f^{\textup{jet}}_{\mu}\in \mathbb Z_p^{\times}\cdot f^{\textup{crys}}_{\mu}$.
 \end{corollary}

 For $n=2$ and with $c\in \mathbb Z_p^{\times}$ as in Equation (\ref{defofccc}) our theory yields the following result.
 
  \begin{corollary}\label{schwazenger}
 Let $\mu,\nu\in \mathbb M_2^{2,+}$. Then we have:
$$\begin{array}{rcl}
f^{\textup{jet}}_{\mu} & = & c\cdot f^{\textup{crys}}_{\mu},\\
f^{\textup{jet}}_{\mu,\nu} & = & c \cdot f^{\textup{crys}}_{\mu,\nu}.\end{array}$$
 \end{corollary}
 
 {\it Proof}. By Corollary \ref{corfrodo2} and Proposition \ref{corfrodo2crys} the forms 
 in the left hand sides of our equations have the  same Serre-Tate expansion as the forms 
 in the rights hand sides, respectively. Since these forms have the same corresponding weights we conclude by the Serre-Tate expansion principle, cf. Theorem \ref{STE}.
 \qed

 \subsection{Forms on the ordinary locus}
 
\begin{definition}
A {\bf $\delta_p$-modular function of order $\leq r$ on the ordinary locus} is a rule
 $f$ assigning to each object $(E/S^0, \omega, S^*)$ with $E/S^0$ ordinary an element $f(E/S^0, \omega, S^*)$ of the ring $S^r$, depending only on the isomorphism class of $(E/S^0, \omega, S^*)$, such that $f$ commutes with base change of the prolongation sequence. 
 We denote by $M_{p,\Phi,\textup{ord}}^r$ the set of all $\delta_p$-modular function of order $\leq r$ on the ordinary locus; it has a structure of $R$-algebra. \end{definition}

In \cite{Bu05} such $f$'s were called {\it ordinary} but we want to avoid here the term {\it ordinary} so no confusion arises with the use of this word in relation to  ODEs/PDEs.

\begin{remark} \label{6parts}
The general theory developed in the preceding subsections can be developed in this context as follows.

1) The set 
$M_{p,\Phi,\textup{ord}}^r$
has an obvious structure of ring. As in \cite{Bu00} we have a natural ring isomorphism
$$M_{p,\Phi,\textup{ord}}^r\simeq J^r_{\pi,\Phi}(M_{p,\Phi}[E_{p-1}^{-1}])= R[\delta_{\mu} a_4,\delta_{\mu} a_6,\Delta^{-1},E_{p-1}^{-1}\ |\ \mu\in \mathbb M_n^r]^{\widehat{\ }},$$
 where $E_{p-1}\in \mathbb Z_p[a_4,a_6]$ corresponds to the Eisenstein series of weight $p-1$.
 
  2) As in Definition \ref{defofweighty} one defines what it means for an element $f\in M_{p,\Phi,\textup{ord}}^r$ to have  {\bf weight} $w\in \mathbb Z^r_{\Phi}$ 
 by requiring that the condition in that definition
  be satisfied only for ordinary elliptic curves.
 We denote by $M_{p,\Phi,\textup{ord}}^r(w)$ the $R$-submodule of $M_{p,\Phi,\textup{ord}}^r$ of weight $w$.
 
  3) As in Definition \ref{defofisocov}, one  defines what it means for an element $f\in M_{p,\Phi,\textup{ord}}^r(w)$ to be {\bf isogeny covariant} 
 by requiring that the condition in that definition be satisfied only for ordinary elliptic curves. 
 We denote by $I_{p,\Phi,\textup{ord}}^r(w)$
 the submodule of all isogeny covariant elements of $M_{p,\Phi,\textup{ord}}^r(w)$. 
The direct sum 
$\bigoplus_{w\in \mathbb Z_{\Phi}} I^r_{p,\Phi,\textup{ord}}(w)$
is a $Z_{\Phi}$-graded $R$-subalgebra of the $R$-algebra $\bigoplus_{w\in \mathbb Z_{\Phi}} M^r_{p,\Phi,\textup{ord}}(w)$. For every $f\in I^r_{p,\Phi,\textup{ord}}(w)$ and every $i$ 
we have $f^{\phi_i}\in I^r_{p,\Phi,\textup{ord}}(\phi_i w)$.

 4) As in Theorem \ref{STE} for every $w\in \mathbb Z_{\Phi}$ there is a natural {\bf Serre-Tate expansion} homomorphism
$\mathcal E:M^r_{p,\Phi,\textup{ord}}(w) \rightarrow S^r_{\textup{for}}$, $ f\mapsto \mathcal E(f)$,
which is injective with torsion free cokernel.

5) As in Theorem \ref{mathcalEf} for every weight $w$ of degree $\deg(w)=-2$ and every $f\in I^r_{p,\Phi,\textup{ord}}(w)$ we have that $\mathcal E(f)$ is a $K$-linear combination of elements in the set 
$$\{\Psi_i^{\phi_{\mu}}\ |\ \mu\in \mathbb M_n^{r-1},\ i\in \{1,\ldots,n\}\}.$$ 
In particular $I^r_{p,\Phi,\textup{ord}}(w)$ has rank $\leq D(n,r)-1$.

6) For every weight $w$ of degree $\deg(w)=0$ and every $f\in I^r_{p,\Phi,\textup{ord}}(w)$ we have that $\mathcal E(f)\in K$. (To prove this one proceeds as in the proof of 
Theorem \ref{mathcalEf}
by noting that, in this case,  the right hand side of Equation \ref{mathcalEff} reduces to 
 $F(\ldots,\delta_{p,\mu}T,\ldots)$ which forces $F\in K$.) In particular, by the Serre-Tate expansion principle 4) above,  the $R$-module $I^r_{p,\Phi,\textup{ord}}(w)$ has rank $1$.
 
 7) There are natural $R$-module homomorphisms $M^r_{p,\Phi}(w)\rightarrow M^r_{p,\Phi,\textup{ord}}(w)$ and $I^r_{p,\Phi}(w)\rightarrow I^r_{p,\Phi,\textup{ord}}(w)$ that are injective with torsion free cokernel.
 \end{remark}
 
 Recall the following result due to Barcau; cf. \cite[Thm. 5.1, Cor. 5.1, Prop. 5.2]{Bar03} and 
 \cite[Thm. 8.83]{Bu05}.
 
 \begin{theorem}\label{shacall}
 Assume $n=1$, $\Phi=\{\phi\}$. 
 There exist elements $f^{\partial}\in I^1_{p,\phi,\textup{ord}}(\phi-1)$ and $f_{\partial}\in I^1_{p,\phi,\textup{ord}}(1-\phi)$  such that
 
 1) $f^{\partial}$ and $f_{\partial}$ are bases modulo torsion for these  $R$-modules, respectively;

 2)  $f^{\partial}\cdot f_{\partial}=1$ in $M^1_{p,\phi,\textup{ord}}$;
 
 3) $\mathcal E(f^{\partial})=\mathcal E(f_{\partial})=1$;
 
 4) $f^{\partial}\equiv E_{p-1}$ and $f_{\partial}\equiv E_{p-1}^{-1}$ mod $p$ in $M^1_{p,\phi,\textup{ord}}$.
 
 5)  $I^1_{p,\phi}(\phi-1)=I^1_{p,\phi}(1-\phi)=0$ hence $f^{\partial},f_{\partial}\not\in M^1_{p,\phi}$.
  \end{theorem}
  
  Part 5 says intuitively that $f^{\partial}_{\phi},f_{\partial}$ are ``genuinely singular along the supersingular locus."
  
  \bigskip
  
  {\it Proof}.
  We recall the idea of the argument using references to \cite{Bu05}. For any triple $(E/S^0,\omega,S^*)$ with $E/S^0$ ordinary we define
  $$f^{\partial}(E/S^0,\omega,S^*):=\frac{\langle \phi u,
  \omega\rangle_{DR}}{\phi(\langle u,\omega\rangle_{DR})}\in S^r$$
  where $u\in H^1_{DR}(E/S^0)$ is a basis of the unit root subspace of  $H^1_{DR}(E/S^0)$;
  cf. \cite[p.269]{Bu05}. We also we define 
  $$f_{\partial}(E/S^0,\omega,S^*):=
  \frac{\phi(\langle u,\omega\rangle_{DR})}{\langle \phi u,
  \omega\rangle_{DR}}\in S^r.$$
  One readily checks that these formulae define elements of 
  $$I^1_{p,\phi,\textup{ord}}(\phi-1),\ I^1_{p,\phi,\textup{ord}}(1-\phi),$$ 
  respectively. For the computation of their Serre-Tate expansion 
  (Part 2 in the Theorem) we refer to \cite[Prop. 8.59]{Bu05}. Then these elements being non-zero  are bases  modulo torsion of the corresponding modules by Remark \ref{6parts}, Part 6, hence Part 1 of the Theorem follows. Part 3 is obvious. Part 4 follows from  \cite[Thm. 8.83, Part 3]{Bu05}.
  \qed
 
 \begin{definition}
 Let $n$ be arbitrary and 
 denote 
 by $f^{\partial}_i$ and $f_{i,\partial}$
 the images of $f^{\partial}$ and $f_{\partial}$ via the face maps
 $$M^1_{p,\phi,\textup{ord}}\simeq M^1_{p,\phi_i,\textup{ord}}\rightarrow M^1_{p,\Phi,\textup{ord}}.$$\end{definition}
 
 \begin{corollary} \label{briann}
 The following claims hold for every $i\in \{1,\ldots,n\}$,
 
 1)  $f^{\partial}_i$ and $f_{i,\partial}$ are bases modulo torsion for the $R$-modules,  
  $I^1_{p,\Phi,\textup{ord}}(\phi_i-1)$ and $I^1_{p,\Phi,\textup{ord}}(1-\phi_i)$,
 respectively;
 
 2)  $f^{\partial}_i\cdot f_{i,\partial}=1$ in $M^1_{p,\Phi,\textup{ord}}$;
 
 3) $\mathcal E(f^{\partial}_i)=\mathcal E(f_{i,\partial})=1$;
 
  4) $f^{\partial}_i\equiv E_{p-1}$ and $f_{i,\partial}\equiv E_{p-1}^{-1}$ mod $p$ in $M^1_{p,\Phi,\textup{ord}}$.

 5)  $f^{\partial}_i,f_{i, \partial}\not\in M^1_{p,\Phi}$.
 \end{corollary}
 
 {\it Proof}.
 Parts 1 to 4 follow from Parts 1 to 4 of Theorem \ref{shacall}. Part 5 follows from Part 5 
 of Theorem \ref{shacall} by using the fact that the images of $f^{\partial}_i$ and $f_{i, \partial}$ via  the degeneration map $M^1_{p,\Phi}\rightarrow M^1_{p,\phi_i}$ are $f^{\partial}$ and $f_{\partial}$, respectively.
 \qed
 
 \bigskip
 
 For the next result let us consider, for every $r\geq 1$, the unique group homomorphism
 \begin{equation}
 \{w\in \mathbb Z^r_{\Phi}\ |\ \deg(w)=0\}\rightarrow (M^r_{p,\Phi,\textup{ord}})^{\times},\ \ \ w\mapsto f_{(w)},
 \end{equation}
 satisfying
 \begin{equation}
 \label{bour1}
 f_{(\phi_{i_1\ldots i_s}-1)}:=f_{i_1}^{\partial} \cdot (f_{i_2}^{\partial})^{\phi_{i_1}}\cdot
 (f_{i_3}^{\partial})^{\phi_{i_1i_2}}\ldots (f_{i_s}^{\partial})^{\phi_{i_1 i_2 i_3 \ldots i_{s-1}}},\ \ 
 s\in \{1,\ldots,r\}.
 \end{equation}
  Note that, by Corollary \ref{briann}, Part 4, we have the following congruences in 
 $M^s_{p,\Phi,\textup{ord}}$:
 \begin{equation}
 f_{(\phi_{i_1\ldots i_s}-1)}\equiv E_{p-1}^{1+p+p^2+\ldots+p^{s-1}}\ \ \text{mod}\ \ p.
 \end{equation}
  
 \begin{corollary}\label{stevee}
 For every $r\geq 1$ the following claims hold.
 
 1) For every $w\in \mathbb Z^r_{\Phi}$ of degree $\deg(w)=0$ the form $f_{(w)}$ is a basis of  the  $R$-module $I^r_{p,\Phi,\textup{ord}}(w)$.
 
 2) For every $v\in \mathbb Z^r_{\Phi}$ of degree $\deg(v)=-2$ the $R$-module $I^r_{p,\Phi,\textup{ord}}(v)$ has rank $D(n,r)-1$ and a basis modulo torsion is given by the set
 $$\{f_{(v+\phi_{\mu}+1)} f_{\mu}^{\textup{crys}}\ |\ \mu\in \mathbb M_n^{r,+}\}.$$
 \end{corollary}
 
 {\it Proof}. To check Part 1 note that by Remark \ref{6parts}, Part 6, $I^r_{p,\Phi,\textup{ord}}(w)$ has rank $1$. We are done by noting that  $f_{(w)}$ belongs to this module and is not divisible by $p$ in this module.
  
 To check Part 2 note that the forms $f_{(v+\phi_{\mu}+1)} f_{\mu}^{\textup{crys}}$ have weight $v$ hence belong to $I^r_{p,\Phi,\textup{ord}}(v)$. Since the latter module has rank $\leq D(n,r)-1$ (cf. Remark \ref{6parts}, Part 5) it is enough to check that the forms 
 $f_{(v+\phi_{\mu}+1)} f_{\mu}^{\textup{crys}}$
 are $R$-linearly independent. For this it is enough to check that their Serre-Tate expansions
 are $R$-linearly independent. However, by Corollary \ref{briann}, Part 3, we have
 $$\mathcal E(f_{(v+\phi_{\mu}+1)} f_{\mu}^{\textup{crys}})=\mathcal E(f_{\mu}^{\textup{crys}}),$$
 and we may conclude by Proposition \ref{corfrodo2crys}.
 \qed
 
 \begin{remark}\label{iubity}
 Following the lead from the ODE case (cf. \cite{Bar03} or \cite[Thm 8.83, Part 2]{Bu05})
 it is natural to ask if for all distinct $\mu,\nu\in \mathbb M^r_n$ we have:
 \begin{equation}
 \label{orpa0}
 \textup{rank}_R I^r_{p,\Phi}(-1-\phi_{\mu})=1,
 \end{equation}
 \begin{equation}
 \label{orpa1}
 \textup{rank}_R I^r_{p,\Phi}(-\phi_{\mu}-\phi_{\nu})=1,
 \end{equation}
  \begin{equation}
 \label{orpa2}
 I^r_{p,\Phi}(-2)=I^r_{p,\Phi}(-2\phi_{\mu})=0.
 \end{equation}
 By loc. cit. the above Equations hold if $n=1$. For $n=2$ the ``simplest case" ($r=1, \mu=1,\nu=2$) of Equation (\ref{orpa1}) holds; cf. Theorem \ref{mult1}.
 Note that if the conditions (\ref{orpa0}) and (\ref{orpa1}) hold in general then there exist constants $\lambda_{\mu}\in \mathbb Z_p$ and $\lambda_{\mu,\nu}\in \mathbb Q_p$ such that:
 \begin{equation}
 f^{\textup{jet}}_{\mu}=\lambda_{\mu}\cdot f^{\textup{crys}}_{\mu},\end{equation}
  \begin{equation}
 f^{\textup{jet}}_{\mu,\nu}=\lambda_{\mu,\nu}\cdot f^{\textup{crys}}_{\mu,\nu}.\end{equation} 
 Indeed by Theorems \ref{uxx} and \ref{theyareisogcov} and by  Proposition  \ref{cucu65} the left hand sides and the right hand sides of the above equations belong to the same $R$-modules of rank one, respectively. On the other hand  the forms $f^{\textup{crys}}_{\mu}$
 are not divisible by $p$ while the forms $f^{\textup{crys}}_{\mu,\nu}$ are non-zero (cf. Proposition \ref{corfrodo2crys}).  Moreover, again under the assumption that conditions (\ref{orpa0}) and (\ref{orpa1}) hold, since 
  by Remark \ref{yuyi} the forms $f^{\textup{jet}}_{\mu}$ are not divisible by $p$, we get $\lambda_{\mu}\in \mathbb Z_p^{\times}$. Neverthless, even under the assumption that 
  conditions (\ref{orpa0}) and (\ref{orpa1}) hold, we cannot conclude that $\lambda_{\mu,\nu}\neq 0$
  (let alone that $\lambda_{\mu,\nu}\in \mathbb Z_p^{\times}$, as in Corollary \ref{schwazenger}).
  We recall that the proof of Corollary \ref{schwazenger} involved ``solving a system of quadratic and cubic equations" satisfied by the $f^{\textup{jet}}$ forms as in the proof of Theorem \ref{nonzzero}. So  even if one can prove conditions (\ref{orpa0}) and (\ref{orpa1}) one  still cannot go around  solving  our system of quadratic and cubic equations if one wants to prove the non-vanishing of the forms $f^{\textup{jet}}_{\mu,\nu}$ for distinct $\mu,\nu$ as in Corollary \ref{schwazenger}.
  
  Finally one may hope to prove conditions (\ref{orpa0}) and (\ref{orpa1}) along the lines of  \cite{Bar03} or \cite[Thm. 8.83]{Bu05}. However,
  even proving condition (\ref{orpa0}) for $n=2, \mu=1$ along these lines does not  seems  not work
  in an obvious way. Indeed, by Corollary \ref{stevee} we have that $I^1_{p,\phi_1,\phi_2,\textup{ord}}(-1-\phi_1)$ has a basis modulo torsion consisting of $f_1$ and $f_{1,\partial}f_2^{\partial}f_2$.
  So in order to prove that $I^1_{p,\phi_1,\phi_2}(-1-\phi_1)$ has rank $1$ we need to show that
 the form $f_{1,\partial}f_2^{\partial}f_2\in M^1_{p,\Phi,\textup{ord}}$ does not belong to $M^1_{p,\Phi}$. The argument in loc. cit. for this type of statement was to show that the image of the corresponding form in $M^1_{p,\Phi,\textup{ord}}\otimes_R k$ does not belong to 
 $M^1_{p,\Phi}\otimes_R k$. But in our case the image of $f_{1,\partial}f_2^{\partial}f_2$  in  $M^1_{p,\Phi,\textup{ord}}\otimes_R k$ equals the image of $f_2$ which {\it does} belong to 
 $M^1_{p,\Phi}\otimes_R k$.
 \end{remark}

 On the other hand, as an application of the theory we get a whole series of identities between our forms $f^{\textup{jet}}$ and $f^{\partial}$; here is an example:
 
 \begin{corollary} \label{dico} 
 The following formula holds in $I^1_{p,\phi_1,\phi_2,\textup{ord}}(-\phi_1-\phi_2)$,
 $$f^{\textup{jet}}_{1,2}= p (f^{\textup{jet}}_{1} f_{2,\partial}-f^{\textup{jet}}_{2}f_{1,\partial}).$$
 \end{corollary}

 {\it Proof}.
 The two sides of the formula  have the same weight equal to $-\phi_1-\phi_2$ and the same Serre-Tate expansions (cf. Remark \ref{remfrodo1} and Theorem \ref{nonzzero}). So they must be equal by the Serre-Tate expansion principle (Remark \ref{6parts}, Part 4).
 \qed
 
 \begin{remark}
 The formula in Corollary \ref{dico}
 is interesting in that $f_{1,\partial}, f_{2,\partial}$ in the right hand side do not belong to $M^1_{p,\Phi}$ (they are ``genuinely singular along the supersingular locus", cf. Corollary \ref{briann}, Part 4) while the left hand side does belong to $M^1_{p,\Phi}$; so, intuitively, the ``singularities" in the right hand side ``cancel each other out." In view of Corollary \ref{stevee} one has similar formulae (exhibiting  similar ``cancellations of singularities") for every  $f\in I^r_{p,\phi_1,\phi_2}(w)$ with $w$ of degree $\deg(w)=-2$. \end{remark}

 Finally we address the total $\delta$-overconvergence aspect, there by strengthening the results in \cite{BM20}.
  
 \begin{theorem}\label{ordforms}
  Let $B=B_1(N)$ be the natural bundle  over an open set $X\subset Y_1(N)$ of the modular curve $Y_1(N)$ over $R$, for $N\geq 4$, $N$ coprime to $p$, and assume the reduction mod $p$ of $X$ is contained in the ordinary locus of the reduction mod $p$ of $Y_1(N)$.  Then the following hold:
  
  1) For every weight $w$ of degree $\deg(w)=0$  and every $f\in I^r_{p,\Phi,\textup{ord}}(w)$
 the element $f^B\in \mathcal O(J^r_{p,\Phi}(B))$ is totally $\d$-overconvergent. 
 
 2) For every $f$ as in 1) the map
 $(f^B)^{\textup{alg}}:B(R^{\textup{alg}})\rightarrow K^{\textup{alg}}$
 extends to a continuous map
  $(f^B)^{\mathbb C_p}:B(\mathbb C_p^{\circ})\rightarrow \mathbb C_p$.
    \end{theorem}
    
    {\it Proof}.
    By Corollary \ref{stevee} we may assume $f$ is either $f_{i,\partial}$ or $f_i^{\partial}$.
    Assume  $f=f_i^{\partial}$; the other case is similar.
     Note that since $f^{\partial}_i$ is induced via a face map from the form $f^{\partial}$ we are reduced to check Part 1 in case $n=1$; but this was proved in \cite[Cor. 5.12]{BM20}. For Part 2 
     we proceed exactly as in the proof  of Proposition \ref{zaza}; note that in our case here, since $\mathcal E(f^{\partial}_i)=1$,  we will simply have
     \begin{equation}
\label{aragorn112}
((f^{\partial}_i)^B)^{\textup{alg}}(P)=u(P)^{\phi_i-1}\end{equation}
for $u$ as in that proof.
    \qed
    
    \begin{remark}
  The maps $(f^B)^{\mathbb C_p}$ in Part 2 of Theorem \ref{ordforms} 
   satisfy a compatibility property with respect to isogenies  at points with coordinates belonging to $R$ because of the isogeny covariance of $f^{\partial},f_{\partial}$; however we do not know if the  maps $(f^B)^{\mathbb C_p}$ continue to satisfy a compatibility property with respect to isogenies  at points with coordinates {\it not} belonging to $R$. 
   For this to hold it would be sufficient  to ``naturally extend" the crystalline definition of $f^{\partial},f_{\partial}$ in the proof of Theorem \ref{shacall} to the ramified case. 
    \end{remark}
    
    \subsection{Finite covers defined by $\delta$-modular forms} 
     Throughout the discussion below fix an element $\pi\in \Pi$ and $\Phi=(\phi_1,\ldots,\phi_n)$ with $n\geq 1$. We recall the forms $f^{\textup{jet}}_{\pi,i}$ for $i\in \{1,\ldots,n\}$; cf. Theorem \ref{ux}. For an affine open set
    $X=\textup{Spec}(A)\subset Y_1(N)$ with non-empty reduction mod $\pi$ denote by $E_X$ the corresponding universal elliptic curve over $X$. Assume there is a basis $\omega$ for the $1$-forms on $E_X/X$ (which can be achieved by shrinking $X$) and consider the unique  elements
    $$\check{f}_{\pi,i}\in J^1_{\pi,\phi_i}(A)\setminus \pi J^1_{\pi,\phi_i}(A)$$
    such that there exist (necessarily unique) integers $n_i\geq 0$ with
    $$\pi^{n_i}\cdot \check{f}_{\pi,i}=f^{\textup{jet}}_{\pi,i}(E_X/X,\omega,J^*_{\pi,\phi_i}(A))\in J^1_{\pi,\phi_i}(A).$$
    We continue to denote by $\check{f}_{\pi,i}$ the images of these elements in $J^1_{\pi,\Phi}(A)$. Our main result here is the following theorem.	
    
    \begin{theorem}\label{fllow}
    There exists an affine open set $X=\textup{Spec}(A)\subset X_1(N)$ of the modular curve $X_1(N)$ over $R_{\pi}$, with non-empty reduction mod $\pi$, and a  basis $\omega$ for the $1$-forms on $E_X/X$
    such that the ring homomorphism
    $$\widehat{A}\rightarrow J^1_{\pi,\Phi}(A)/(\check{f}_{\pi,1},\ldots,\check{f}_{\pi,n})$$
    is a finite algebra map.    \end{theorem}
    
    If the map above is an isomorphism (which happens for instance if $\pi=p$ as one can easily see from the proof below) then one can view the arithmetic differential equations $\check{f}_{\pi,1},\ldots,\check{f}_{\pi,n}$ as defining an `arithmetic flow' on $X$;
    in the more general case when the map in the theorem is merely a finite algebra map
    one should view $\check{f}_{\pi,1},\ldots,\check{f}_{\pi,n}$ as defining a structure slightly more general than that of  an `arithmetic flow.' 
    
    \ 
    
    {\it Proof}.
    Since the source and the target of the map in the theorem are $p$-adically complete rings
    it is enough to show that there exists $X=\textup{Spec}(A)$ such that the map
    $$A/\pi A\rightarrow (J^1_{\pi,\Phi}(A)/(\check{f}_{\pi,1},\ldots,\check{f}_{\pi,n}))/(\pi)$$
   is a finite algebra map. Start with an arbitrary $X$ as in the paragraph before our Theorem. Replacing $X$ by an affine open set we may assume there is an \'{e}tale map $R_{\pi}[y]\rightarrow A$ so we have  identifications
   $$J^1_{\pi,\phi_i}(A)=A[\delta_{\pi,i}y]^{\widehat{\ }},\ \ \ \ J^1_{\pi,\Phi}(A)=A[\delta_{\pi,1}y,\ldots,\delta_{\pi,n}y]^{\widehat{\ }}.$$
   The theorem will be proved if we show that for every $i\in \{1,\ldots,n\}$ the image $\overline{f}_{\pi,i}$ of $\check{f}_{\pi,i}$
   in the ring $(A/\pi A)[\delta_{\pi,i}y]$ is not contained in the ring $A/\pi A$.
   Note that by definition  $\overline{f}_{\pi,i}\neq 0$  for all $i$.
   Assume  for some $i$ we have 
   $\overline{f}_{\pi,i}\in A/\pi A$
   and seek a contradiction.
   Consider the natural map
   $$\mathcal E_{\pi}: J^1_{\pi,\phi_i}(A)\rightarrow 
   R_{\pi}[[T]][\delta_{\pi,i}T]^{\widehat{\ }}$$
   defined similar to 
  the Serre-Tate expansion map. Since the reduction mod $\pi$ of this map,
  $$\overline{\mathcal E_{\pi}}: J^1_{\pi,\phi_i}(A)/(\pi)=(A/\pi A)[\delta_{\pi,i} y] \rightarrow k[[T]][\delta_{\pi,i}T]$$
  is injective it follows that the image $\overline{E_{\pi}}(\overline{f}_{\pi,i})\in k[[T]][\delta_{\pi,i}T]$
  of $\overline{f}_{\pi,i}$ is non-zero and is contained in $k[[T]]$. Let $z$ be a variable and consider the $k[[T]]$-algebra isomorphism
  $$\sigma:k[[T]][z]\rightarrow k[[T]][\delta_{\pi,i}T],\ \ \ \sigma(z)=\frac{\delta_{\pi,i}(1+T)}{(1+T)^p}.$$
  We have
  \begin{equation}
  \label{condddi}
   0 \neq \sigma^{-1}(\overline{\mathcal E_{\pi}}(\overline{f}_{\pi,i}))\in k[[T]]\subset k[[T]][z].\end{equation}
  By the compatibility of $\mathcal E_{\pi}$ with the Serre-Tate expansion map and in view of
  Remark \ref{remfrodo1} it follows that there exists an integer $N\in \mathbb Z$ such that
  $$\mathcal E(\check{f}_{\pi,i})=u(T)^{-1-\phi_i}\cdot \pi^N \sum_{m\geq 1}(-1)^{m+1}\frac{\pi^m}{m}\left(
  \frac{\delta_{\pi,i}(1+T)}{(1+T)^p}\right)^m$$
  for some invertible series $u(T)\in R_{\pi}[[T]]^{\times}$.
  Letting $\lambda_m\in k$ be the image of 
  $$(-1)^{m+1}\frac{\pi^{N+n}}{m}\in R_{\pi}$$
  we have that  \begin{equation}
  \label{ppii}
  \sigma^{-1}(\overline{\mathcal E_{\pi}}(\overline{f}_{\pi,i}))=u(T)^{-1-p}\cdot \sum_{m=1}^d\lambda_m z^m\in k[[T]][z]
  \end{equation}
  for some $d\geq 1$ which by Equation \ref{condddi} implies
  $$0\neq \sum_{m=1}^d\lambda_m z^m\in k[[T]]\subset k[[T]][z],$$
   a contradiction.
    \qed
    
    \begin{remark}
    The integer $N$ and the polynomial $S(z)=\sum_{i=1}^d \lambda_iz^i$  in the above proof can be computed explicitly. 
    Indeed let $e$ be the ramification index of $R_{\pi}$ over $R$. Then $-N$ is the minimum of the $\pi$-adic valuations of the numbers $\frac{\pi^m}{m}$ hence writing $m=p^\kappa s$ with $s\in\mathbb Z\setminus p\mathbb Z$ we have
    $$-N=\min\{p^\kappa s-\kappa e\ |\ \kappa\geq 0,\ s\geq 1\}=\min\{p^\kappa-\kappa e\ |\ \kappa\geq 0\}.$$
    On the other hand the function $f:\mathbb R\rightarrow \mathbb R$ defined by $f(x)=p^x-ex$
    has a unique minimum at 
    $$\theta:=\frac{\log e-\log \log p}{\log p}.$$ 
    So if $e<\log p$ then $f$ is strictly increasing on $\mathbb Z_{\geq 0}$ which implies that $N=-1$ and $S(z)=z$. On the other hand if $e>\log p$
    and $\kappa_0\in \mathbb Z_{\geq 0}$ is such that $\theta\in [\kappa_0,\kappa_0+1]$ 
     then the restriction of $f$ to $\mathbb Z_{\geq 0}$ attains its minimum either at $\kappa_0$ or $\kappa_0+1$ or at both (the last case occurring if and only if $e=p^{\kappa_0+1}-p^{\kappa_0}$). Consequently we have that $S(z)$ is either $\lambda z^{p^{\kappa_0}}$ or $\lambda z^{p^{\kappa_0+1}}$
     or $\lambda z^{p^{\kappa_0}}+\lambda' z^{p^{\kappa_0+1}}$ for some $\lambda,\lambda'\in k^{\times}$ (the last case occurring if and only if $e=p^{\kappa_0+1}-p^{\kappa_0}$).
    \end{remark}
    
    \subsection{Application to Modular parameterizations} We present in what follows an application of Theorem \ref{fllow} along the lines of \cite[Thm. 1.3]{BP09}. In this section, we prove Theorem~\ref{thm:mainapp} as well as the enhanced version of Strassman's theorem. We need some notation first. 
    
    We fix again $\pi\in \Pi$ and $n\geq 2$. 
    Consider an elliptic curve $E$ over $R_{\pi}$ and a surjective
     morphism of $R_{\pi}$-schemes
    $$\Theta:X_1(N)\rightarrow E$$
    where $X_1(N)$ is the complete modular curve over $R_{\pi}$. In particular, one can take $E$ to come from an elliptic curve over $\mathbb Q$ and $\Theta$ to be induced by a newform of weight $2$ as in the Eichler-Shimura theory. We denote by
    $$\Theta_{R_{\pi}} \colon X_1(N)(R_{\pi})\rightarrow E(R_{\pi})$$
    the induced map on $R_{\pi}$-points. Similarly for an open set $X=\textup{Spec}(A)\subset X_1(N)$ and a function $g\in \widehat{A}$ we denote by
    $$g_{R_{\pi}} \colon X(R_{\pi})\rightarrow R_{\pi}$$
    the induced map on $R_{\pi}$-points. For such a $g$ we consider the sets of zeroes of $g$:
    $$Z(g)=\{P\in X(R_{\pi})\ |\ g_{R_{\pi}}(P)=0\}\subset X(R_{\pi}).$$
    
    \begin{definition}
    A point $P\in X_1(N)(R_{\pi})$ is called a {\bf quasi-canonical lift} if the corresponding elliptic curve is ordinary with Serre-Tate parameter a root of unity.
    For an open set $X\subset X_1(N)$ we denote by 
    $\textup{QCL}(X(R_{\pi}))$ the set of all points in $X(R_{\pi})$ that are quasi-canonical lifts.  \end{definition}
    
    For every $\delta_{\pi}$-character $\psi\in \mathbf{X}^1_{\pi,\Phi}(E)$
     denote, as usual,  by
    $$\psi_{R_{\pi}}:E(R_{\pi})\rightarrow R_{\pi}$$
    the induced group homomorphism. For instance, if $n=2$ one can take $\psi=\psi_{1,2}$ from Section~\ref{casenr2}. For  $n\geq 3$ one can take $\psi$ to be any $R_{\pi}$-linear combination of images of $\psi_{1,2}$ via the different face maps.

    \begin{theorem}
    \label{modpar}
    There exists an affine open set $X=\textup{Spec}(A)\subset X_1(N)$ with non-empty reduction mod $\pi$ 
    such that for every $\delta_{\pi}$-character $\psi\in \mathbf{X}^1_{\pi,\Phi}(E)$
    there exists a monic polynomial $G\in \widehat{A}[t]$ with the property that for every $P\in E(R_{\pi})$ the following holds:
    $$\textup{QCL}(X(R_{\pi}))\cap \Theta_{R_{\pi}}^{-1}(\textup{Ker}(\psi_{R_{\pi}})+P)\subset 
    Z(G(\psi_{R_{\pi}}(P))).$$
    \end{theorem}
    
    {\it Proof}. Take $X$ as in Theorem \ref{fllow}. The for every $\psi\in \mathbf{X}^1_{\pi,\Phi}(E)$
    consider the composition 
    $$\Theta^{\sharp}:J^1_{\pi,\Phi}(X_1(N))\stackrel{J^1(\Theta)}{\longrightarrow} J^1_{\pi,\Phi}(E)\stackrel{\psi}{\rightarrow} \widehat{\mathbb G_{a,R_{\pi}}},$$
    which we identify with an element (still denoted by)
     $$\Theta^{\sharp}\in \mathcal O(J^1_{\pi,\Phi}(X_1(N))\subset \mathcal O(J^1_{\pi,\Phi}(X))=J^1_{\pi,\Phi}(A).$$
     By Theorem \ref{fllow} the image of $\Theta^{\sharp}$ in the ring
     $$J^1_{\pi,\Phi}(A)/(\check{f}_{\pi,1},\ldots,\check{f}_{\pi,n})
     $$
     is integral over $\widehat{A}$ hence there exists a monic polynomial
     $$G(t)=t^s+g_1t^{s-1}+\ldots+g_s\in \widehat{A}[t],\
      \ g_1,\ldots,g_s\in \widehat{A}$$
     and there exist $h_1,\ldots,h_n\in J^1_{\pi,\Phi}(A)$ such that
     $$(\Theta^{\sharp})^s+g_1\cdot (\Theta^{\sharp})^{s-1}+\ldots + g_s=h_1 \cdot \check{f}_{\pi,1}+\ldots+h_n\cdot \check{f}_{\pi,n}$$
     in the ring $J^1_{\pi,\Phi}(A)$.
     Let us   denote by $g_{i,R_{\pi}}$, $\check{f}_{\pi,i,R_{\pi}}$ and $h_{j,R_{\pi}}$ the functions
     $X(R_{\pi})\rightarrow R_{\pi}$ induced by $g_i$, $\check{f}_{\pi,i}$, and $h_j$ respectively.
    Then for all  $Q\in X(R_{\pi})$
    we get an equality
    $$(\psi_{R_{\pi}}(\Theta_{R_{\pi}}(Q)))^s+g_{1,R_{\pi}}(Q)\cdot (\psi_{R_{\pi}}(\Theta_{R_{\pi}}(Q)))^{s-1}+\ldots + g_{s,R_{\pi}}(Q)=$$
    $$=h_{1,R_{\pi}}(Q) \cdot \check{f}_{\pi,1,R_{\pi}}(Q)+\ldots +h_{n,R_{\pi}}(Q)\cdot \check{f}_{\pi,n,R_{\pi}}(Q).$$
     By Proposition \ref{chara} we have
     $$\check{f}_{\pi,i,R_{\pi}}(Q)=0\ \ \text{for}\ \ i\in \{1,\ldots,n\},\ \ Q\in \textup{QCL}(X(R_{\pi})).$$
     Hence for all $P\in E(R_{\pi})$ and all $Q\in \textup{QCL}(X(R_{\pi}))\cap \Theta_{R_{\pi}}^{-1}(\textup{Ker}(\psi_{R_{\pi}})+P)$ we have 
     $\Theta_{R_{\pi}}(Q)=\psi_{R_{\pi}}(P)$
     hence
     $$
     (\psi_{R_{\pi}}(P))^s+g_{1,R_{\pi}}(Q)\cdot (\psi_{R_{\pi}}(P))^{s-1}+\ldots + g_{s,R_{\pi}}(Q)=0,
     $$
     which implies $Q\in Z(G(\psi_{R_{\pi}}(P)))$.
    \qed
    
    \ 

We conclude with a finiteness result; cf. Corollary \ref{finnn} below.  We need the following variant of Strassman's theorem \cite[p. 306]{Rob00}. The classic case is that of the affine line over a not necessarily discrete valuation ring. We need here the case of an arbitrary smooth curve over a complete discrete valuation ring; the fact that the valuation ring is discrete greatly simplifies the proof.

\begin{lemma}[Strassman's theorem for curves over a DVR]
\label{strass} 
Let $V$ be a complete 
DVR with maximal ideal generated by $\pi\in V$. Fix $X/V$ a smooth affine curve 
with connected closed fiber and let $\widehat{\mathcal{O}(X)}$ be the $\pi$-adic completion of $\mathcal O(X)$. Then every non-zero $g \in \widehat{\mathcal{O}(X)}$ has finitely many zeros in $\widehat{X}(V)=X(V)$.
\end{lemma}

{\it Proof}.
 Set $A = \mathcal{O}(X)$ and for $g\in \widehat{A}$ denote by $Z(g)$ the set of zeros of $g$ in $\widehat{X}(V)=X(V)$. Without loss of generality, assume $g$ is not in $\pi \widehat{A}$ as $\pi$ is a regular element. Note the usual bijection $Z(g) \cong \textrm{Hom}_{V}(\widehat{A}/(g),V)$. For each map $\varphi$ in this set 
 denote by $P_{\varphi}$ the kernel of $\varphi$ and by $M_{\varphi}$ the kernel of the composition $$\overline{\varphi} \colon \widehat{A}/(g) \stackrel{\varphi}{\to} V \to k := V/ \pi V.$$  
 
 \

 \noindent {\it Claim}. The map $\textrm{Hom}_V(\widehat{A}/(g),V)\rightarrow \textup{Spec}(\widehat{A}/(g))$, $\varphi\mapsto P_{\varphi}$  is injective.
 
 \

 Indeed if $P_{\varphi_1}=P_{\varphi_2}$ and $\iota:R_{\pi}\rightarrow \widehat{A}/(g)$ is the natural map  then for all $x\in \widehat{A}$ we have 
 $$\varphi_1(x-\iota(\varphi_1(x))=\varphi_1(x)-\varphi_1(x)=0$$
  hence $x-\iota(\varphi_1(x))\in P_{\varphi_1}=P_{\varphi_2}$ so 
 $$0=\varphi_2(x-\iota(\varphi_1(x))=\varphi_2(x)-\varphi_1(x)$$
 and our claim is proved.
 
 Similarly we have that 
 the map 
 $\textrm{Hom}_V(\widehat{A}/(g), k)\rightarrow \textup{Spec}(\widehat{A}/(g,\pi))$
 is injective. 
 Since $A/\pi A$ is regular and connected, it is an integral domain of dimension $1$. It follows that $\widehat{A}/(g,\pi)$ is an Artin ring and therefore it has a finite spectrum. 
 In particular, 
  the set $\textrm{Hom}_V(\widehat{A}/(g), k)$ is finite. 
  
 Consider the natural map 
 \begin{equation}
 \label{themapp}
 \rho \colon \textrm{Hom}_{V}(\widehat{A}/(g), V) \to \textrm{Hom}_V(\widehat{A}/(g), k).
 \end{equation}
  As the target of this map was shown to be  finite  the lemma follows by showing that the fibers of \eqref{themapp} are also finite. Fix a $V$-homomorphism $\varphi_0\colon \widehat{A}/(g)\rightarrow V$ 
 and let $\varphi\colon \widehat{A}/(g)\rightarrow V$ be any $V$-homomorphism such that $\varphi$ and $\varphi_0$ induce the same map $\overline{\varphi}=\overline{\varphi_0}\colon \widehat{A}/(g)\rightarrow k$ hence $M_{\varphi}=M_{\varphi_0}$. In particular, $P_{\varphi}$ is contained in $M_{\varphi_0}$. But there are only finitely many prime ideals of $\widehat{A}/(g)$ contained in $M_{\varphi_0}$ because the ring $(\widehat{A}/(g))_{M_{\varphi_0}}$ is local and Noetherian of dimension $1$. We conclude by the claim above that there are only finitely many $V$-homomorphisms $\varphi:\widehat{A}/(g)\rightarrow V$ for which $\overline{\varphi}=\overline{\varphi_0}$ which proves the finiteness of the fibers of the map \eqref{themapp}.\qed

\begin{corollary}\label{finnn}  
There exists an open set $X\subset X_1(N)$ with non-empty reduction mod $\pi$ such that for every $\delta_{\pi}$-character $\psi\in \mathbf{X}^1_{\pi,\Phi}(E)$ there exists a finite set $\Sigma\subset R_{\pi}$ with the following property. For all $P\in E(R_{\pi})$ if $\psi_{\pi}(P)\not\in \Sigma$ then the set
$$\textup{QCL}(X(R_{\pi}))\cap \Theta_{R_{\pi}}^{-1}(\textup{Ker}(\psi_{R_{\pi}})+P)$$
is finite.\end{corollary}

{\it Proof}.
Let $G$ be the polynomial in Theorem \ref{modpar}. Then,
in view of Lemma \ref{strass} applied to $V=R_{\pi}$, it is enough to take $\Sigma$ the set of all roots of  $G$. 
\qed

\begin{remark}
The above result is a ramified version of \cite[Thm. 1.3]{BP09} which dealt with the case $\pi=p$ and with the set of canonical lifts in $X(R)$. 
The group $\textup{Ker}(\psi_{R_{\pi}})$ in Corollary \ref{finnn} contains the torsion of $E(R_{\pi})$ but is generally bigger than the torsion. For $n=2$ and $\psi=\psi_{1,2}$ this group was explicitly described in our arithmetic ``Theorem of the Kernel," cf. Theorem \ref{frodooo}. As remarked before that theorem, for every $E$ that  is ordinary but not a quasi-canonical lift  the group 
$\textup{Ker}(\psi_{R_{\pi}})$ is not torsion.
\end{remark}

\subsection{$\delta$-Serre operators} We conclude now by applying, as in the ODE case,  $\delta$-Serre operators and  $\delta$-Euler operators to  produce systems of differential equations satisfied by 
our partial $\delta$-modular forms.
In this subsection we assume $n$ is arbitrary, $\pi=p$, 
and we consider an affine open subset $X\subset Y_1(N)$ with non-empty reduction mod $p$,
where $N\geq 4$ is coprime to $p$. We let $B:=B_1(N)=\textup{Spec}(\oplus_{m\in \mathbb Z} L^m)\rightarrow X$ be as in Definition \ref{defofB}. 
  Recall from \cite[Sec. 8.3.2 and Rem. 3.58]{Bu05}  the classic {\bf Serre operator} 
  $$\partial \colon \bigoplus_{m \in \mathbb{Z}} L^m \to \bigoplus_{m \in \mathbb{Z}} L^m$$   
  (which has the property $\partial(L^{\otimes m})\subset L^{\otimes m+2}$)
 and the {\bf Euler operator}
 $$\mathcal D\colon \bigoplus_{m \in \mathbb{Z}} L^m \to \bigoplus_{m \in \mathbb{Z}} L^m,\ \ \mathcal D(\sum \alpha_m):=m\alpha_m,\ \ \alpha_m\in L^{\otimes m}$$ 
 (which has the property that $\mathcal D(L^{\otimes m})\subset L^{\otimes m}$).
 The above operators induce $R$-derivations of the algebra $\mathcal O(B_1(N))$. Hence,
 if $\mu\in \mathbb M_n^r$ we may consider the induced derivations
 $$\partial_{\mu},\mathcal D_{\mu}\colon \mathcal O(J^r_{p,\Phi}(B_1(N)))\rightarrow 
 \mathcal O(J^r_{p,\Phi}(B_1(N)));$$
 cf. Proposition \ref{thm:derivations}. Note that under the identification (\ref{wqe}) we have
 $\mathcal D=x\frac{d}{x}$ and hence, for $w=\sum a_{\nu}\phi_{\nu}$ and $f\in \mathcal O(J^r_{p,\Phi}(X))$ we have
 $$\mathcal D_{\mu} (f\cdot x^w)=p^r\cdot a_{\mu}\cdot f\cdot x^w.
 $$
 
\begin{proposition}\label{prop:derlinebundle}
Assume we are given a derivation $D \colon \bigoplus_{m \in \mathbb{Z}} L^m \to \bigoplus_{m \in \mathbb{Z}} L^m$ so that there is $c \geq 0$ with $\partial( L^m ) \subset L^{m+c}$ for all $m \geq 0$. Then, for all $\mu\in \mathbb M_n^r$ and all $w\in \mathbb Z_{\Phi}^r$,  the induced derivation (cf. Proposition \ref{thm:derivations})
$$D_{\mu} \colon \mathcal O(J^r_{p,\Phi}(B_1(N)))\rightarrow
 \mathcal O(J^r_{p,\Phi}(B_1(N)))$$ satisfies
$$D_{\mu}(M^r_{p,\Phi,X}(w))\subset M^r_{p,\Phi,X}(w+c\phi_{\mu}).$$
\end{proposition}

{\it Proof}.
Similar to \cite[Prop. 3.56]{Bu05}.
\qed

\bigskip

Recall from \cite[Sec. 8.3.3]{Bu05} that if  $X$ has reduction mod $p$ contained in the ordinary locus then we may consider the   {\bf Ramanujan form}
\begin{equation}
\label{rama}
P\in L^2.
\end{equation}
Then, as in \cite[Eqn. 8.94]{Bu05}, we consider the derivation 
$$\partial^*:=\partial + P \mathcal D
\colon \bigoplus_{m \in \mathbb{Z}} L^m \to \bigoplus_{m \in \mathbb{Z}} L^m$$   
  which has the property that $\partial^*(L^m)\subset L^{m+2}$ for all $m\in \mathbb Z$.
  
  On the other hand, as in \cite[Sec. 8.3.5]{Bu05} we may consider the {\bf canonical derivation}
  $$\partial^{\textup{can}}:=(1+T)\frac{d}{dT}:R[[T]]\rightarrow R[[T]].$$
  As in Proposition \ref{thm:derivations} one shows that for every $\mu\in \mathbb M^r_{\mu}$ there is a unique derivation
  $$\partial^{\textup{can}}_{\mu}:S^r_{\textup{for}}\rightarrow S^r_{\textup{for}}$$
  satisfying

1) $\partial^{\textup{can}}_{\mu}\phi_{\mu} F=p^r\cdot \phi_{\mu} \partial^{\textup{can}}F$ for all $F\in R[[T]]$;

2) $\partial^{\textup{can}}_{\mu} \phi_{\nu} F=0$ for all $F\in R[[T]]$ 
and all $\nu\in \mathbb M^r_n\setminus \{\mu\}.$

It is given by
$$F\mapsto \partial^{\textup{can}}_{\mu}F:=(1+T^{\phi_{\mu}})\frac{\partial F}{\partial \delta_{\mu} T}.
$$
  
  Finally recall the injective homomorphism $\mathcal E:M^r_{p,\Phi,\textup{ord}}(w) \rightarrow S^r_{\textup{for}}$ from Remark \ref{6parts}, Part 4. With the notation above
  we have the following:
  
  \begin{proposition}\label{rrw}
 For every $\mu\in \mathbb M_n^r$ and every $w\in \mathbb Z_{\Phi}^r$
  we have an equality 
  $$\mathcal E\circ \partial^*_{\mu}=\partial^{\textup{can}}_{\mu}\circ \mathcal E$$
  of maps
  $M^r_{\pi,\Phi,X}(w)\rightarrow S^r_{\textup{for}}$.
  Equivalently, for  $w=\sum a_{\nu}\phi_{\nu}$ 
  and $f\in M^r_{p,\Phi,X}(w)$ we have:
  $$\mathcal E(\partial^*_{\mu}f)= \mathcal E(\partial_{\mu}f+a_{\mu}p^r P^{\phi_{\mu}}f)=(1+T^{\phi_{\mu}})\frac{\partial(\mathcal E(f))}{\partial\delta_{\mu}T}.
$$
  \end{proposition}
  
  {\it Proof}.
  Similar to \cite[Prop. 8.42]{Bu05}. The value of $\epsilon$ in loc.cit. is $1$ in view of the comment after \cite[Equation 8.52]{Bu05}.
  \qed

 \begin{remark}\label{vvd}
  Note that for  $\mu=i\in \{1,\ldots,n\}$ we easily compute 
  \begin{equation}\label{ttroo}
  \partial_i^{\textup{can}}\Psi_i=(1+T^{\phi_i})\frac{\partial \Psi_i}{\partial \delta_i T}
  =(1+T^{\phi_i})\frac{\partial}{\partial \delta_i T} \{\frac{1}{p}(\phi_i-p)\log(1+T)\}
  =1.
  \end{equation}\end{remark}
  
  The following 
 corollary follows from the ODE case in \cite[Prop. 8.64]{Bu05}.
 
  \begin{corollary} The following equality holds:
  \begin{equation}
  \partial_i^* f_i^{\textup{jet}}=
  \partial_i f_i^{\textup{jet}}-pP^{\phi_i}f_i^{\textup{jet}}=cf_i^{\partial}.
    \end{equation}
  \end{corollary}
  
  {\it Proof}. For convenience we repeat the argument.
  The form $\partial_i^* f_i^{\textup{jet}}$ 
  has weight $\phi_i-1$ (cf. Proposition \ref{prop:derlinebundle}) and Serre-Tate expansion
  $c$ (cf. Proposition \ref{rrw} and Remarks \ref{bgt} and \ref{vvd}) while 
  the form $cf_i^{\partial}$ has the same weight and Serre-Tate expansion (cf. Theorem \ref{shacall}); hence the two forms must coincide by the Serre-Tate expansion principle.
   \qed

   \bigskip
   
   One can get new results along these lines in the PDE case; here is an example.
   
    \begin{corollary} For $n=2$ the following equalities hold:
  \begin{equation}
  \partial_1^* f_{1,2}^{\textup{jet}}=
  \partial_1 f_{1,2}^{\textup{jet}}-pP^{\phi_1}f_{1,2}^{\textup{jet}}=cp f_1^{\partial}f_{2,\partial},
    \end{equation}
    \begin{equation}
  \partial_2^* f_{1,2}^{\textup{jet}}=
  \partial_2 f_{1,2}^{\textup{jet}}-pP^{\phi_2}f_{1,2}^{\textup{jet}}=-cp f_2^{\partial}f_{1,\partial},
    \end{equation}
     \begin{equation}
  (\partial_1 f_{1,2}^{\textup{jet}}-pP^{\phi_1}f_{1,2}^{\textup{jet}})
  (\partial_2 f_{1,2}^{\textup{jet}}-pP^{\phi_2}f_{1,2}^{\textup{jet}})+c^2p^2=0.
    \end{equation}
  \end{corollary}
  
  {\it Proof}. 
  The form $\partial_1^* f_{1,2}^{\textup{jet}}$ 
  has weight $\phi_1-\phi_2$ (cf. Proposition \ref{prop:derlinebundle}) and Serre-Tate expansion
  $cp$ (cf. Proposition \ref{rrw} and Remarks \ref{bgt} and \ref{vvd}) while 
  the form $cpf_1^{\partial}f_{2,\partial}$ has the same weight and Serre-Tate expansion (cf. Theorem \ref{shacall}); hence the two forms must coincide by the Serre-Tate expansion principle which proves the first equality. The second equality is proved similarly. The third equality follows by multiplying the first two equalities.
   \qed

\end{document}